\documentclass{amsproc}
\usepackage{fullpage}
\usepackage{amssymb}
\usepackage{amsmath}
\usepackage{amsthm}
\usepackage{amsfonts}
\usepackage{dsfont}
\usepackage{mathrsfs}
\usepackage{hyperref}
\usepackage{graphics}
\usepackage{graphicx}
\usepackage{url}
\usepackage{enumitem}
\usepackage{mathrsfs}
\DeclareSymbolFontAlphabet{\mathrsfs}{rsfs}
\usepackage[mathscr]{euscript}
\usepackage{epsfig}

\newcommand{\cone}{\mathbb{C}}
\newcommand{\cpone}{\mathbb{P}^1}

\newcommand{\rone}{\mathbb{R}}

\newcommand{\upsnu}{\Upsilon^{(\nu)}}

\newcommand{\bein}{\mathcal{B}_{\text{in}}}
\newcommand{\bout}{\mathcal{B}_{\text{out}}}

\newcommand{\ccwarrow}{\text{\Large$\curvearrowleft$}}

\newtheorem{theorem}{Theorem}[section]
\newtheorem{lemma}[theorem]{Lemma}

\theoremstyle{definition}

\theoremstyle{remark}
\newtheorem{remark}[theorem]{Remark}

\numberwithin{equation}{section}

\newtheoremstyle{citing}{}{}{\itshape}{}{\bfseries}{.}{ }{\thmnote{#3}}
\theoremstyle{citing}

\begin{document}
\title{A construction of hyperk\"{a}hler metrics through Riemann-Hilbert problems I}


\author{C. Garza}
\address{Department of Mathematics, IUPUI, Indianapolis, USA}
\curraddr{}
\email{cegarza@iu.edu}
\thanks{}


\subjclass[2010]{Primary }

\date{}

\begin{abstract}
In 2009 Gaiotto, Moore and Neitzke presented a new construction of hyperk\"{a}hler metrics on the total spaces of certain complex integrable systems, represented as a torus fibration $\mathcal{M}$ over a base space $\mathcal{B}$, except for a divisor $D$ in $\mathcal{B}$, in which the torus fiber degenerates into a nodal torus. The hyperk\"{a}hler metric $g$ is obtained via solutions $\mathcal{X}_\gamma$ of a Riemann-Hilbert problem. We interpret the Kontsevich-Soibelman Wall Crossing Formula as an isomonodromic deformation of a family of RH problems, therefore guaranteeing continuity of $\mathcal{X}_\gamma$ at the walls of marginal stability. The technical details about solving the different classes of Riemann-Hilbert problems that arise here are left to a second article. To extend this construction to singular fibers, we use the Ooguri-Vafa case as our model and choose a suitable gauge transformation that allow us to define an integral equation defined at the degenerate fiber, whose solutions are the desired Darboux coordinates $\mathcal{X}_\gamma$. We show that these functions yield a holomorphic symplectic form $\varpi(\zeta)$, which, by Hitchin's twistor construction, constructs the desired hyperk\"{a}hler metric.
\end{abstract}

\maketitle

\section{Introduction}\label{intr}

Hyperk\"{a}hler manifolds first appeared within the framework of differential geometry as Riemannian manifolds with holonomy group of special restricted group. Nowadays, hyperk\"{a}hler geometry forms a separate research subject fusing traditional areas of mathematics such as differential and algebraic geometry of complex manifolds, holomorphic symplectic geometry, Hodge theory and many others. 

One of the latest links can be found in theoretical physics: In 2009, Gaiotto, Moore and Neitzke \cite{gaiotto} proposed  a new construction of hyperk\"{a}hler metrics $g$ on target spaces $\mathcal{M}$ of quantum field theories with $d = 4, \mathcal{N} = 2$ superysmmetry. Such manifolds were already known to be hyperk\"{a}hler (see \cite{seiberg}), but no known explicit hyperk\"{a}hler metrics have been constructed.

The manifold $\mathcal{M}$ is a total space of a complex integrable system and it can be expressed as follows. There exists a complex manifold $\mathcal{B}$, a divisor $D \subset \mathcal{B}$ and a subset $\mathcal{M}' \subset \mathcal{M}$ such that $\mathcal{M}'$ is a torus fibration over $\mathcal{B}' := \mathcal{B} \backslash D$. On the divisor $D$, the torus fibers of $\mathcal{M}$ degenerate, as Figure \ref{nodtorus} shows.

\begin{figure}[htbp]
	\centering
		\includegraphics[width=0.60\textwidth]{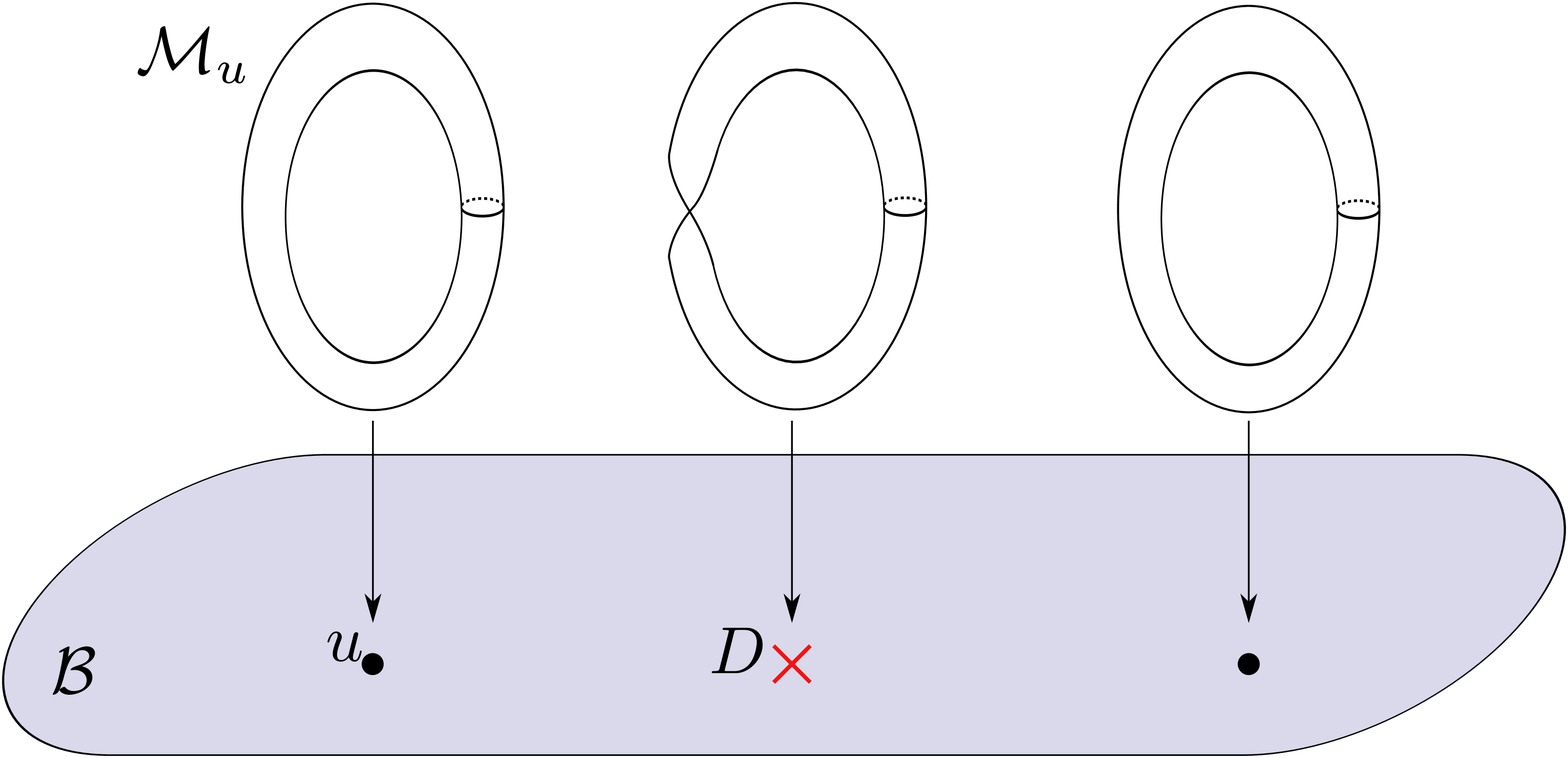}
	\caption{Hyperk\"{a}hler manifolds realized as torus fibrations}
	\label{nodtorus}
\end{figure}

Moduli spaces $\mathcal{M}$ of Higgs bundles on Riemann surfaces with prescribed singularities at finitely many points are one of the prime examples of this construction. Hyperk\"{a}hler geometry is useful since we can use Hitchin's twistor space construction \cite{hitchin} and consider all $\cpone$-worth of complex structures at once. In the case of moduli spaces of Higgs bundles, this allows us to consider $\mathcal{M}$ from three distinct viewpoints:

\begin{enumerate}
	\item (Dolbeault) $\mathcal{M}_{\text{Dol}}$ is the moduli space of Higgs bundles, i.e. pairs $(E, \Phi)$, $E \to C$ a rank $n$ degree zero holomorphic vector bundle and $\Phi \in \Gamma(\text{End}(E) \otimes \Omega^1)$ a Higgs field.
	
	\item (De Rham) $\mathcal{M}_{\text{DR}}$  is the moduli space of flat connections on rank $n$ holomorphic vector bundles, consisting of pairs $(E, \nabla)$ with $\nabla : E \to \Omega^1 \otimes E$ a holomorphic connection and
	
	\item (Betti) $\mathcal{M}_{\text{B}} = \text{Hom}(\pi_1(C) \to \text{GL}_n(\cone))/\text{GL}_n(\cone)$ of conjugacy classes of representations of the fundamental group of $C$. 
\end{enumerate}

\noindent All these algebraic structures form part of the family of complex structures making $\mathcal{M}$ into a hyperk\"{a}hler manifold. 

To prove that the manifolds $\mathcal{M}$ from the integrable systems are indeed hyperk\"{a}hler, we start with the existence of a simple, explicit hyperk\"{a}hler metric $g^{\text{sf}}$ on $\mathcal{M}'$. Unfortunately, $g^{\text{sf}}$ does not extend to $\mathcal{M}$.  To construct a complete metric $g$, it is necessary to do ``quantum corrections'' to $g^{\text{sf}}$. These are obtained by solving a certain explicit integral equation (see \eqref{inteq} below). The novelty is that the solutions, acting as Darboux coordinates for the hyperk\"{a}hler metric $g$, have discontinuities at a specific locus in $\mathcal{B}$. Such discontinuities cancel the global monodromy around $D$ and is thus feasible to expect that $g$ extends to the entire $\mathcal{M}$.

We start by defining a Riemann-Hilbert problem on the $\cpone$-slice of the twistor space $\mathcal{Z} = \mathcal{M}' \times \cpone$. That is, we look for functions $\mathcal{X}_\gamma$ with prescribed discontinuities and asymptotics. In the language of Riemann-Hilbert theory, this is known as \textit{monodromy data}. Rather than a single Riemann-Hilbert problem, we have a whole family of them parametrized by the $\mathcal{M}'$ manifold. We show that this family constitutes an \textit{isomonodromic deformation} since by the Kontsevich-Soibelman Wall-Crossing Formula, the monodromy data remains invariant. 

Although solving Riemann-Hilbert problems in general is not always possible, in this case it can be reduced to an integral equation solved by standard Banach contraction principles. We will focus on a particular case known as the ``Pentagon'' (a case of Hitchin systems with gauge group $\text{SU}(2)$). The family of Riemann-Hilbert problems and their methods of solutions is a topic of independent study so we leave this construction to a second article that can be of interest in the study of boundary-value problems.

The extension of the manifold $\mathcal{M}'$ is obtained by gluing a circle bundle with an appropriate gauge transformation eliminating any monodromy problems near the divisor $D$. The circle bundle constructs the degenerate tori at the discriminant locus $D$  (see Figure \ref{pinch}).

\begin{figure}[htbp]
	\centering
		\includegraphics[width=0.15\textwidth]{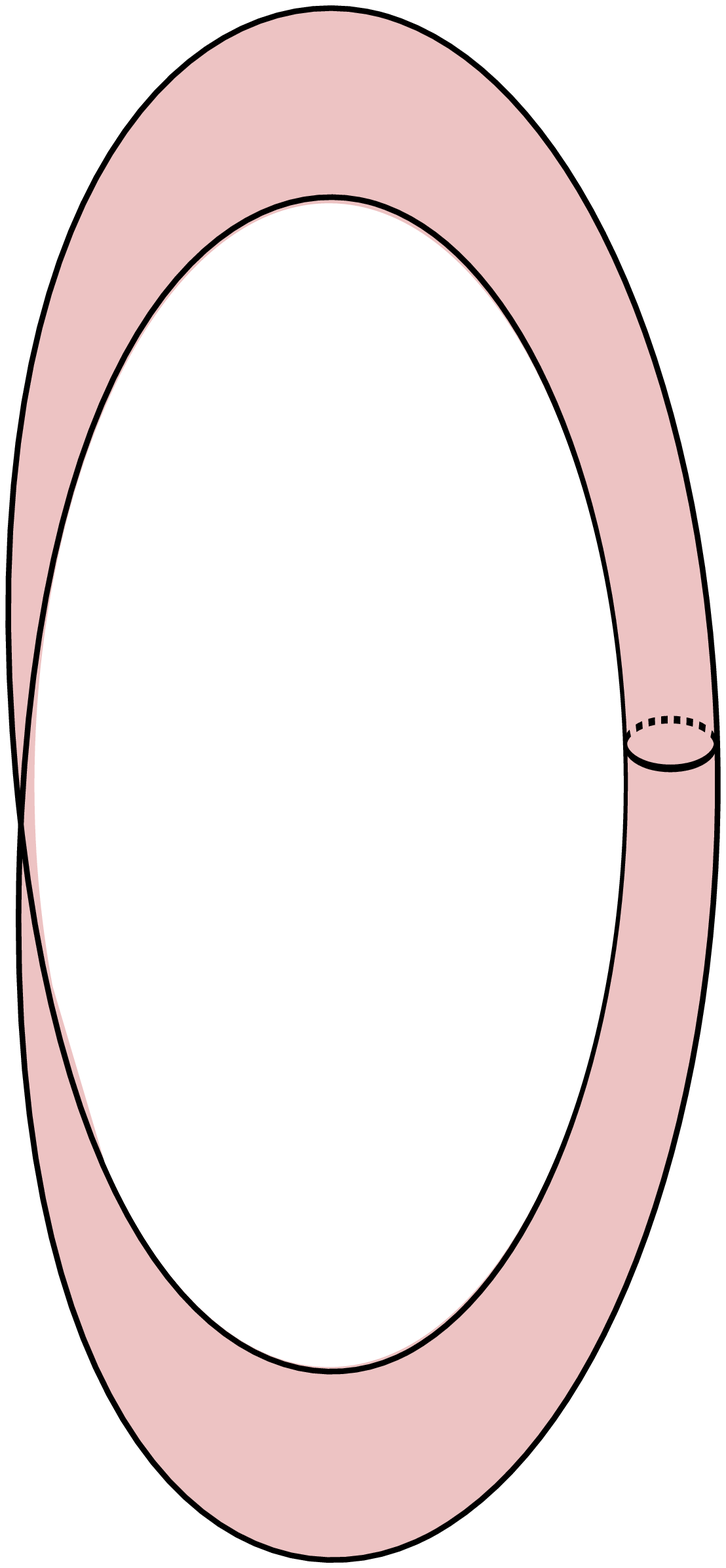}
	\caption{Construction of degenerate fibers}
	\label{pinch}
\end{figure}

On the extended manifold $\mathcal{M}$ we prove that the solutions $\mathcal{X}_\gamma$ of the Riemann-Hilbert problem on $\mathcal{M}'$ extend and the resulting holomorphic symplectic form $\varpi(\zeta)$ gives the desired hyperk\"{a}hler metric $g$. 

Although for the most basic examples of this construction such as the moduli space of Higgs bundles it was already known that $\mathcal{M}'$ extends to a hyperk\"{a}hler manifold $\mathcal{M}$ with degenerate torus fibers, the construction here works for the general case of $\dim_\cone \mathcal{B} = 1$. Moreover, the functions $\mathcal{X}_\gamma$ here are special coordinates arising in moduli spaces of flat connections, Teichm\"{u}ller theory and Mirror Symmetry. In particular, these functions are used in \cite{chan} for the construction of holomorphic discs with boundary on special Lagrangian torus fibers of mirror manifolds.

The organization of the paper is as follows. In Section \ref{intsys} we introduce the complex integrable systems to be considered in this paper. These systems arose first in the study of moduli spaces of Higgs bundles and they can be written in terms of initial data and studied abstractly. This leads to a formulation of a family of Riemann-Hilbert problems, whose solutions provide Darboux coordinates for the moduli spaces $\mathcal{M}$ considered and hence equip the latter with a hyperk\"{a}hler structure. In Section \ref{ov} we fully work the simplest example of these integrable systems: the Ooguri-Vafa case. Although the existence of this hyperk\"{a}hler metric was already known, this is the first time it is obtained via Riemann-Hilbert methods. In Section \ref{gmetric}, we explicitly show that this metric is a smooth deformation of the well-known Taub-NUT metric near the singular fiber of $\mathcal{M}$ thus proving its extension to the entire manifold. In Section \ref{pent} we introduce our main object of study, the Pentagon case. This is the first nontrivial example of the integrable systems considered and here the Wall Crossing phenomenon is present. We use the KS wall-crossing formula to apply an isomonodromic deformation of the Riemann-Hilbert problems leading to solutions continuous at the wall of marginal stability. Finally, Section \ref{sfiber} deals with the extension of these solutions $\mathcal{X}_\gamma$ to singular fibers of $\mathcal{M}$ thought as a torus fibration. What we do is to actually complete the manifold $\mathcal{M}$ from a regular torus fibration $\mathcal{M}'$ by gluing circle bundles near a discriminant locus $D$. This involves a change of the torus coordinates for the fibers of $\mathcal{M}'$. In terms of the new coordinates, the $\mathcal{X}_\gamma$ functions extend to the new patch and parametrize the complete manifold $\mathcal{M}$. We finish the paper by showing that, near the singular fibers of $\mathcal{M}$, the hyperk\"{a}hler metric $g$ looks like the metric for the Ooguri-Vafa case plus some smooth corrections, thus proving that this metric is complete.

\textbf{Acknowledgment:} The author likes to thank Andrew Neitzke for his guidance, support and incredibly helpful conversations. 

\section{Integrable Systems Data}\label{intsys}

We start by presenting the complex integrable systems introduced in \cite{gaiotto}. As motivation, consider the moduli space $\mathcal{M}$ of Higgs bundles on a complex curve $C$ with Higgs field $\Phi$ having prescribed singularities at finitely many points. In \cite{wkb}, it is shown that the space of quadratic differentials $u$ on $C$ with fixed poles and residues is a complex affine space $\mathcal{B}$ and the map $\text{det} : \mathcal{M} \to \mathcal{B}$ is proper with generic fiber $\text{Jac}(\Sigma_u)$, a compact torus obtained from the \textit{spectral curve} $\Sigma_u : = \{(z, \phi) \in T^*C : \phi^2 = u\}$,  a double-branched cover of $C$ over the zeroes of the quadratic differential $u$. $\Sigma_u$ has an involution that flips $\phi \mapsto -\phi$. If we take $\Gamma_u$ to be the subgroup of $H_1(\Sigma_u, \mathbb{Z})$ odd under this involution, $\Gamma$ forms a lattice of rank 2 over $\mathcal{B}'$, the space of quadratic differentials with only simple zeroes. This lattice comes with a non-degenerate anti-symmetric pairing $\left\langle , \right\rangle$ from the intersection pairing in $H_1$. It is also proved in \cite{wkb} that the fiber $\text{Jac}(\Sigma_u)$ can be identified with the set of characters $\text{Hom}(\Gamma_u, \rone/2\pi \mathbb{Z})$. If $\lambda$ denotes the tautological 1-form in $T^* C$, then for any $\gamma \in \Gamma$,
\[ Z_\gamma = \frac{1}{\pi}  \oint_\gamma \lambda \]
defines a holomorphic function $Z_\gamma$ in $\mathcal{B}'$. Let $\{\gamma_1, \gamma_2\}$ be a local basis of $\Gamma$ with $\{\gamma^1, \gamma^2\}$ the dual basis of $\Gamma^*$. Without loss of generality, we also denote by $\left\langle ,  \right\rangle$ the pairing in $\Gamma^*$. Let $\left\langle dZ \wedge dZ \right\rangle$ be short notation for $\left\langle \gamma^1, \gamma^2 \right\rangle dZ_{\gamma_1} \wedge dZ_{\gamma_2}$. Since $\dim_\cone \mathcal{B}' = 1$, $\left\langle dZ \wedge dZ \right\rangle = 0$.

This type of data arises in the construction of hyperk\"{a}hler manifolds as in \cite{gaiotto} and \cite{notes}, so we summarize the conditions required:

We start with a complex manifold $\mathcal{B}$ (later shown to be affine) of dimension $n$ and a divisor $D \subset \mathcal{B}$. Let $\mathcal{B}' = \mathcal{B} \backslash D$. Over $\mathcal{B}'$ there is a local system $\Gamma$ with fiber a rank $2n$ lattice, equipped with a non-degenerate anti-symmetric integer valued pairing $\left\langle \, , \right\rangle$. 

We will denote by $\Gamma^*$ the dual of $\Gamma$ and, by abuse of notation, we'll also use $\left\langle \, , \right\rangle$ for the dual pairing (not necessarily integer-valued) in $\Gamma^*$. Let $u$ denote a general point of $\mathcal{B}'$. We want to obtain a torus fibration over $\mathcal{B}'$, so let $\text{TChar}_u(\Gamma)$ be the set of twisted unitary characters of $\Gamma_u$\footnote{Although we can also work with the set of unitary characters (no twisting involved) by shifting the $\theta$-coordinates, we choose not to do so, as that results in more complex calculations}, i.e. maps $\theta : \Gamma_u \to \mathbb{R}/2\pi \mathbb{Z}$ satisfying
\begin{equation*}
\theta_\gamma + \theta_{\gamma'} = \theta_{\gamma + \gamma'} + \pi \left\langle \gamma, \gamma' \right\rangle.
\end{equation*}
Topologically, $\text{TChar}_u(\Gamma)$ is a torus $(S^1)^{2n}$. Letting $u$ vary, the $\text{TChar}_u(\Gamma)$ form a torus bundle $\mathcal{M}'$ over $\mathcal{B}'$. Any local section $\gamma$ gives a local angular coordinate of $\mathcal{M}'$ by ``evaluation on $\gamma$'', $\theta_\gamma : \mathcal{M}' \to \mathbb{R}/2\pi \mathbb{Z}$.

We also assume there exists a homomorphism $Z : \Gamma \to \cone$ such that the vector $Z(u) \in \Gamma^*_u \otimes \cone$ varies holomorphically with $u$. If we pick a patch $U \subset \mathcal{B}'$ on which $\Gamma$ admits a basis $\{\gamma_1, \ldots, \gamma_{2n}\}$ of local sections in which $\left\langle , \right\rangle$ is the standard symplectic pairing, then (after possibly shrinking $U$) the functions
\[ f_i = \text{Re}(Z_{\gamma_i}) \]
are real local coordinates. The transition functions on overlaps $U \cap U'$ are valued on $\text{Sp}(2n, \mathbb{Z})$, as different choices of basis in $\Gamma$ must fix the symplectic pairing. This gives an affine structure on $\mathcal{B}'$.

 By differentiating and evaluating in $\gamma$, we get 1-forms $d\theta_\gamma, d Z_\gamma$ on $\mathcal{M}'$ which are linear on $\Gamma$. For a local basis $\{\gamma_1, \ldots, \gamma_{2n}\}$ as in the previous paragraph, let $\{\gamma^1, \ldots, \gamma^{2n}\}$ denote its dual basis on $\Gamma^*$. We write $\left\langle dZ \wedge dZ \right\rangle$ as short notation for
 \begin{equation}\label{dzdz}
  \left\langle \gamma^i , \gamma^j \right\rangle dZ_{\gamma_i} \wedge dZ_{\gamma_j},
  \end{equation}
 where we sum over repeated indices. Observe that the anti-symmetric pairing $\left\langle \, , \right\rangle$ and the anti-symmetric wedge product of 1-forms makes \eqref{dzdz} symmetric.  We require that:
 
\begin{equation}\label{dz0}
\left\langle dZ \wedge dZ \right\rangle = 0,
\end{equation}

By \eqref{dz0}, near $u$, $\mathcal{B}'$ can be locally identified with a complex Lagrangian submanifold of $\Gamma^* \otimes_\mathbb{Z} \cone$.

In the example of moduli spaces of Higgs bundles, as $u$ approaches a quadratic differential with non-simple zeros, one homology cycles vanishes (see Figure \ref{nodtorus}). This cycle $\gamma_0$ is primitive in $H_1$ and its monodromy around the critical quadratic differential is governed by the Picard-Lefschetz formula. In the general case, let $D_0$ be a component of the divisor $D \subset \mathcal{B}$. We also assume the following:

\begin{itemize}
	\item $Z_{\gamma_0}(u) \to 0$ as $u \to u_0 \in D_0$ for some $\gamma_0 \in \Gamma$.
	
	
	\item $\gamma_0$ is primitive (i.e. there exists some $\gamma'$ with $\left\langle \gamma_0, \gamma'\right\rangle = 1$).
	
	\item The monodromy of $\Gamma$ around $D_0$ is of ``Picard-Lefschetz type'', i.e.
	\begin{equation}\label{piclf}
	\gamma \mapsto \gamma + \left\langle \gamma, \gamma_0 \right\rangle \gamma_0
	\end{equation}
\end{itemize}

We assign a complex structure and a holomorphic symplectic form on $\mathcal{M}'$ as follows (see \cite{notes} and the references therein for proofs). Take a local basis $\{\gamma_1, \ldots, \gamma_{2n}\}$ of $\Gamma$. If $\epsilon^{ij} = \left\langle \gamma_i , \gamma_j\right\rangle$ and $\epsilon_{ij}$ is its dual, let
\begin{equation}\label{ome}
\omega_+ = \left\langle dZ \wedge d\theta \right\rangle = \epsilon_{ij} \, dZ_{\gamma_i} \wedge d\theta_{\gamma_j}.
\end{equation}
By linearity on $\gamma$ of the 1-forms, $\omega_+$ is independent of the choice of basis. There is a unique complex structure $J$ on $\mathcal{M}'$ for which $\omega_+$ is of type (2,0). The 2-form $\omega_+$ gives a holomorphic symplectic structure on $(\mathcal{M}', J)$. With respect to this structure, the projection $\pi: \mathcal{M}' \to \mathcal{B}'$ is holomorphic, and the torus fibers $\mathcal{M}'_u = \pi^{-1}(u)$ are compact complex Lagrangian submanifolds. 

Recall that a positive 2-form $\omega$ on a complex manifold is a real 2-form for which $\omega(v,Jv) >0$ for all real tangent vectors $v$. From now on, we assume that $\left\langle dZ \wedge d\overline{Z} \right\rangle$ is a positive 2-form on $\mathcal{B}'$. Now fix $R > 0$. Then we can define a 2-form on $\mathcal{M}'$ by
\begin{equation*}
\omega_3^{\text{sf}} = \frac{R}{4} \left\langle dZ \wedge d\overline{Z} \right\rangle - \frac{1}{8\pi^2 R}\left\langle d\theta \wedge d\theta \right\rangle.
\end{equation*}
This is a positive form of type (1,1) in the $J$ complex structure. Thus, the triple $(\mathcal{M}', J, \omega_3^{\text{sf}})$ determines a K\"{a}hler metric $g^{\text{sf}}$ on $\mathcal{M}'$. This metric is in fact hyperk\"{a}hler (see \cite{freed}), so we have a whole $\cpone$-worth of complex structures for $\mathcal{M}'$, parametrized by $\zeta \in \cpone$. The above complex structure $J$ represents $J(\zeta = 0)$, the complex structure at $\zeta = 0$ in $\cpone$. The superscript ${}^\text{sf}$ stands for ``semiflat''. This is because $g^{\text{sf}}$ is flat on the torus fibers $\mathcal{M}'_u$.

Alternatively, it is shown in \cite{gaiotto} that if
\begin{equation}\label{xsfr}
\mathcal{X}_\gamma^{\text{sf}}(\zeta) = \exp\left( \frac{\pi R Z_\gamma}{\zeta} + i\theta_\gamma + \pi R \zeta \overline{Z_\gamma} \right)
\end{equation}
Then the 2-form
\[ \varpi(\zeta) = \frac{1}{8\pi^2 R} \left\langle d\log \mathcal{X}^{\text{sf}}(\zeta) \wedge d\log \mathcal{X}^{\text{sf}}(\zeta) \right\rangle \]
(where the DeRham operator $d$ is applied to the $\mathcal{M}'$ part only) can be expressed as
\[ -\frac{i}{2\zeta}\omega_+ + \omega^{\text{sf}}_3 -\frac{i \zeta}{2} \omega_-, \]
for $\omega_- = \overline{\omega_+} = \left\langle d\overline{Z} \wedge d\overline{\theta} \right\rangle$, that is, in the twistor space $\mathcal{Z} = \mathcal{M}' \times \cpone$ of \cite{hitchin}, $\varpi(\zeta)$ is a holomorphic section of $\Omega_{\mathcal{Z}/\cpone} \otimes \mathcal{O}(2)$ (the twisting by $\mathcal{O}(2)$ is due to the poles at $\zeta = 0$ and $\zeta = \infty$ in $\cpone$). This is the key step in Hitchin's twistor space construction. By \cite[\S 3]{gaiotto}, $\mathcal{M}'$ is hyperk\"{a}hler.

We want to reproduce the same construction of a hyperk\"{a}hler metric now with corrected Darboux coordinates $\mathcal{X}_\gamma(\zeta)$. For that, we need another piece of data. Namely, a function $\Omega : \Gamma \to \mathbb{Z}$ such that $\Omega(\gamma;u) = \Omega(-\gamma;u)$. Furthermore, we impose a condition on the nonzero $\Omega(\gamma;u)$. Introduce a positive definite norm on $\Gamma$. Then we require the existence of $K > 0$ such that
\begin{equation}\label{support}
\frac{|Z_\gamma|}{\left\| \gamma \right\|} > K
\end{equation}
for those $\gamma$ such that $\Omega(\gamma; u) \neq 0$. This is called the \emph{Support Property}, as in \cite{gaiotto}.

For a component of the singular locus $D_0$ and for $\gamma_0$ the primitive element in $\Gamma$ for which $Z_{\gamma_0} \to 0$ as $u \to u_0 \in D_0$, we also require
\[ \Omega(\gamma_0; u) = 1 \text{ for all $u$ in a neighborhood of $D_0$} \]

To see where these invariants arise from, consider the example of moduli spaces of Higgs bundles again. A quadratic differential $u \in \mathcal{B}'$ determines a metric $h$ on $C$. Namely, if $u = P(z)dz^2$, $h = |P(z)| dz d\overline{z}$. Let $C'$ be the curve obtained after removing the poles and zeroes of $u$. Consider the finite length inextensible geodesics on $C'$ in the metric $h$. These come in two types:

\begin{enumerate}
	\item \textit{Saddle connections}: geodesics running between two zeroes of $u$. See Figure \ref{saddlec}.
	
	\begin{figure}[htbp]
	\centering
		\includegraphics[width=0.40\textwidth]{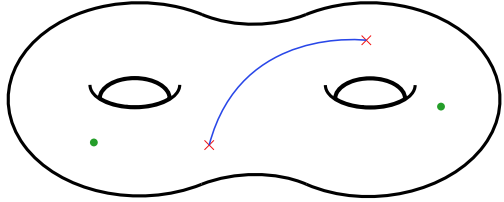}
	\caption{Saddle connections on $C'$}
	\label{saddlec}
\end{figure}

\item \textit{Closed geodesics}: When they exist, they come in 1-parameter families sweeping out annuli in $C'$. See Figure \ref{clgeod}.

	\begin{figure}[htbp]
	\centering
		\includegraphics[width=0.40\textwidth]{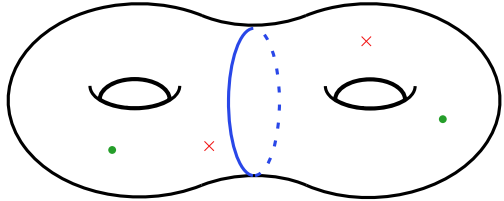}
	\caption{Closed geodesics on $C'$ sweeping annuli}
	\label{clgeod}
\end{figure}

\end{enumerate}

On the branched cover $\Sigma_u \to C$, each geodesic can be lifted to a union of closed curves in $\Sigma_u$, representing some homology class $\gamma \in H_1(\Sigma_u, \mathbb{Z})$. See Figure \ref{lift}.

	\begin{figure}[htbp]
	\centering
		\includegraphics[width=0.40\textwidth]{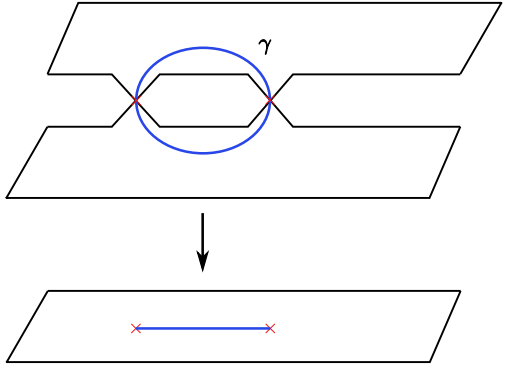}
	\caption{Lift of geodesics to $\Sigma_u$}
	\label{lift}
\end{figure}

 In this case, $\Omega(\gamma,u)$ counts these finite length geodesics: every saddle connection with lift $\gamma$ contributes $+1$ and every closed geodesic with lift $\gamma$ contributes $-2$.

Back to the general case, we're ready to formulate a Riemann-Hilbert problem on the $\cpone$-slice of the twistor space $\mathcal{Z} = \mathcal{M}' \times \cpone$. Recall that in a RH problem we have a contour $\Sigma$ dividing a complex plane (or its compactification) and one tries to obtain functions which are analytic in the regions defined by the contour, with continuous extensions along the boundary and with prescribed discontinuities along $\Sigma$ and fixed asymptotics at the points where $\Sigma$ is non-smooth. In our case, the contour is a collection of rays at the origin and the discontinuities can be expressed as symplectomorphisms of a complex torus:

Define a ray associated to each $\gamma \in \Gamma_u$ as:
\[ \ell_\gamma(u) = Z_\gamma \mathbb{R}_-. \]
We also define a transformation of the functions $\mathcal{X}_{\gamma'}$ given by each $\gamma \in \Gamma_u$:
\begin{equation}\label{kjump}
\mathcal{K}_\gamma \mathcal{X}_{\gamma'} = \mathcal{X}_{\gamma'} (1- \mathcal{X}_{\gamma})^{\left\langle \gamma', \gamma \right\rangle}
\end{equation}
Let $T_u$ denote the space of twisted complex characters of $\Gamma_u$, i.e. maps $\mathcal{X} : \Gamma_u \to \cone^{\times}$ satisfying
\begin{equation}\label{xprop}
 \mathcal{X}_\gamma \mathcal{X}_{\gamma'} = (-1)^{\left\langle \gamma, \gamma'\right\rangle} \mathcal{X}_{\gamma + \gamma'}
\end{equation}
$T_u$ has a canonical Poisson structure given by
\[ \{ \mathcal{X}_\gamma,  \mathcal{X}_{\gamma'} \} = \left\langle \gamma, \gamma' \right\rangle \mathcal{X}_{\gamma + \gamma'}\]
The $T_u$ glue together into a bundle over $\mathcal{B}'$ with fiber a complex Poisson torus. Let $T$ be the pullback of this system to $\mathcal{M}'$. We can interpret the transformations $\mathcal{K}_\gamma$ as birational automorphisms of $T$.
To each ray $\ell$ going from 0 to $\infty$ in the $\zeta$-plane, we can define a transformation
\begin{equation}\label{stkfac}
S_\ell = \prod_{\gamma : \ell_\gamma(u) = \ell} \mathcal{K}_\gamma^{\Omega(\gamma;u)}
\end{equation}
Note that all the $\gamma$'s involved in this product are multiples of each other, so the $\mathcal{K}_\gamma$ commute and it is not necessary to specify an order for the product.

To obtain the corrected $\mathcal{X}_\gamma$, we can formulate a Riemann-Hilbert problem for which the former functions are solutions to it.  We seek a map $\mathcal{X} : \mathcal{M}'_u \times \cone^{\times} \to T_u$ with the following properties:
\begin{enumerate}[label=\textnormal{(\arabic*)}]
	\item $\mathcal{X}$ depends piecewise holomorphically on $\zeta$, with discontinuities only at the rays $\ell_\gamma(u)$ for which $\Omega(\gamma;u) \neq 0$.
	\item The limits $\mathcal{X}^{\pm}$ as $\zeta$ approaches any ray $\ell$ from both sides exist and are related by
	\begin{equation}\label{invjmp}
	\mathcal{X}^+ = S_\ell^{-1} \circ \mathcal{X}^-
	\end{equation}
	\item $\mathcal{X}$ obeys the reality condition
	\begin{equation}\label{realcond}
	\overline{\mathcal{X}_{-\gamma}(-1/\overline{\zeta})} = \mathcal{X}_\gamma(\zeta)
	\end{equation}
	\item For any $\gamma \in \Gamma_u$, $\lim_{\zeta \to 0} \mathcal{X}_\gamma(\zeta) / \mathcal{X}^{\text{sf}}_\gamma(\zeta)$ exists and is real. \label{asymptotic}
\end{enumerate}

In \cite{gaiotto}, this RH problem is formulated as an integral equation:
\begin{equation}\label{inteq}
\mathcal{X}_\gamma(u,\zeta) = \mathcal{X}^{\text{sf}}_\gamma(u,\zeta)\exp\left[ -\frac{1}{4\pi i} \sum_{\gamma'} \Omega(\gamma';u) \left\langle \gamma, \gamma' \right\rangle \int_{\ell_{\gamma'(u)}} \frac{d\zeta'}{\zeta'} \frac{\zeta'+\zeta}{\zeta'-\zeta}\log\left( 1 - \mathcal{X}_{\gamma'}(u,\zeta')\right)\right],
\end{equation}
One can define recursively, setting $\mathcal{X}^{(0)} = \mathcal{X}^{\text{sf}}$:
\begin{equation}\label{recurs}
\mathcal{X}^{(\nu+1)}_\gamma(u,\zeta) = \mathcal{X}^{\text{sf}}_\gamma(u,\zeta)\exp\left[ -\frac{1}{4\pi i} \sum_{\gamma'} \Omega(\gamma';u) \left\langle \gamma, \gamma' \right\rangle \int_{\ell_{\gamma'(u)}} \frac{d\zeta'}{\zeta'} \frac{\zeta'+\zeta}{\zeta'-\zeta}\log\left( 1 - \mathcal{X}^{(\nu)}_{\gamma'}(u,\zeta')\right)\right],
\end{equation}

More precisely, we have a family of RH problems, parametrized by $u \in \mathcal{B}'$, as this defines the rays $\ell_\gamma(u)$, the complex torus $T_u$ where the symplectomorphisms are defined and the invariants $\Omega(\gamma;u)$ involved in the definition of the problem.

We still need one more piece of the puzzle, since the latter function $\Omega$ may not be continuous. In fact, $\Omega$ jumps along a real codimension-1 loci in $\mathcal{B}'$ called the ``wall of marginal stability''. This is the locus where 2 or more functions $Z_\gamma$ coincide in phase, so two or more rays $\ell_{\gamma}(u)$ become one. More precisely:
\[ W = \{u \in \mathcal{B}': \exists \gamma_1, \gamma_2 \text{ with } \Omega(\gamma_1;u) \neq 0, \Omega(\gamma_2;u) \neq 0, \left\langle \gamma_1, \gamma_2\right\rangle \neq 0, Z_{\gamma_1}/Z_{\gamma_2} \in \rone_+\}\]
 The jumps of $\Omega$ are not arbitrary; they are governed by the Kontsevich-Soibelman wall-crossing formula.

To describe this, let $V$ be a strictly convex cone in the $\zeta$-plane with apex at the origin. Then for any $u \notin W$ define
\begin{equation}
A_V(u) = \prod^\ccwarrow_{\gamma : Z_\gamma(u) \in V} \mathcal{K}_\gamma^{\Omega(\gamma;u)} = \prod^\ccwarrow_{\ell \subset V} S_\ell\footnote{This product may be infinite. One should more precisely think of $A_V(u)$ as living in a certain prounipotent completion of the group generated by $\{\mathcal{K}_\gamma\}_{\gamma : Z_\gamma(u) \in V}$ as explained in \cite{kont}}
\end{equation}

The arrow indicates the order of the rational maps $\mathcal{K}_\gamma$. $A_V(u)$ is a birational Poisson automorphism of $T_u$. Define a $V$-\textit{good path} to be a path $p$ in $\mathcal{B}'$ along which there is no point $u$ with $Z_\gamma(u) \in \partial V$ and $\Omega(\gamma;u) \neq 0$. (So as we travel along a $V$-good path, no $\ell_\gamma$ rays enter or exit V.) If $u, u'$ are the endpoints of a $V$-good path $p$, the wall-crossing formula is the condition that $A_V(u), A_V(u')$ are related by parallel transport in $T$ along $p$. See Figure \ref{partrnpt}.

	\begin{figure}[htbp]
	\centering
		\includegraphics[width=0.80\textwidth]{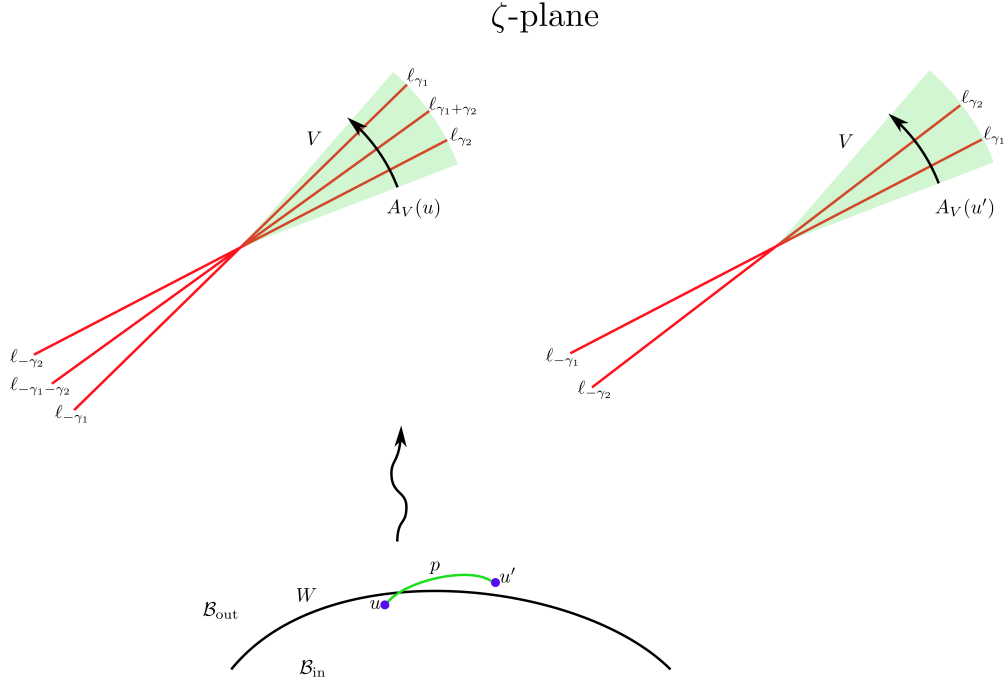}
	\caption{For a good path $p$, the two automorphisms $A_V(u), A_V(u')$ are related by parallel transport}
	\label{partrnpt}
\end{figure}


\subsection{Statement of Results}\label{results}

We will restrict in this paper to the case $\dim_\cone \mathcal{B} = 1$, so $n = \dim \Gamma = 2$. We want to extend the torus fibration $\mathcal{M}'$ to a manifold $\mathcal{M}$ with degenerate torus fibers.  To give an example, in the case of Hitchin systems, the torus bundle $\mathcal{M}'$ is not the moduli space of Higgs bundles yet, as we have to consider quadratic differentials with non-simple zeroes too. The main results of this paper center on the extension of the manifold $\mathcal{M}'$ to a manifold $\mathcal{M}$ with an extended fibration $\mathcal{M} \to \mathcal{B}$ such that the torus fibers $\mathcal{M}'_u$ degenerate to nodal torus (i.e. ``singular'' or ``bad'' fibers) for $u \in D$.

We start by fully working out the simplest example known as Ooguri-Vafa \cite{cecotti}. Here we have a fibration over the open unit disk $\mathcal{B} := \{u \in \cone : |u| < 1 \}$. At the discriminant locus $D : = \{ u = 0 \}$, the fibers degenerate into a nodal torus. The local rank-2 lattice $\Gamma$ has a basis $(\gamma_m, \gamma_e)$ and the skew-symmetric pairing is defined by $\left\langle \gamma_m, \gamma_e \right\rangle = 1$. The monodromy of $\Gamma$ around $u = 0$ is $\gamma_e \mapsto \gamma_e, \gamma_m \mapsto \gamma_m + \gamma_e$. We also have functions $Z_{\gamma_e}(u) = u, Z_{\gamma_m}(u) = \frac{u}{2\pi i }( \log u - 1) + f(u)$, for $f$ holomorphic and admitting an extension to $\mathcal{B}$. Finally, the integer-valued function $\Omega$ in $\Gamma$ is here: $\Omega(\pm \gamma_e; u) = 1$ and $\Omega(\gamma; u) = 0$ for any other $\gamma \in \Gamma_u$. There is no wall of marginal stability in this case. The integral equation \eqref{inteq} can be solved after just 1 iteration.

For all other nontrivial cases, in order to give a satisfactory extension of the $\mathcal{X}_\gamma$ coordinates, it was necessary to develop the theory of Riemann-Hilbert-Birkhoff problems to suit these infinite-dimensional systems (as the transformations $S_\ell$ defining the problem can be thought as operators on $C^\infty(T_u)$, rather than matrices). It is not clear that such coordinates can be extended, since we may approach the bad fiber from two different sides of the wall of marginal stability and obtain two different extensions. To overcome this first obstacle, we have to use the theory of isomonodromic deformations as in \cite{boalch} to reformulate the Riemann-Hilbert problem in \cite{gaiotto} independent of the regions determined by the wall.

Having redefined the problem, we want our $\mathcal{X}_\gamma$ to be smooth on the parameters $\theta_{\gamma_1},  \theta_{\gamma_{2}}$ and $u$,
away from where the prescribed jumps are. Even at $\mathcal{M}'$, there was no mathematical proof that such condition must be true. In the companion paper \cite{rhprob}, we combine classical Banach contraction methods and Arzela-Ascoli results on uniform convergence in compact sets to obtain:
\begin{theorem}\label{smooth}
If the collection $J$ of nonzero $\Omega(u; \gamma)$ satisfies the support property \eqref{support} and if the parameter $R$ of \eqref{xsfr} is large enough (determined by the values $|Z_\gamma(u)|, \gamma \in J$), there exists a unique collection of functions $\mathcal{X}_\gamma$ with the prescribed asymptotics and jumps as in \cite{gaiotto}. These functions are smooth on $u$ and the torus coordinates $\theta_1, \theta_{2}$ (even for $u$ at the wall of marginal stability), and piecewise holomorphic on $\zeta$.
\end{theorem}

Since we're considering only the case $n=1$, $\Gamma$ is a rank-1 lattice over the Riemann surface $\mathcal{B}'$ and the discriminant locus $D$ where the torus fibers degenerate is a discrete subset of $\mathcal{B}'$. 

From this point on, we restrict our attention to the next nontrivial system, known as the Pentagon case \cite{notes}. Here $\mathcal{B} = \cone$ with 2 bad fibers which we can assume are at $u = -2, u = 2$ and $\mathcal{B}'$ is the twice-punctured plane. There is a wall of marginal stability where all $Z_\gamma$ are contained in the same line. This separates $\mathcal{B}$ in two domains $\mathcal{B}_\text{out}$ and a simply-connected $\mathcal{B}_\text{in}$. See Figure \ref{aplane}.

\begin{figure}[htbp]
	\centering
		\includegraphics[width=0.60\textwidth]{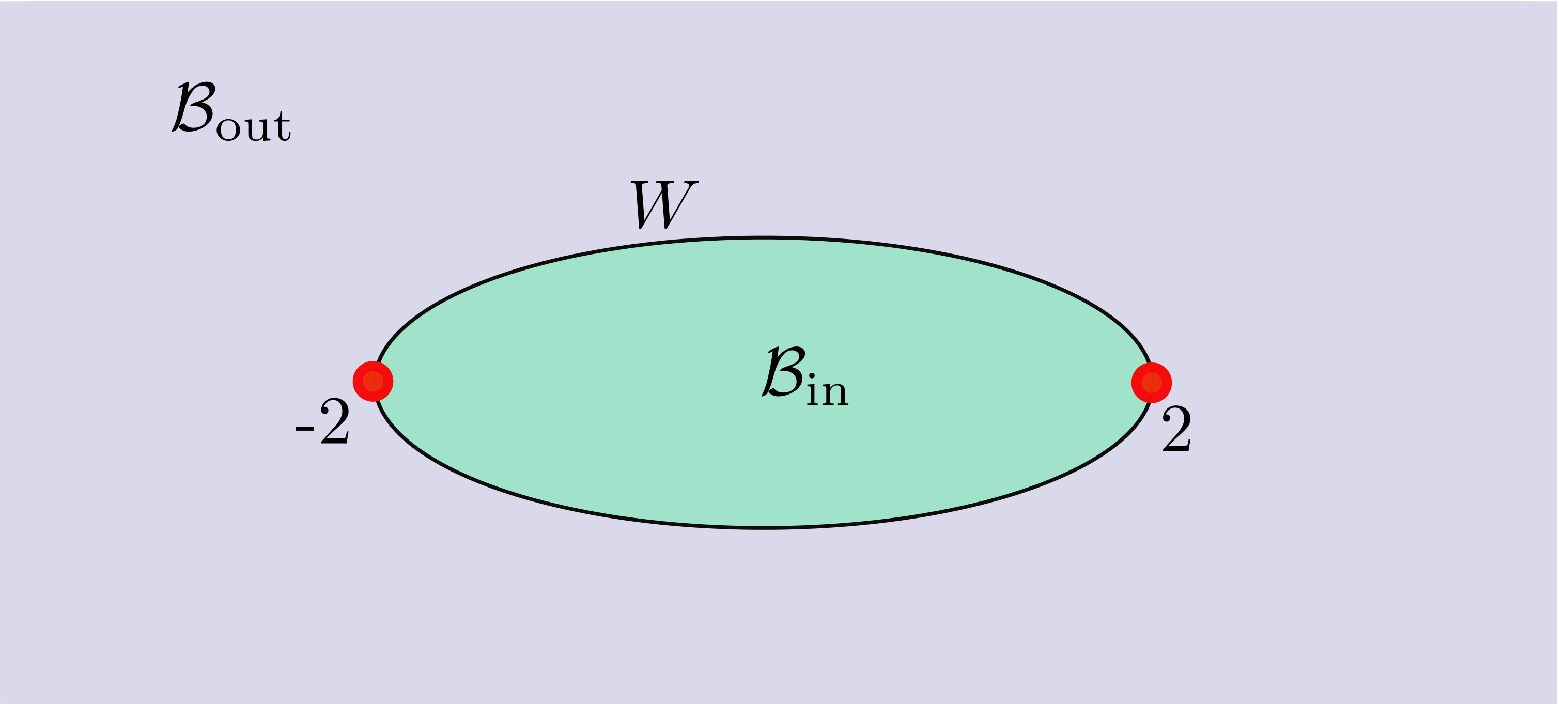}
	\caption{The wall $W$ in $\mathcal{B}$ for the Pentagon case}
	\label{aplane}
\end{figure}

On $\mathcal{B}_\text{in}$ we can trivialize $\Gamma$ and choose a basis $\{\gamma_1, \gamma_2\}$ with pairing $\left\langle \gamma_1, \gamma_2\right\rangle = 1$. This basis does not extend to a global basis for $\Gamma$ since it is not invariant under monodromy. However, the set $\{\gamma_1, \gamma_2, -\gamma_1, -\gamma_2, \gamma_1 + \gamma_2, -\gamma_1 - \gamma_2\}$ is indeed invariant so the following definition of $\Omega$ makes global sense:
  \begin{align*}
  \text{For $u \in \mathcal{B}_\text{in}$}, \Omega(\gamma; u) = & \left\{ \begin{array}{ll}
    1 & \text{for } \gamma \in \{ \gamma_1, \gamma_2, -\gamma_1, -\gamma_2\}\\
    0 & \text{otherwise}
  \end{array} \right. \\
 \text{For $u \in \mathcal{B}_\text{out}$} , \Omega(\gamma; u) = & \left\{ \begin{array}{ll}
    1 & \text{for } \gamma \in \{ \gamma_1, \gamma_2, -\gamma_1, -\gamma_2, \gamma_1 + \gamma_2, -\gamma_1 - \gamma_2\}\\
    0 & \text{otherwise}
  \end{array}  \right.
  \end{align*}
  
  
  
   The Pentagon case appears in the study of Hitchin systems with gauge group $\text{SU}(2)$. The extension of $\mathcal{M}'$ was previously obtained by hyperk\"{a}hler quotient methods in \cite{biquard}, but no explicit hyperk\"{a}hler metric was constructed. 

Once the $\{ \mathcal{X}_{\gamma_i} \}$ are obtained by Theorem \ref{smooth}, 
it is necessary to do an analytic continuation along $\mathcal{B}'$ for the particular $\mathcal{X}_{\gamma_i}$ for which $Z_{\gamma_i} \to 0$ as $u \to u_0 \in D$. Without loss of generality, we can assume there
is a local basis $\{\gamma_1, \gamma_{2}\}$ of $\Gamma$ such that $Z_{\gamma_2} \to 0$ in $D$. After that, an analysis of the possible divergence of $\mathcal{X}_\gamma$ as $u \to u_0$ shows the necessity of performing a gauge transformation on the torus coordinates of the fibers $\mathcal{M}_u$ that allows us to define an integral equation even at $u_0 \in D$. This series of transformations are defined in \eqref{newmp}, \eqref{outmp}, \eqref{outmp2} and \eqref{fingauge}, and constitute a new result that was not expected in \cite{gaiotto}. We basically deal with a family of boundary value problems for which the jump function vanishes at certain points and
singularities of certain kind appear as $u \to u_0$. As this is of independent interest, we leave the relevant results to \cite{rhprob} and we show that our solutions contain at worst branch singularities at 0 or $\infty$ in the $\zeta$-plane. As in the case of normal fibers, we can run a contraction argument to obtain Darboux coordinates even at the singular fibers and conclude:

\begin{theorem}\label{extbf}
Let $\{\gamma_{1}, \gamma_2\}$ be a local basis for $\Gamma$ in a small sector centered at $u_0 \in D$ such that $Z_{\gamma_2} \to 0$ as $u \to u_0 \in D$. For the Pentagon integrable system, the local function $\mathcal{X}_{\gamma_1}$ admits an analytic continuation $\widetilde{\mathcal{X}}_{\gamma_1}$ to a punctured disk centered at $u_0$ in $\mathcal{B}$. There exists a gauge transformation $\theta_1 \mapsto \widetilde{\theta}_1$ that extends the torus fibration $\mathcal{M}'$ to a manifold $\mathcal{M}$ that is locally, for each point in $D$, a (trivial) fibration over $\mathcal{B} \times S^{1}$ with fiber $S^1$ coordinatized by $\theta_1$ and with one fiber collapsed into a point. For $R > 0$ big enough, it is possible to extend $\widetilde{\mathcal{X}}_{\gamma_1}$ and $\mathcal{X}_{\gamma_2}$ to $\mathcal{M}$, still preserving the smooth properties as in Theorem \ref{smooth}.
\end{theorem}


After we have the smooth extension of the $\{ \mathcal{X}_{\gamma_i} \}$ by Theorem \ref{extbf}, we can extend the holomorphic symplectic form $\varpi(\zeta)$ labeled by $\zeta \in \cpone$ as in \cite{hitchin} for all points except possibly one at the singular fiber. From $\varpi(\zeta)$ we can obtain the hyperk\"{a}hler metric $g$ and, in the case of the Pentagon, after a change of coordinates, we realize $g$ locally as the Taub-NUT metric plus smooth corrections, finishing the construction of $\mathcal{M}$ and its hyperk\"{a}hler metric. The following is the main theorem of the paper.

\begin{theorem}\label{smfrm}
For the Pentagon case, the extension $\mathcal{M}$ of the manifold $\mathcal{M}'$ constructed in Theorem \ref{extbf} admits, for $R$ large enough, a hyperk\"{a}hler metric $g$ obtained by extending the hyperk\"{a}hler metric on $\mathcal{M}'$ determined by the Darboux coordinates $\{ \mathcal{X}_{\gamma_i} \}$.
\end{theorem}


\section{The Ooguri-Vafa Case}\label{ov}

\subsection{Classical Case}\label{clasov}

We start with one of the simplest cases, known as the Ooguri-Vafa case, first treated in \cite{cecotti}. To see where this case comes from, recall that by the SYZ picture of K3 surfaces \cite{gross}, any K3 surface $\mathcal{M}$ is a hyperk\"{a}hler manifold. In one of its complex structures (say $J^{(\zeta = 0)}$) is elliptically fibered, with base manifold $\mathcal{B} = \cpone$ and generic fiber a compact complex torus. There are a total of 24 singular fibers, although the total space is smooth. See Figure \ref{k3}.

	\begin{figure}[htbp]
	\centering
		\includegraphics[width=0.40\textwidth]{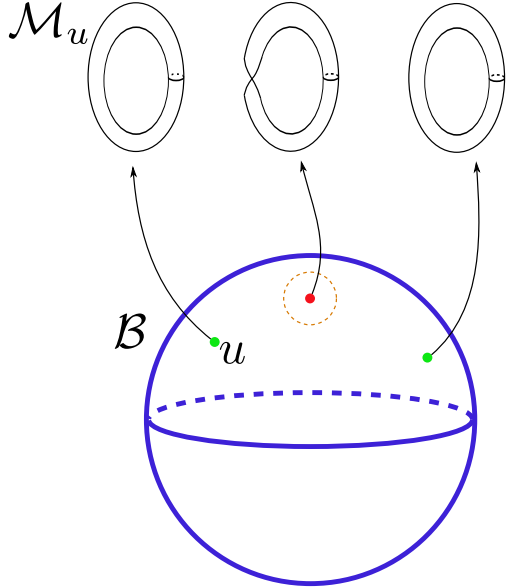}
	\caption{A K3 surface $\mathcal{M}$ as an elliptic fibration}
	\label{k3}
\end{figure}

Gross and Wilson \cite{gross2} constructed a hyperk\"{a}hler metric $g$ on a K3 surface by gluing in the Ooguri-Vafa metric constructed in \cite{oovf} with a standard metric $g^{\text{sf}}$ away from the degenerate fiber. Thus, this simple case can be regarded as a local model for K3 surfaces.
 
We have a fibration over the open unit disk $\mathcal{B} := \{a \in \cone : |a| < 1 \}$. At the locus $D : = \{ a = 0 \}$ (in the literature this is also called the \textit{discriminant locus}), the fibers degenerate into a nodal torus. Define $\mathcal{B}'$ as $\mathcal{B} \backslash D$, the punctured unit disk. On $\mathcal{B}'$ there exists a local system $\Gamma$ of rank-2 lattices with basis $(\gamma_m, \gamma_e)$ and skew-symmetric pairing defined by $\left\langle \gamma_m, \gamma_e \right\rangle = 1$. The monodromy of $\Gamma$ around $a = 0$ is $\gamma_e \mapsto \gamma_e, \gamma_m \mapsto \gamma_m + \gamma_e$. We also have functions $Z_{\gamma_e}(a) = a, Z_{\gamma_m}(a) = \frac{a}{2\pi i }( \log a - 1)$. On $\mathcal{B}'$ we have local coordinates $(\theta_m, \theta_e)$ for the torus fibers with monodromy $\theta_e \mapsto \theta_e, \theta_m \mapsto \theta_m + \theta_e - \pi$. Finally, the integer-valued function $\Omega$ in $\Gamma$ is here: $\Omega(\pm \gamma_e, a) = 1$ and $\Omega(\gamma, a) = 0$ for any other $\gamma \in \Gamma_a$. There is no wall of marginal stability in this case.

 We call this the ``classical Ooguri-Vafa'' case as it is the one appearing in \cite{oovf} already mentioned at the beginning of this section. In the next section, we'll generalize this case by adding a function $f(a)$ to the definition of $Z_{\gamma_m}$.

Let
\begin{equation}\label{xesf}
 \mathcal{X}^{\text{sf}}_\gamma(\zeta, a) := \exp\left( \pi R \zeta^{-1} Z_\gamma(a) + i\theta_\gamma + \pi R \zeta \overline{Z_\gamma(a)}\right)
\end{equation}
These functions receive corrections defined as in \cite{gaiotto}. We are only interested in the pair $(\mathcal{X}_m, \mathcal{X}_e)$ which will constitute our desired Darboux coordinates for the holomorphic symplectic form $\varpi$. The fact that $\Omega(\gamma_m, a) = 0$ gives that $\mathcal{X}_e = \mathcal{X}^{\text{sf}}_e$. As $a \to 0$, $Z_{\gamma_e}$ and $Z_{\gamma_m}$ approach 0. Thus $\mathcal{X}_e|_{a = 0} = e^{i\theta_e}$. Since $\mathcal{X}_e = \mathcal{X}^{\text{sf}}_e$ the actual $\mathcal{X}_m$ is obtained after only 1 iteration of \eqref{recurs}. For each $a \in \mathcal{B}'$, let $\ell_+$ be the ray in the $\zeta$-plane defined by $\{\zeta : a/\zeta \in \rone_- \}$. Similarly, $\ell_- : = \{\zeta : a/\zeta \in \rone_+\}$.

Let
\begin{equation}\label{defxm}
\mathcal{X}_m = \mathcal{X}^{\text{sf}}_m \exp \left[ \frac{i}{4\pi} \int_{\ell_+} \frac{d\zeta'}{\zeta'} \frac{\zeta' +
 \zeta}{\zeta' - \zeta} \log[1 - \mathcal{X}_e(\zeta')] -  \frac{i}{4\pi} \int_{\ell_-} \frac{d\zeta'}{\zeta'} \frac{\zeta' +
 \zeta}{\zeta' - \zeta} \log[1 - \mathcal{X}_e(\zeta')^{-1}] \right].
\end{equation}
For convenience, from this point on we assume $a$ is of the form  $sb$, where $s$ is a positive number, $b$ is fixed and $|b| = 1$. Moreover, in $\ell_+$, $\zeta' = -tb$, for $t \in (0, \infty)$, and a similar parametrization holds in $\ell_-$.
\begin{lemma}
For fixed $b$, $\mathcal{X}_m$ as in \eqref{defxm} has a limit as $|a| \to 0$.
\end{lemma}
\begin{proof}
Writing $\dfrac{\zeta' + \zeta}{\zeta'(\zeta' - \zeta)} = \dfrac{-1}{\zeta'} + \dfrac{2}{\zeta' - \zeta}$, we want to find the limit as $a \to 0$ of
\begin{align}
& \int_{\ell_+} \left\{ \dfrac{-1}{\zeta'} + \dfrac{2}{\zeta' - \zeta} \right\} \log[1 - \exp(\pi Ra/\zeta' + i\theta_e + \pi R\zeta' \bar{a})] d\zeta' \notag\\
& - \int_{\ell_-} \left\{ \dfrac{-1}{\zeta'} + \dfrac{2}{\zeta' - \zeta} \right\} \log[1 - \exp(-\pi Ra/\zeta' - i\theta_e - \pi R\zeta' \bar{a})] d\zeta' \label{integs}.
\end{align}
For simplicity, we'll focus in the first integral only, the second one can be handled similarly. Rewrite:
\begin{align}
& \int_{\ell_+} \left\{ \dfrac{-1}{\zeta'} + \dfrac{2}{\zeta' - \zeta} \right\} \log[1 - \exp(\pi Ra/\zeta' + i\theta_e + \pi R\zeta' \bar{a})] d\zeta' \notag\\
& = \int_{0}^{-b} \left\{ \dfrac{-1}{\zeta'} + \dfrac{2}{\zeta' - \zeta} \right\} \log[1 - \exp(\pi Ra/\zeta' + i\theta_e + \pi R\zeta' \bar{a})] d\zeta'  \notag\\
& + \int_{-b}^{-b\infty} \left\{ \dfrac{-1}{\zeta'} + \dfrac{2}{\zeta' - \zeta} \right\} \log[1 - \exp(\pi Ra/\zeta' + i\theta_e + \pi R\zeta' \bar{a})] d\zeta' \notag\\
& = \int_{0}^{-b} \left\{ \dfrac{-1}{\zeta'} + \dfrac{2}{\zeta' - \zeta} \right\} \log[1 - \exp(\pi Ra/\zeta' + i\theta_e + \pi R\zeta' \bar{a})] d\zeta' \notag\\
& + \int_{-b}^{-b\infty} \left\{ \dfrac{-1}{\zeta'} + \frac{2}{\zeta'} + \dfrac{2}{\zeta' - \zeta} - \frac{2}{\zeta'} \right\} \log[1 - \exp(\pi Ra/\zeta' + i\theta_e + \pi R\zeta' \bar{a})] d\zeta' \notag\\
& = \int_{0}^{-b} \dfrac{-1}{\zeta'} \log[1 - \exp(\pi Ra/\zeta' + i\theta_e + \pi R\zeta' \bar{a})] d\zeta' \notag\\
& + \int_{-b}^{-b\infty} \dfrac{1}{\zeta'} \log[1 - \exp(\pi Ra/\zeta' + i\theta_e + \pi R\zeta' \bar{a})] d\zeta' \notag\\
& + \int_{0}^{-b} \dfrac{2}{\zeta' - \zeta} \log[1 - \exp(\pi Ra/\zeta' + i\theta_e + \pi R\zeta' \bar{a})] d\zeta' \notag\\
& + \int_{-b}^{-b\infty} \left\{ \dfrac{2}{\zeta' - \zeta} - \frac{2}{\zeta'} \right\} \log[1 - \exp(\pi Ra/\zeta' + i\theta_e + \pi R\zeta' \bar{a})] d\zeta' \label{4sums}
\end{align}

\noindent Observe that
\begin{align*}
& \int_{0}^{-b} \dfrac{-1}{\zeta'} \log[1 - \exp(\pi Ra/\zeta' + i\theta_e + \pi R\zeta' \bar{a})] d\zeta'\\
& = -\int_0^1 \frac{1}{t} \log[1 - \exp(-\pi Rs(t + 1/t))] dt\\
\intertext{and after a change of variables $\tilde{t} = 1/t$, we get}
& = -\int_1^\infty \frac{1}{\tilde{t}} \log[1 - \exp(-\pi Rs(\tilde{t} + 1/\tilde{t}))] d\tilde{t}\\
& = -\int_{-b}^{-b\infty} \dfrac{1}{\zeta'} \log[1 - \exp(\pi Ra/\zeta' + i\theta_e + \pi R\zeta' \bar{a})] d\zeta'.
\end{align*}
Thus, (\ref{4sums}) reduces to
\begin{align}
& \int_{0}^{-b} \dfrac{2}{\zeta' - \zeta} \log[1 - \exp(\pi Ra/\zeta' + i\theta_e + \pi R\zeta' \bar{a})] d\zeta' \notag\\
& + \int_{-b}^{-b\infty} \left\{ \dfrac{2}{\zeta' - \zeta} - \frac{2}{\zeta'} \right\} \log[1 - \exp(\pi Ra/\zeta' + i\theta_e + \pi R\zeta' \bar{a})] d\zeta' \label{2sums}.
\end{align}
If $\theta_e = 0$, (\ref{integs}) diverges to $-\infty$, in which case $\mathcal{X}_m = 0$. Otherwise, $\log[1 - \exp(\pi Ra/\zeta' + i\theta_e + \pi R\zeta' \bar{a})]$ is bounded away from 0. Consequently,
$|\log[1 - \exp(\pi Ra/\zeta' + i\theta_e + \pi R\zeta' \bar{a})]| < C < \infty$ in $\ell_+$.
As $a \to 0$, the integrals are dominated by
\[ \int_0^{-b} \dfrac{2C}{|\zeta' - \zeta|} |d\zeta'| + \int_{-b}^{-b\infty} \frac{C|\zeta/b|}{|\zeta'(\zeta' - \zeta)|} |d\zeta'| < \infty \]
if $\theta_e \neq 0$. Hence we can interchange the limit and the integral in (\ref{2sums}) and obtain that, as $a \to 0$, this reduces to

\begin{align}
 & 2\log(1 - e^{i \theta_e})\left[\int_{0}^{-b} \frac{d\zeta'}{\zeta' - \zeta} + \int_{-b}^{-b\infty} d\zeta' \left\{ \frac{1}{\zeta' - \zeta} - \frac{1}{\zeta'}\right\} \right] \notag\\
 & = 2\log(1 - e^{i \theta_e})[F(-b) + G(-b)], \label{odes}
 \end{align}
where
\[ F(z) := \log\left( 1 - \dfrac{z}{\zeta}\right), G(z) := \log\left( 1 - \dfrac{\zeta}{z}\right) \]
 are the (unique) holomorphic solutions in the simply connected domain $U := \cone - \{z : z/\zeta \in \rone_+\}$ to the ODEs
\[ F'(z) = \frac{1}{z - \zeta}, F(0) = 0 \hspace{10 mm} G'(z) = \frac{1}{z - \zeta} - \frac{1}{z}, \lim_{z \to \infty} G(z) = 0. \]
This forces us to rewrite (\ref{odes}) uniquely as
\begin{equation}\label{prinbran}
2\log(1 - e^{i \theta_e})\left[\log\left(1 + \frac{b}{\zeta}\right) - \log\left(1 + \frac{\zeta}{b}\right)\right]
\end{equation}
Here $\log$ denotes the principal branch of the log in both cases, and the equation makes sense for $\{b \in \cone : b \notin \ell_+ \}$ (recall that by construction, we have the additional datum $|b| = 1$). We want to conclude that
\begin{equation}\label{logfusion}
\log(1 + b/\zeta) - \log(1 + \zeta/b) = \log(b/\zeta),
\end{equation}
still using the principal branch of the log. To see this, define $H(z)$ as $F(z) - G(z) - \log(-z/\zeta)$. This is an analytic function on $U$ and clearly $H'(z) \equiv 0$. Thus $H$ is constant in $U$. It is easy to show that the identity holds for a suitable choice of $z$ (for example, if $\zeta$ is not real, choose $z = 1$) and by the above, it holds on all of $U$; in particular, for $z = -b$.

All the arguments so far can be repeated to the ray $\ell_-$ to get the final form of (\ref{integs}):
 \begin{equation}\label{fextov}
  2\left\{\log\left[\frac{b}{\zeta}\right]\log(1 - e^{i\theta_e})
   -\log\left[\frac{- b}{\zeta}\right]\log(1 - e^{-i\theta_e}) \right\}, \hspace{3 mm} \theta_e \neq 0.
  \end{equation}
   This yields that (\ref{defxm}) simplifies to:
\begin{align}
\mathcal{X}_m & = \mathcal{X}^{\text{sf}}_m \exp\left( \frac{i}{2\pi} \left\{ \log\left[\frac{b}{\zeta}\right]\log(1 - e^{i\theta_e})
   -\log\left[\frac{- b}{\zeta}\right]\log(1 - e^{-i\theta_e})\right\} \right) \notag\\
  & = \mathcal{X}^{\text{sf}}_m \exp\left( \frac{i}{2\pi} \left\{ \log\left[\frac{a}{|a|\zeta}\right]\log(1 - e^{i\theta_e})
   -\log\left[\frac{- a}{|a|\zeta}\right]\log(1 - e^{-i\theta_e})\right\} \right)  \label{xmnice}
\end{align}
in the limiting case $a \to 0$.
\end{proof}

To obtain a function that is continuous everywhere and independent of $\arg a$, define regions I, II and III in the $a$-plane as follows: $\mathcal{X}^{\text{sf}}_m$ has a fixed cut in the negative real axis, both in the $\zeta$-plane and the $a$-plane. Assuming for the moment that $\arg \zeta \in (0,\pi)$, define region I as the half plane $\{a \in \cone : \text{Im}\left( a/\zeta\right) < 0 \}$. Region II is that enclosed by the $\ell_-$ ray and the cut in the negative real axis, and region III is the remaining domain so that as we travel counterclockwise we traverse regions I, II and III in this order (see Figure \ref{3reg}).

\begin{figure}[htbp]
	\centering
		\includegraphics[width=0.50\textwidth]{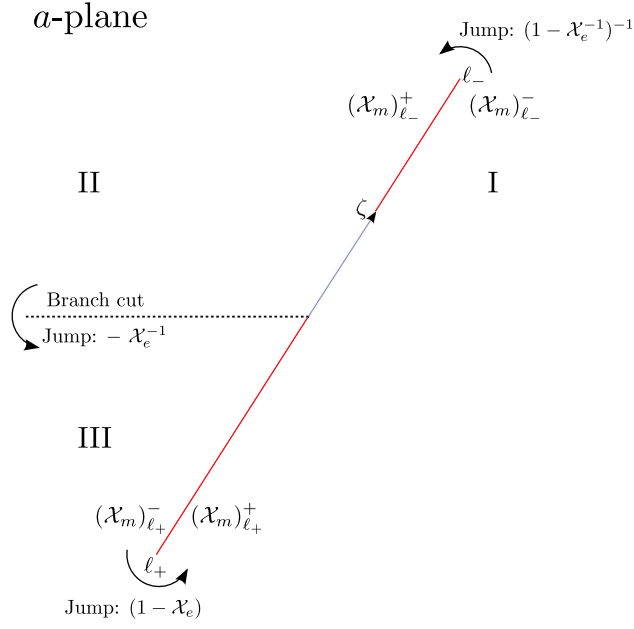}
	\caption{The three regions in the $a$-plane, as we traverse them counterclockwise}
	\label{3reg}
\end{figure}

For $a \neq 0$, Gaiotto, Moore and Neitzke \cite{gaiotto} proved that $\mathcal{X}_m$ has a continuous extension to the punctured disk of the form:

\begin{equation}\label{regmod}
 \widetilde{\mathcal{X}}_m = \left\{ \begin{array}{ll}
	   \mathcal{X}_m & \text{in region I}\\
	   (1 - \mathcal{X}^{-1}_e) \mathcal{X}_m & \text{in region II}\\
	   - \mathcal{X}_e (1 - \mathcal{X}^{-1}_e) \mathcal{X}_m = (1 - \mathcal{X}_e)\mathcal{X}_m  & \text{in region III}
	   \end{array} \right. 
\end{equation}

If we regard $\mathcal{M}'$ as a $S^1$-bundle over $\mathcal{B}' \times S^1$, with the fiber parametrized by $\theta_m$, then we seek to extend $\mathcal{M}'$ to a manifold $\mathcal{M}$ by gluing to $\mathcal{M}'$ another $S^1$-bundle over $D \times (0,2\pi)$, for $D$ a small open disk around $a = 0$, and $\theta_e \in (0,2\pi)$. The $S^1$-fiber is parametrized by a different coordinate $\theta'_m$ where the Darboux coordinate $\widetilde{\mathcal{X}}_m$ can be extended to $\mathcal{M}$. This is the content of the next theorem. 

\begin{theorem}\label{mprtom}
$\mathcal{M}'$ can be extended to a manifold $\mathcal{M}$ where the torus fibers over $\mathcal{B}'$ degenerate at $D = \{a = 0\}$ and $\widetilde{\mathcal{X}}_m $ can be extended to $D$, independent of the value of $\arg a$.
\end{theorem}
\begin{proof}
We'll use the following identities:
\begin{align}
\log(1 - e^{i\theta_e}) & = \log(1 - e^{-i\theta_e}) +i(\theta_e - \pi), \hspace{5mm} \text{for } \theta_e \in (0, 2\pi) \label{logs}\\
 \log\left[\frac{-a}{|a|\zeta}\right] & = \left\{ \begin{array}{ll}
     \log\left[\frac{a}{|a|\zeta}\right] + i\pi & \text{in region I}\\
	   \log\left[\frac{a}{|a|\zeta}\right] - i\pi & \text{in regions II and III}
	   \end{array} \right. \label{regions}\\
	   \log [a/\zeta] & = \left\{ \begin{array}{ll}
     \log a - \log \zeta & \text{in regions I and II}\\
	   \log a - \log \zeta + 2\pi i & \text{in region III}
	   \end{array} \right. \label{breaklog}
\end{align}
to obtain a formula for $\widetilde{\mathcal{X}}_m$ at $a = 0$ independent of the region. Formula \eqref{breaklog} can be proved with an argument analogous to that used for the proof of \eqref{logfusion}. Starting with region I, by \eqref{xmnice}, (\ref{regmod}), (\ref{logs}) and (\ref{regions}):
\begin{align*}
\widetilde{\mathcal{X}}_m & = \exp\left[ i\theta_m - \frac{1}{2\pi}  (\theta_e - \pi) \log\left[\frac{a}{|a|\zeta}\right] + \frac{1}{2} \log\left(1 - e^{-i\theta_e}\right)  \right] \hspace{5 mm} \text{in region I.}\\
\intertext{By \eqref{breaklog},}
& = \exp\left[ i\theta_m - \frac{1}{2\pi}  (\theta_e - \pi) \log\left[\frac{a}{|a|}\right] + \frac{\theta_e - \pi}{2\pi}\log \zeta + \frac{1}{2} \log\left(1 - e^{-i\theta_e}\right) \right]
\end{align*}
In region II, by our formulas above, we get
\begin{align*}
\widetilde{\mathcal{X}}_m & = \exp\left[i\theta_m - \frac{1}{2\pi}  (\theta_e - \pi) \log\left[\frac{a}{|a|\zeta}\right] - \frac{1}{2} \log \left(1 - e^{-i\theta_e}\right)  \right]\left(1 - e^{-i\theta_e}\right) \\
& =  \exp\left[i\theta_m -  \frac{1}{2\pi}  (\theta_e - \pi) \log\left[\frac{a}{|a|\zeta}\right] - \frac{1}{2} \log\left(1 - e^{-i\theta_e}\right) + \log\left(1 - e^{-i\theta_e}\right)  \right]\\
 & = \exp\left[ i\theta_m - \frac{1}{2\pi}  (\theta_e - \pi) \log\left[\frac{a}{|a|}\right] + \frac{\theta_e - \pi}{2\pi}\log \zeta + \frac{1}{2} \log\left(1 - e^{-i\theta_e}\right) \right]  \text{in region II.}
\end{align*}
Finally, in region III, and making use of \eqref{logs}, \eqref{regions}, \eqref{breaklog}:
\begin{align}
\widetilde{\mathcal{X}}_m & = \exp\left[i\theta_m - \frac{1}{2\pi}  (\theta_e - \pi) \log\left[\frac{a}{|a|\zeta}\right]  - \frac{1}{2} \log \left(1 - e^{-i\theta_e}\right)  \right]\left(1 - e^{i\theta_e}\right) \notag\\
& = \exp\left[ i\theta_m - \frac{1}{2\pi}  (\theta_e - \pi) \log\left[\frac{a}{|a|}\right] + \frac{\theta_e - \pi}{2\pi}\log \zeta - i(\theta_e - \pi) \notag \right. \\
&  \left. \hspace{14 mm} - \frac{1}{2} \log\left(1 - e^{-i\theta_e}\right) + \log\left(1 - e^{-i\theta_e}\right) + i(\theta_e - \pi)   \right] \notag \\
& = \exp\left[ i\theta_m - \frac{1}{2\pi}  (\theta_e - \pi) \log\left[\frac{a}{|a|}\right] + \frac{\theta_e - \pi}{2\pi}\log
 \zeta + \frac{1}{2} \log\left(1 - e^{-i\theta_e}\right)  \right] \label{xmany}.
\end{align}

 Observe that, throughout all these calculations, we only had to use the natural branch of the complex logarithm. In summary, (\ref{xmany}) works for any region in the $a$-plane, with a cut in the negative real axis.

This also suggest the following coordinate transformation
\begin{equation}\label{thetapr}
 \theta'_m = \theta_m + \frac{i(\theta_e - \pi)}{4\pi} \left( \log\frac{a}{\Lambda} - \log\frac{\bar{a}}{\overline{\Lambda}} \right) 
\end{equation}
Here $\Lambda$ is the same cutoff constant as in \cite{gaiotto}. Let $\varphi$ parametrize the phase of $a/|a|$. Then \eqref{thetapr} simplifies to
\begin{equation}\label{nicethm}
\theta'_m = \theta_m - \frac{(\theta_e - \pi)\varphi}{2\pi}
\end{equation}

On a coordinate patch around the singular fiber, $\theta'_m$ is single-valued. 
Thus, the above shows that we can glue to $\mathcal{M}'$ another $S^1$-bundle over $D \times (0,2\pi)$, for $D$ a small open disk around $a = 0$, and $\theta_e \in (0,2\pi)$. The $S^1$-fiber is parametrized by $\theta'_m$ and the transition function is given by \eqref{nicethm}, yielding a manifold $\mathcal{M}$. In this patch, we can extend $\widetilde{\mathcal{X}}_m$ to $a = 0$ as:

\begin{equation}\label{reblow}
\left. \widetilde{\mathcal{X}}_m\right|_{a = 0} = e^{i\theta'_m} \zeta^{\frac{\theta_e - \pi}{2\pi}} (1 - e^{-i\theta_e})^{\frac{1}{2}}
\end{equation}
where the branch of $\zeta^{\frac{\theta_e - \pi}{2\pi}}$ is determined by the natural branch of the logarithm in the $\zeta$ plane. Note that when $\theta_e = 0$, $\widetilde{\mathcal{X}}_m \equiv 0$ in \eqref{reblow} and by definition, $\mathcal{X}_e \equiv 1$. Since these two functions are Darboux coordinates for $\mathcal{M}$, the $S^1$ fibration over $D \times (0, 2\pi)$ we glued to $\mathcal{M}'$ to get $\mathcal{M}$ degenerates into a point when $\theta_e = 0$. 

Now consider the case that $\arg \zeta \in (-\pi, 0)$. Label the regions as one travels counterclockwise, starting with the region bounded by the cut and the $\ell_-$ (See Figure \ref{negarg}). We can do an analytic continuation similar to \eqref{regmod} starting in region I, but formulas \eqref{regions}, \eqref{breaklog} become now:

\begin{figure}[htbp]
	\centering
		\includegraphics[width=0.50\textwidth]{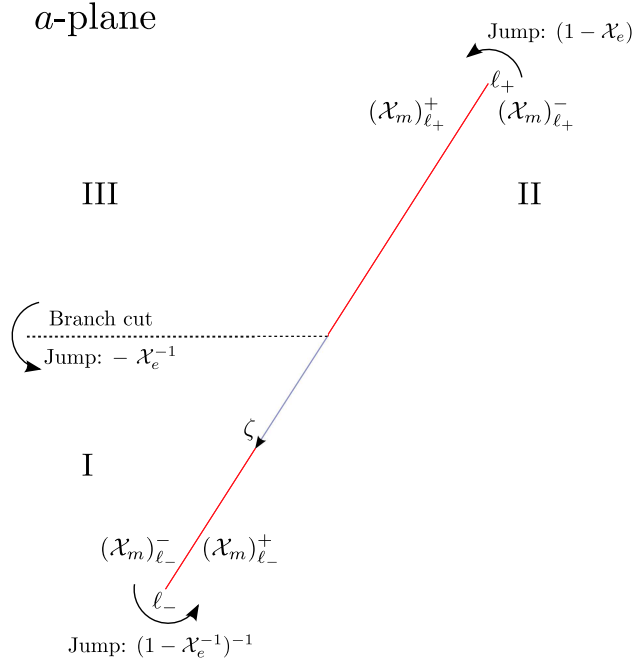}
	\caption{The three regions in the case $\arg \zeta < 0$.}
	\label{negarg}
\end{figure}

\begin{align*}
 \log\left[\frac{-a}{|a|\zeta}\right] & = \left\{ \begin{array}{ll}
     \log\left[\frac{a}{|a|\zeta}\right] - i\pi & \text{in region II}\\
	   \log\left[\frac{a}{|a|\zeta}\right] + i\pi & \text{in regions I and III}
	   \end{array} \right. \\
	   \log [a/\zeta] & = \left\{ \begin{array}{ll}
     \log a - \log \zeta & \text{in regions I and II}\\
	   \log a - \log \zeta - 2\pi i & \text{in region III}
	   \end{array} \right.
\end{align*}

By an argument entirely analogous to the case $\arg \zeta > 0$, we get again:
\begin{equation*}
\left. \widetilde{\mathcal{X}}_m\right|_{a = 0} = e^{i\theta'_m} \zeta^{\frac{\theta_e - \pi}{2\pi}} (1 - e^{-i\theta_e})^{\frac{1}{2}}
\end{equation*}

The case $\zeta$ real and positive is even simpler, as Figure \ref{zeroarg} shows. Here we have only two regions, and the jumps at the cut and the $\ell_+$ ray are combined, since these two lines are the same. Label the lower half-plane as region I and the upper half-plane as region II. Start an analytic continuation of $\mathcal{X}_m$ in region I as before, using the formulas:

\begin{figure}[htbp]
	\centering
		\includegraphics[width=0.50\textwidth]{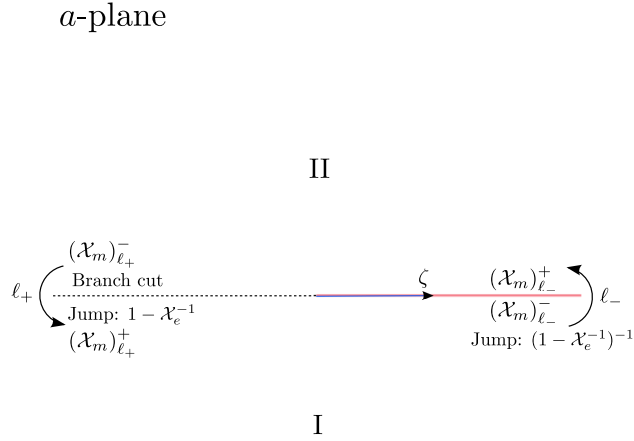}
	\caption{Only two regions in the case $\arg \zeta = 0$.}
	\label{zeroarg}
\end{figure}

\begin{align*}
 \log\left[\frac{-a}{|a|\zeta}\right] & = \left\{ \begin{array}{ll}
     \log\left[\frac{a}{|a|\zeta}\right] - i\pi & \text{in region II}\\
	   \log\left[\frac{a}{|a|\zeta}\right] + i\pi & \text{in region I}
	   \end{array} \right. \\
	   \log [a/\zeta] & = \log a - \log \zeta \hspace{4 mm} \text{in both regions}
\end{align*}

The result is equation \eqref{reblow} again. The case $\arg \zeta = \pi$ is entirely analogous to this and it yields the same formula, thus proving that \eqref{reblow} holds for all $\zeta$ and is independent of $\arg a$.
\end{proof}

\subsection{Alternative Riemann-Hilbert problem}\label{altrh}

We may obtain the function $\mathcal{X}_m$ (and consequently, the analytic extension $\widetilde{\mathcal{X}}_m$) at $a = 0$ through a slightly different formulation of the Riemann-Hilbert problem stated in \eqref{defxm}. Namely, instead of defining a jump of $\mathcal{X}_m$ at two opposite rays $\ell_+, \ell_-$, we combine these into a single jump at the line $\ell$ defined by $\ell_+$ and $\ell_-$, as in Figure \ref{onejump}. Note that because of the orientation of $\ell$ one of the previous jumps has to be reversed.

\begin{figure}[htbp]
	\centering
		\includegraphics[width=0.40\textwidth]{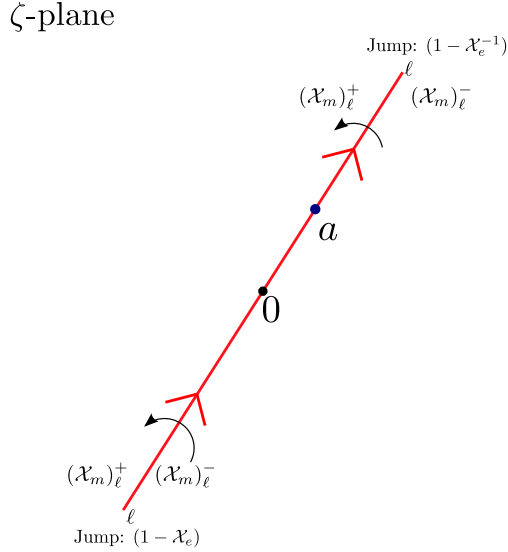}
	\caption{The reversed orientation on $\ell_+$ inverts the jump.}
	\label{onejump}
\end{figure}

For all values $a \neq 0$, $\mathcal{X}_e = \mathcal{X}_e^{\text{sf}}$ approaches 0 as $\zeta \to 0$ or $\zeta \to \infty$ along the $\ell$ ray due to the exponential decay in formula \eqref{xesf}. Thus, the jump function
\[ G(\zeta) := \left\{ \begin{array}{ll} 1-\mathcal{X}^{-1}_e & \text{ for $\zeta = t a, 0 \leq t \leq \infty$} \\ 
                                     1- \mathcal{X}_e & \text{ for $\zeta = t a, -\infty \leq t \leq 0$} \end{array} \right. \]                                     
is continuous on $\ell$ regarded as a closed contour on $\cpone$, and it approaches the identity transformation exponentially fast at the points $0$ and $\infty$.



The advantage of this reformulation of the Riemann-Hilbert problem is that it can be extended to the case $a = 0$ and we can obtain estimates on the solutions $\mathcal{X}_m$ even without an explicit formulation. If we fix $\arg a$ and let $|a| \to 0$ as before, the jump function $G(\zeta)$ approaches the constant jumps
\begin{equation}\label{jumpd}
\left. G(\zeta) \right|_{|a|=0} := \left\{ \begin{array}{ll} 1-e^{-i\theta_e} & \text{ for $\zeta = t a, 0 < t < \infty$} \\ 
                                     1- e^{i\theta_e} & \text{ for $\zeta = t a, -\infty < t < 0$} \end{array} \right.
\end{equation}
Thus, $\left. G(\zeta) \right|_{|a|=0}$ has two discontinuities at $0$ and $\infty$. If we denote by 
\[ \Delta_0 = \lim_{t \to 0^+} G(\zeta) - \lim_{t \to 0^-} G(\zeta), \qquad  \Delta_\infty = \lim_{t \to \infty^+} G(\zeta) - \lim_{t \to \infty^-} G(\zeta), \]
then, by \eqref{jumpd},
\begin{equation*}
\Delta_0 = -\Delta_\infty
\end{equation*}


Let $D^+$ be the region in $\cpone$ bounded by $\ell$ with the positive, counterclockwise orientation. Denote by $D^-$ the region where $\ell$ as a boundary has the negative orientation. We look for solutions of the homogeneous boundary problem
\begin{equation}\label{bcond}
X_m^+(\zeta) = G(\zeta) X_m^-(\zeta)
\end{equation}
with $G(\zeta)$ as in \eqref{jumpd}. This is Lemma 4.1 in \cite{rhprob}.

 The solutions $X_m^\pm$ obtained therein are related to $\mathcal{X}_m$ via $\mathcal{X}_m (\zeta) = \mathcal{X}^{\text{sf}}_m (\zeta) X_m (\zeta)$. Uniqueness of solutions of the homogeneous Riemann-Hilbert problem shows that these are the same functions (up to a constant factor) constructed in the previous section. Observe that the term $\zeta^{\frac{\theta_e - \pi}{2\pi}}$ appears naturally due to the nature of the discontinuity of the jump function at 0 and $\infty$. The analytic continuation around the point $a = 0$ and the gauge transformation $\theta_m \mapsto \theta'_m$ are still performed as before.

\subsection{Generalized Ooguri-Vafa coordinates}\label{genvafa}

We can generalize the previous extension to the case $Z_{\gamma_m} := \frac{1}{2\pi i}a \log a + f(a)$, where $f : \mathcal{B}' \to \cone$ is holomorphic and admits a holomorphic extension into $\mathcal{B}$. In particular,
\begin{equation}\label{newxmsf}
\mathcal{X}_m^\text{sf} = \exp \left( \frac{-iR}{2\zeta}a \log a + \frac{\pi R f(a)}{\zeta} + i \theta_m + \frac{i \zeta R}{2} \overline{a} \log \overline{a} + \pi R \zeta \overline{f(a)}\right)
\end{equation}
The value at the singular locus $f(0)$ does not have to be 0. All the other data remains the same. 

The first thing we observe is that $\mathcal{X}_e$ remains the same. Consequently, the corrections for the generalized $\mathcal{X}_m$ are as before. Using the change of coordinates as in \eqref{nicethm}, we can thus write
\begin{equation}\label{genov}
\left. \widetilde{\mathcal{X}}_m\right|_{a = 0} = \exp\left[ \frac{\pi R f(0)}{\zeta} + i\theta'_m + \pi R \zeta f(0) \right] \zeta^{\frac{\theta_e - \pi}{2\pi}} (1 - e^{-i\theta_e})^{\frac{1}{2}}
\end{equation}

\section{Extension of the Ooguri-Vafa metric}\label{gmetric}

\subsection{Classical Case}\label{clasmet}

\subsubsection{A $C^1$ extension of the coordinates}\label{c1ext}

In section \ref{clasov} we extended the fibered manifold $\mathcal{M}'$ to a manifold $\mathcal{M}$ with a degenerate fiber at $a = 0$ in $\mathcal{B}$. We also extended $\widetilde{\mathcal{X}}_m$ continuously to this bad fiber. Now we extend the metric by enlarging the holomorphic symplectic form $\varpi(\zeta)$. Recall that this is of the form
\[ \varpi(\zeta) = -\frac{1}{4\pi^2 R} \frac{d \mathcal{X}_e}{\mathcal{X}_e} \wedge  \frac{d\widetilde{\mathcal{X}}_m}{\widetilde{\mathcal{X}}_m} \]
Clearly there are no problems extending $d \log \mathcal{X}_e$, so it remains only to extend $d \log \widetilde{\mathcal{X}}_m$.

\begin{lemma}
Let $\widetilde{\mathcal{X}}_m$ denote the analytic continuation around $a = 0$ of the magnetic function, as in the last section. The 1-form
\begin{equation}\label{c1dlog}
d \log \widetilde{\mathcal{X}}_m = \frac{d \widetilde{\mathcal{X}}_m}{\widetilde{\mathcal{X}}_m},
\end{equation}
(where $d$ denotes the differential of a function on the torus fibration $\mathcal{M}'$ only) has an extension to $\mathcal{M}$
\end{lemma}
\begin{proof}
We proceed as in section \ref{clasov} and work in different regions in the $a$-plane (see Figure \ref{3reg}), starting with region I, where $\widetilde{\mathcal{X}}_m = \mathcal{X}_m$. Then observe that we can write the corrections on $\mathcal{X}_m$ as a complex number $\Upsilon_m(\zeta) \in (\mathcal{M}'_a)^{\cone}$ such that
\[ \mathcal{X}_m = \exp\left( \frac{-i R }{2\zeta}(a\log a - a) + i \Upsilon_m + \frac{i\zeta R}{2} (\overline{a} \log
\overline{a} - \overline{a} )\right). \]
Thus, by \eqref{c1dlog} and ignoring the $i$ factor, it suffices to obtain an extension of 
\begin{align}
& d\left[ \frac{- R }{2\zeta}(a\log a - a) +  \Upsilon_m + \frac{\zeta R}{2} (\overline{a} \log
\overline{a} - \overline{a} ) \right] \notag \\
& = \frac{ -R}{2\zeta} \log a \, da +  d \Upsilon_m + \frac{ \zeta R}{2} \log \overline{a} \, d\overline{a}. \label{3ext}
\end{align}
Using \eqref{defxm},
\begin{align*}
d \Upsilon_m  = d\theta_m & - \frac{1}{4\pi}\int_{\ell_+} \dfrac{d\zeta'}{\zeta'} \frac{\zeta'+\zeta}{\zeta'-\zeta} \frac{\mathcal{X}_e}{1-\mathcal{X}_e}\left( \frac{\pi R}{\zeta'} da +id\theta_e+ \pi R \zeta' d\overline{a}\right) \notag\\
& +\frac{1}{4\pi}\int_{\ell_-} \dfrac{d\zeta'}{\zeta'} \frac{\zeta'+\zeta}{\zeta'-\zeta} \frac{\mathcal{X}^{-1}_e}{1-\mathcal{X}^{-1}_e}\left( -\frac{\pi R}{\zeta'} da -id\theta_e - \pi R \zeta' d\overline{a}\right).
\end{align*}
We have to change our $\theta_m$ coordinate into $\theta'_m$ according to \eqref{nicethm} and differentiate to obtain:
\begin{align}
d \Upsilon_m  & =  
d\theta'_m - \frac{i(\theta_e - \pi)}{4\pi} \left( \frac{da}{a} - \frac{d\overline{a}}{\overline{a}}\right)+\frac{\arg a}{2\pi}d\theta_e \notag \\
& - \frac{1}{4\pi}\int_{\ell_+} \dfrac{d\zeta'}{\zeta'} \frac{\zeta'+\zeta}{\zeta'-\zeta} \frac{\mathcal{X}_e}{1-\mathcal{X}_e}\left( \frac{\pi R}{\zeta'} da +id\theta_e+ \pi R \zeta' d\overline{a}\right) \notag\\
& +\frac{1}{4\pi}\int_{\ell_-} \dfrac{d\zeta'}{\zeta'} \frac{\zeta'+\zeta}{\zeta'-\zeta} \frac{\mathcal{X}^{-1}_e}{1-\mathcal{X}^{-1}_e}\left( -\frac{\pi R}{\zeta'} da -id\theta_e - \pi R \zeta' d\overline{a}\right) \label{dups1}
\end{align}
Recall that, since we have introduced the change of coordinates $\theta_m \mapsto \theta'_m$, we are working on a patch on $\mathcal{M}$ that contains $a = 0$ with a degenerate fiber here. It then makes sense to ask if \eqref{3ext} extends to $a =0$. If this is true, then every independent 1-form extends individually. Let's consider the form involving $d\theta_e$ first. By \eqref{dups1}, this part consists of:
\begin{equation}\label{dthpart}
\frac{\arg a}{2\pi}d\theta_e  - \frac{i}{4\pi}\int_{\ell_+} \dfrac{d\zeta'}{\zeta'} \frac{\zeta'+\zeta}{\zeta'-\zeta} \frac{\mathcal{X}_e}{1-\mathcal{X}_e} d\theta_e - \frac{i}{4\pi}\int_{\ell_-} \dfrac{d\zeta'}{\zeta'} \frac{\zeta'+\zeta}{\zeta'-\zeta} \frac{\mathcal{X}^{-1}_e}{1-\mathcal{X}^{-1}_e} d\theta_e.
\end{equation}
We can use the exact same technique in section \ref{clasov} to find the limit of \eqref{dthpart} as $a \to 0$. Namely, split each integral into four parts, use the symmetry of $\dfrac{\mathcal{X}_e}{1- \mathcal{X}_e}$ between $0$ and $\infty$ to cancel two of these integrals and take the limit in the remaining ones. The result is:
\begin{align}
& \frac{\arg a}{2\pi} - \frac{ie^{i\theta_e}}{2\pi(1-e^{i\theta_e})}\log\left[ \frac{e^{i \arg a}}{\zeta} \right] - 
\frac{ie^{-i\theta_e}}{2\pi(1-e^{-i\theta_e})}\log\left[ \frac{-e^{i \arg a}}{\zeta} \right]  \notag \\
& = \frac{\arg a}{2\pi} - \frac{ie^{i\theta_e}}{2\pi(1-e^{i\theta_e})}\log\left[ \frac{e^{i \arg a}}{\zeta} \right] + 
\frac{i}{2\pi(1-e^{i\theta_e})}\log\left[ \frac{-e^{i \arg a}}{\zeta} \right] \label{dthe}
\end{align}
in region I (we omitted the $d\theta_e$ factor for simplicity). Making use of formulas \eqref{regions} and \eqref{breaklog}, we can simplify the above expression and get rid of the apparent dependence on $\arg a$ until finally getting:
\[ -\frac{i\log \zeta}{2\pi} - \frac{1}{2(1-e^{i\theta_e})}, \hspace{7 mm} \theta_e \neq 0. \]
In other regions of the $a$-plane we have to modify $\widetilde{\mathcal{X}}_m$ as in \eqref{regmod}. Nonetheless, by \eqref{regions} and \eqref{breaklog}, the result is the same and we conclude that at least the terms involving $d\theta_e$ have an extension to $a=0$ for $\theta_e \neq 0$.

Next we extend the terms involving $da$. By \eqref{3ext} and \eqref{dups1}, these are:
\begin{equation*}
\frac{ -R}{2\zeta} \log a \, da  - \frac{i(\theta_e - \pi)}{4\pi a} da - \frac{R}{4}\int_{\ell_+} \dfrac{d\zeta'}{(\zeta')^2} \frac{\zeta'+\zeta}{\zeta'-\zeta} \frac{\mathcal{X}_e}{1-\mathcal{X}_e} da - \frac{R}{4}\int_{\ell_-} \dfrac{d\zeta'}{(\zeta')^2} \frac{\zeta'+\zeta}{\zeta'-\zeta} \frac{\mathcal{X}^{-1}_e}{1-\mathcal{X}^{-1}_e} da 
\end{equation*}
In what follows, we ignore the $da$ part and focus on the coefficients for the extension. The partial fraction decomposition
\begin{equation}\label{parfrac}
\frac{\zeta'+\zeta}{(\zeta')^2(\zeta'-\zeta)} = \frac{2}{\zeta'(\zeta'-\zeta)} - \frac{1}{(\zeta')^2}
\end{equation}
splits each integral above into two parts. We will consider first the terms
\begin{equation}\label{dap1}
- \frac{i(\theta_e - \pi)}{4\pi a} + \frac{R}{4}\int_{\ell_+} \dfrac{d\zeta'}{(\zeta')^2}\frac{\mathcal{X}_e}{1-\mathcal{X}_e} + \frac{R}{4}\int_{\ell_-} \dfrac{d\zeta'}{(\zeta')^2}  \frac{\mathcal{X}^{-1}_e}{1-\mathcal{X}^{-1}_e}.
\end{equation}
Use the fact that $\mathcal{X}_e$ (resp. $\mathcal{X}^{-1}_e$) has norm less than 1 on $\ell_+$ (resp. $\ell_-$) and the uniform convergence of the geometric series on $\zeta'$ to write \eqref{dap1} as:
\begin{align*}
- \frac{i(\theta_e - \pi)}{4\pi a} + \frac{R}{4}\sum_{n=1}^\infty \left\{ \vphantom{\int_{\ell_+}} \right. & \int_{\ell_+}\frac{d\zeta'}{(\zeta')^2}\exp\left(
\frac{\pi R n a}{\zeta'} +i n \theta_e +\pi R n \zeta' \overline{a}\right) + \\
& \left.
\int_{\ell_-}\frac{d\zeta'}{(\zeta')^2}\exp\left(
\frac{-\pi R n a}{\zeta'} -i n \theta_e -\pi R n \zeta' \overline{a}\right)\right\},
\end{align*}
\begin{align*}
& = - \frac{i(\theta_e - \pi)}{4\pi a} + \left(\frac{R}{4}\right) \left( \frac{-2|a|}{a}\right)\sum_{n=1}^\infty \left( e^{in\theta_e} - e^{-in\theta_e}\right)K_1(2\pi R n |a|)\\
& = - \frac{i(\theta_e - \pi)}{4\pi a} - \frac{R|a|}{2a}\sum_{n=1}^\infty \left( e^{in\theta_e} - e^{-in\theta_e}\right)K_1(2\pi R n |a|).
\end{align*}
Since $K_1(x) \thicksim 1/x$, for $x$ real and $x \to 0$, we obtain, letting $a \to 0$:
\begin{align*}
& - \frac{i(\theta_e - \pi)}{4\pi a} - \frac{R|a|}{2a\cdot 2\pi R |a|} \sum_{n=1}^\infty \frac{\left( e^{in\theta_e} - e^{-in\theta_e}\right)}{n}\\
& =  - \frac{i(\theta_e - \pi)}{4\pi a} + \frac{1}{4\pi a}[\log(1-e^{i\theta_e})-\log(1-e^{-i\theta_e})]\\
\intertext{and by \eqref{logs},}
& =  - \frac{i(\theta_e - \pi)}{4\pi a} +\frac{i(\theta_e -\pi)}{4\pi a} = 0.
\end{align*}
Therefore this part of the $da$ terms extends trivially to 0 in the singular fiber.

It remains to extend the other terms involving $da$. Recall that by \eqref{parfrac}, these terms are (after getting rid of a factor of $-R/2$):
\begin{equation}\label{last3}
\frac{ \log a}{\zeta} + \int_{\ell_+} \frac{d\zeta'}{\zeta'(\zeta'-\zeta)} \frac{\mathcal{X}_e}{1-\mathcal{X}_e} +  \int_{\ell_-} \frac{d\zeta'}{\zeta'(\zeta'-\zeta)} \frac{\mathcal{X}^{-1}_e}{1-\mathcal{X}^{-1}_e}.
\end{equation}

We'll focus in the first integral in \eqref{last3}. As a starting point, we'll prove that as $a \to 0$, the limiting value of this integral is the same as the limit of
\begin{equation}\label{simpler}
 \int_{\ell_+}  \frac{d\zeta'}{\zeta'(\zeta'-\zeta)} \frac{\exp\left( \frac{\pi R a}{\zeta'} +i\theta_e \right)}{1-\exp\left( \frac{\pi R a}{\zeta'} +i\theta_e +\pi R \zeta' \overline{a}\right)}.
\end{equation}
It suffices to show that
\begin{equation}\label{lebes}
 \int_{\ell_+}  \frac{d\zeta'}{\zeta'(\zeta'-\zeta)} \frac{\exp\left( \frac{\pi R a}{\zeta'}  \right)}{1-\exp\left( \frac{\pi R a}{\zeta'} +i\theta_e +\pi R \zeta' \overline{a}\right)} [1-\exp(\pi R \zeta' \overline{a})] \to 0, \hspace{5 mm} \text{as $a \to 0$, $\theta_e \neq 0$}
\end{equation}
To see this, we can assume $|a| < 1$. Let $b = a/|a|$. Observe that in the $\ell_+$ ray, $|\exp(\pi Ra/\zeta')| < 1$, and since $\theta_e \neq 0$, we can bound \eqref{lebes} by
\[ \text{const} \int_{\ell_+} \frac{d\zeta'}{\zeta'(\zeta'-\zeta)} [1-\exp(\pi R \zeta' \overline{b})] < \infty. \]
Equation \eqref{lebes} now follows from Lebesgue Dominated Convergence and the fact that $1-\exp(\pi R \zeta' \overline{a}) \to 0$ as $a \to 0$. A similar application of Dominated Convergence allows us to reduce the problem to the extension of
\begin{equation}\label{secsimp}
\int_{\ell_+}  \frac{d\zeta'}{\zeta'(\zeta'-\zeta)} \frac{\exp\left( \frac{\pi R a}{\zeta'} +i\theta_e \right)}{1-\exp\left( \frac{\pi R a}{\zeta'} +i\theta_e \right)}.
\end{equation}
Introduce the real variable $s = -\pi R a / \zeta'$. We can write \eqref{secsimp} as:
\begin{align}
& e^{i\theta_e}\int_0^\infty \frac{ds}{s\left[ \frac{-\pi R a}{s} - \zeta \right]} \frac{e^{-s}}{1-e^{i\theta_e-s}} \notag\\
& = -\frac{1}{\zeta}\int_0^\infty \frac{ds}{s+\frac{\pi R a}{\zeta}} \cdot \frac{e^{-s}}{e^{-i\theta_e}-e^{-s}} \notag\\
& = \frac{1}{\zeta}\int_0^\infty \frac{ds}{s+\frac{\pi R a}{\zeta}} \cdot \frac{1}{1-e^{s-i\theta_e}} \label{doublez}
\end{align}
The integrand of \eqref{doublez} has a double zero at $\infty$, when $a \to 0$, so the only possible non-convergent part in the limit $a=0$ is the integral
\[ \frac{1}{\zeta}\int_0^1 \frac{ds}{s+\frac{\pi R a}{\zeta}} \cdot \frac{1}{1-e^{s-i\theta_e}}. \]
Since
\[ \int_0^1 \frac{ds}{s} \left[ \frac{1}{1-e^{s-i\theta_e}} - \frac{1}{1-e^{-i\theta_e}}\right] < \infty, \]
we can simplify this analysis even further and focus only on
\begin{align}
& \frac{1}{\zeta(1-e^{-i\theta_e})} \int_0^1 \frac{ds}{s+\frac{\pi R a}{\zeta}} \\
& = -\frac{\log (\pi R a /\zeta)}{\zeta(1-e^{-i\theta_e})}.
\end{align}
We can apply the same technique to obtain a limit for the second integral in \eqref{last3}. The result is
\[ -\frac{\log (-\pi R a /\zeta)}{\zeta(1-e^{i\theta_e})}, \]
which means that the possibly non-convergent terms in \eqref{last3} are:
\begin{equation}\label{cancel}
\frac{\log a}{\zeta} - \frac{\log a}{\zeta(1-e^{-i\theta_e})} - \frac{\log a}{\zeta(1-e^{i\theta_e})} = 0.
\end{equation}
Note that the corrections of $\mathcal{X}_m$ in other regions of the $a$-plane as in \eqref{regmod} depend only on $\mathcal{X}_e$, which clearly has a smooth extension to the singular fiber.

The extension of the $d\overline{a}$ part is performed in exactly the same way as with the $da$ forms. We conclude that the 1-form
\[ \frac{d\widetilde{\mathcal{X}}_m}{\widetilde{\mathcal{X}}_m} \]
has an extension to $\mathcal{M}$; more explicitly, to the fiber at $a=0$ in the classical Ooguri-Vafa case. This holds true also in the generalized Ooguri-Vafa case since here we simply add factors of the form $f'(a)da$ and it is assumed that $f(a)$ has a smooth extension to the singular fiber.
\end{proof}

In section \ref{sfiber}, we will reinterpret these extension of the derivatives of $\mathcal{X}_m$ if we regard the gauge transformation \eqref{nicethm} as a contour integral between symmetric contours. It will be then easier to see that the extension can be made smooth.

\subsubsection{Extension of the metric}\label{extmetric}

The results of the previous section already show the continuous extension of the holomorphic symplectic form
\[ \varpi(\zeta) = -\frac{1}{4\pi^2 R} \frac{d \mathcal{X}_e}{\mathcal{X}_e} \wedge  \frac{d\widetilde{\mathcal{X}}_m}{\widetilde{\mathcal{X}}_m} \]
to the limiting case $a = 0$, but we excluded the special case $\theta_e = 0$. Here we obtain $\varpi(\zeta)$ at the singular fiber with a different approach that will allow us to see that such an extension is smooth without testing the extension for each derivative. Although it was already known that $\mathcal{M}'$ extends to the hyperk\"{a}hler manifold $\mathcal{M}$ constructed here, this approach is new, as it gives an explicit construction of the metric as we will see. Furthermore, the Ooguri-Vafa model can be thought as an elementary model for which more complex integrable systems are modeled locally (see \S \ref{sfiber}).
\begin{theorem}
The holomorphic symplectic form $\varpi(\zeta)$ extends smoothly to $\mathcal{M}$. Near $a = 0$ and $\theta_e = 0$, the hyperk\"{a}hler metric $g$ looks like a constant multiple of the Taub-NUT metric $g_{\text{Taub-NUT}}$ plus some smooth corrections.
\end{theorem}
\begin{proof}
 By \cite{gaiotto}, near $a = 0$,
\[ \varpi(\zeta) = -\frac{1}{4\pi^2 R} \frac{d \mathcal{X}_e}{\mathcal{X}_e} \wedge \left[ id\theta_m + 2\pi i A + \pi i V 
\left(\frac{1}{\zeta}da - \zeta d\bar{a}\right)\right], \]
where
\[ A = \frac{1}{8\pi^2}\left( \log \frac{a}{\Lambda} - \log\frac{\bar{a}}{\overline{\Lambda}} \right)d\theta_e - \frac{R}{4\pi} \left( \frac{da}{a} - \frac{d\bar{a}}{\bar{a}}\right)\sum_{n \neq 0} (\text{sgn} \,n) e^{in\theta_e} |a|
   K_1(2\pi R|na|) \]
should be understood as a $U(1)$ connection over the open subset of $\cone \times S^1$ parametrized by $(a,\theta_e)$ and $V$ is given by Poisson re-summation as
\begin{equation}\label{potent}
V = \frac {R}{4\pi}\left[ \frac{1}{\sqrt{R^2|a|^2 + \frac{\theta_e^2}{4\pi^2}}} + \sum_{\substack{n = -\infty \\ n \neq 0}}^\infty \left( \frac{1}{\sqrt{R^2 |a|^2 + (\frac{\theta_e}{2\pi} + n)^2}} - \kappa_n \right) \right].
\end{equation}
Here $\kappa_n$ is a regularization constant introduced to make the sum convergent, even at $a = 0, \theta_e \neq 0$. The curvature $F$ of the unitary connection satisfies
\begin{equation}\label{curv}
dA = *dV.
\end{equation}
Consider now a gauge transformation $\theta_m \mapsto \theta_m + \alpha$ and its induced change in the connection $A \mapsto A' = A - d\alpha/2\pi$ (see \cite{gaiotto}). We have $id\theta'_m + 2\pi i A' = id\theta_m + id\alpha + 2\pi i A - id\alpha = id\theta_m + 2\pi i A$. Furthermore, for the particular gauge transformation in (\ref{thetapr}), at $a = 0$ and for $\theta_e \neq 0$:
\begin{align*}
A' & = A - \frac{d\alpha}{2\pi}\\
& = \frac{1}{8\pi^2}\left( \log \frac{a}{\Lambda} - \log\frac{\bar{a}}{\overline{\Lambda}} \right)d\theta_e - \frac{1}{8\pi^2} \left( \frac{da}{a} - \frac{d\bar{a}}{\bar{a}}\right) \left[ \sum_{n = 1}^\infty \frac{e^{in\theta_e}}{n} - \sum_{n = 1}^\infty \frac{e^{-in\theta_e}}{n} \right]\\
& - \frac{1}{8\pi^2}\left( \log \frac{a}{\Lambda} - \log\frac{\bar{a}}{\overline{\Lambda}} \right)d\theta_e - \frac{i(\theta_e - \pi)}{8\pi^2}\left( \frac{da}{a} - \frac{d\bar{a}}{\bar{a}}\right),\\
\intertext{(here we're using the fact that $K_1(x) \to 1/x$ as $x \to 0$)}
& = \frac{i(\theta_e - \pi)}{8\pi^2}\left( \frac{da}{a} - \frac{d\bar{a}}{\bar{a}}\right) - \frac{i(\theta_e - \pi)}{8\pi^2}\left( \frac{da}{a} - \frac{d\bar{a}}{\bar{a}}\right) = 0.\\
\intertext{since the above sums converge to $-\log(1 - e^{i\theta_e}) + \log(1 - e^{-i\theta_e}) = -i(\theta_e - \pi)$ for $\theta_e \neq 0$.}
\end{align*}

Writing $V_0$ (observe that this only depends on $\theta_e$) for the limit of $V$ as $a \to 0$, we get at $a = 0$
\begin{align*}
\varpi(\zeta) & = -\frac{1}{4\pi^2 R} \left( \frac{\pi R}{\zeta}da + id\theta_e + \pi R \zeta d\bar{a} \right) \wedge \left(
  id\theta'_m + \pi i V_0 \left( \frac{da}{\zeta} - \zeta d\bar{a}\right) \right)  \\
  & = \frac{1}{4\pi^2 R} d\theta_e \wedge d\theta'_m + \frac{iV_0}{2}da \wedge d\bar{a} -\frac{i}{4\pi \zeta}da \wedge d\theta'_m - \frac{V_0}{4\pi R\zeta}da \wedge d\theta_e  \\
  & - \frac{i\zeta}{4\pi} d\bar{a} \wedge d\theta'_m + \frac{V_0 \zeta}{4\pi R} d\bar{a} \wedge d\theta_e.
\end{align*}

This yields that, at the singular fiber,
\begin{align}
\omega_3 & = \frac{1}{4\pi^2 R} d\theta_e \wedge d\theta'_m + \frac{iV_0}{2}da \wedge d\bar{a} \label{symp3}\\
\omega_+ & = \frac{1}{2\pi} da \wedge \left( d\theta'_m - \frac{iV_0}{R}d\theta_e \right)\label{symp+}\\
\omega_- & = \frac{1}{2\pi} d\bar{a} \wedge \left( d\theta'_m + \frac{iV_0}{R}d\theta_e \right)\label{symp-}
\end{align}

From the last two equations we obtain that $d\theta'_m - iV_0/R d\theta_e$ and $d\theta'_m + iV_0/R d\theta_e$ are respectively (1,0) and (0,1) forms under the complex structure $J_3$. A $(1,0)$ vector field dual to the $(1,0)$ form above is then $\dfrac{1}{2}\left(\partial_{\theta'_m} + iR/V_0 \partial_{\theta_e}\right)$. In particular,
\[ J_3(\partial_{\theta'_m}) = -\frac{R}{V_0} \partial_{\theta_e}, \hspace{5 mm} J_3 \left(-\frac{R}{V_0}\partial_{\theta_e} \right) = -\partial_{\theta'_m}. \]
With this and (\ref{symp3}) we can reconstruct the metric at $a = 0$. Observe that
\begin{align*}
g(\partial_{\theta_e}, \partial_{\theta_e}) & = \omega_3(\partial_{\theta_e}, J_3(\partial_{\theta_e})) = \omega_3\left(\partial_{\theta_e}, \frac{V_0}{R}\partial_{\theta'_m}\right) = \frac{V_0}{4\pi^2 R^2} \\
g(\partial_{\theta'_m}, \partial_{\theta'_m}) & = \omega_3(\partial_{\theta'_m}, J_3(\partial_{\theta'_m})) = \omega_3\left(\partial_{\theta'_m}, -\frac{R}{V_0}\partial_{\theta_e}\right) = \frac{1}{4\pi^2 V_0}
\end{align*}

Consequently,
\[ g = \frac{1}{V_0} \left( \frac{d\theta'_m}{2\pi}\right)^2 + V_0 d\vec{x}^2, \]
where $a = x^1 + ix^2, \theta_e = 2\pi R x^3$. Since $V_0(\theta_e)$ is undefined for $\theta_e = 0$, we have to check that $g$ extends to this point. Let $(r,\vartheta, \phi)$ denote spherical coordinates for $\vec{x}$. The formula above is the natural extension of the metric given in \cite{gaiotto} for nonzero $a$:
\[ g = \frac{1}{V(\vec{x})} \left( \frac{d\theta'_m}{2\pi} + A'(\vec{x})\right)^2 + V(\vec{x}) d\vec{x}^2 \]
To see that this extends to $r =0$, we rewrite
\begin{align}
V & = \frac{R}{4\pi}\left[ \frac{1}{\sqrt{R^2 |a|^2 + \frac{\theta_e^2}{4\pi^2}}} + \sum_{n \neq 0} \left(
 \frac{1}{\sqrt{R^2 |a|^2 + (\frac{\theta_e}{2\pi} + n)^2}} - \kappa_n \right)\right] \notag\\
 & = \frac{1}{4\pi}\left[ \frac{1}{\sqrt{ |a|^2 + \frac{\theta_e^2}{4R^2 \pi^2}}} + R\sum_{n \neq 0} \left(
 \frac{1}{\sqrt{R^2 |a|^2 + (\frac{\theta_e}{2\pi} + n)^2}} - \kappa_n \right) \right] \notag\\
 & = \frac{1}{4\pi} \left( \frac{1}{r} + C(\vec{x}) \right), \label{vtaub}
\end{align}
where $C(\vec{x})$ is smooth and bounded in a neighborhood of the origin.

Similarly, we do Poisson re-summation for the unitary connection

\[ A' = - \frac{1}{4\pi} \left( \frac{da}{a} - \frac{d\bar{a}}{\bar{a}}\right) \left[ \frac{i(\theta_e - \pi)}{2\pi} + R \sum_{n \neq 0} (\text{sgn} \,n) e^{in\theta_e} |a| K_1(2\pi R|na|) \right]. \]

Using the fact that the inverse Fourier transform of $(\text{sgn }\xi)e^{i\theta_e \xi}|a|K_1(2\pi R|a\xi|)$ is
\[ \frac{i(\frac{\theta_e}{2\pi} + t)}{2R\sqrt{R^2|a|^2 + ( \frac{\theta_e}{2\pi} + t)^2}}, \]
we obtain
\begin{align}
A' & = - \frac{i}{8\pi} \left( \frac{da}{a} - \frac{d\bar{a}}{\bar{a}}\right)\sum_{n = -\infty}^\infty \left( \dfrac{ \frac{\theta_e}{2\pi} + n}{\sqrt{R^2 |a|^2 + (\frac{\theta_e}{2\pi} + n)^2}} - \kappa_n  \right) \notag\\
& = \frac{1}{4\pi}\left( \frac{da}{a} - \frac{d\bar{a}}{\bar{a}}\right)\left[-\frac{i\theta_e}{4\pi \sqrt{R^2 |a|^2 + \left(  
 \frac{\theta_e}{2\pi}\right)^2}} - \frac{i}{2}\sum_{n \neq 0} \left( \dfrac{ \frac{\theta_e}{2\pi} + n}{\sqrt{R^2 |a|^2 + (\frac{\theta_e}{2\pi} + n)^2}} - \kappa_n  \right)\right] \notag\\
 \intertext{since $d\phi = d\arg a = -id\log \dfrac{a}{|a|} = -\dfrac{i}{2}\left(\dfrac{da}{a} - \dfrac{d\bar{a}}{\bar{a}} \right)$ and $\cos \vartheta = \dfrac{x^3}{r}$, this simplifies to:}
 & = \frac{1}{4\pi}(\cos \vartheta + D(\vec{x}))d\phi. \label{ataub}
\end{align}
Here $\kappa_n$ is a regularization constant that makes the sum converge, and $D(\vec{x})$ is smooth and bounded in a neighborhood of $r = 0$. By (\ref{vtaub}) and (\ref{ataub}), it follows that near $r = 0$
\begin{align*}
g & = V^{-1}\left( \frac{d\theta'_m}{2\pi} + A' \right)^2 + Vd\vec{x}^2\\
& = 4\pi \left( \frac{1}{r} + C \right)^{-1} \left( \frac{d\theta'_m}{2\pi} + \frac{1}{4\pi}\cos \vartheta d\phi + D d\phi \right)^2 + \frac{1}{4\pi} \left( \frac{1}{r} + C \right) d\vec{x}^2\\
& = \frac{1}{4\pi}\left[ \left( \frac{1}{r} + C \right)^{-1} \left( 2d\theta'_m + \cos \vartheta d\phi + \tilde{D} d\phi \right)^2
 + \left( \frac{1}{r} + C \right) d\vec{x}^2 \right]\\
& = \frac{1}{4\pi} g_{\text{Taub-NUT}} + \text{smooth corrections}.
\end{align*}
This shows that our metric extends to $r = 0$ and finishes the construction of the singular fiber.
\end{proof}

\subsection{General case}\label{genextmtr}

Here we work with the assumption in subsection \ref{genvafa}. To distinguish this case to the previous one, we will denote by $\varpi_\text{old}, g_\text{old}$, etc. the forms obtained in the classical case.

Let $C := -i/2 + \pi f'(0)$ and let
\[ B_0 = V_0 + \frac{R \, \text{Im }C}{\pi}. \]
We will see that, to extend the holomorphic symplectic form $\varpi(\zeta)$  and consequently the hyperk\"{a}hler metric $g$ to $\mathcal{M}$, it is necessary to impose a restriction on the class of functions $f(a)$ on $\mathcal{B}$ for the generalized Ooguri-Vafa case.

\begin{theorem}
In the General Ooguri-Vafa case, the holomorphic symplectic form $\varpi(\zeta)$ and the hyperk\"{a}hler metric $g$ extend to $\mathcal{M}$, at least for the set of functions $f(a)$ as in \S \ref{genvafa} with $f'(0) > B_0$.
\end{theorem}
\begin{proof}
 By formula \eqref{newxmsf},
\begin{equation}\label{dlogxm}
d \log \mathcal{X}_m^\text{sf} = d \log \mathcal{X}_{m, \text{old}}^\text{sf} + \frac{R}{\zeta}\left( -\frac{i}{2} + \pi f'(a) \right)da + R\zeta \left( \frac{i}{2} + \pi \overline{f'(a)} \right)d\overline{a}
\end{equation}

Recall that the corrections of $\mathcal{X}_m$ are the same as the classical Ooguri-Vafa case. Thus, using \eqref{dlogxm}, at $a = 0$
\begin{equation*}
\varpi(\zeta) = \varpi_\text{old}(\zeta) + \frac{iR }{2\pi} \text{Im } C da \wedge d\overline{a} + \frac{i C}{4\pi^2 \zeta}
   da \wedge d\theta_e + \frac{i \zeta \overline{C}}{4\pi^2} d\overline{a} \wedge d\theta_e.
\end{equation*}

Decomposing $\varpi(\zeta) = -i/2\zeta \omega_+ + \omega_3 -i\zeta /2 \omega_-$, we obtain:
\begin{align}
\omega_3 & = \omega_{3, \text{old}} + \frac{i R}{2\pi} \text{Im } C da \wedge d\overline{a}, \label{omeg3}\\
\omega_+ & = \omega_{+, \text{old}} - \frac{C}{2\pi^2} da \wedge d\theta_e \label{newomp}\\
\omega_- & = \omega_{-, \text{old}} - \frac{\overline{C}}{2\pi^2} d\overline{a} \wedge d\theta_e \label{newomm}
\end{align}

By \eqref{newomp} and \eqref{newomm},
\[ d\theta'_m - \frac{i}{R}\left( V_0 - \frac{iRC}{\pi} \right)d\theta_e \hspace{6 mm} \text{and} \hspace{6 mm} d\theta'_m + \frac{i}{R}\left( V_0 + \frac{iR\overline{C}}{\pi} \right)d\theta_e \]
are, respectively, (1,0) and (0,1) forms. It's not hard to see that
\begin{align*}
\frac{-V_0 \pi -iR \overline{C}}{R\pi}\partial_{\theta'_m} &  - i\partial_{\theta_e} \\
\intertext{or, rearranging real parts,}
\left( -\frac{V_0}{R} -\frac{\text{Im }C}{\pi} \right) \partial_{\theta'_m} & -i \left( \frac{\text{Re }C}{\pi} \partial_{\theta'_m} + \partial_{\theta_e}\right)
\end{align*}
is a $(1,0)$ vector field. This allow us to obtain
\begin{align*}
J_3\left[ \left( -\frac{V_0}{R} -\frac{\text{Im }C}{\pi} \right) \partial_{\theta'_m} \right] & = \frac{\text{Re }C}{\pi} \partial_{\theta'_m} + \partial_{\theta_e}\\
J_3\left[\frac{\text{Re }C}{\pi} \partial_{\theta'_m} + \partial_{\theta_e} \right] & = \left( \frac{V_0}{R} +\frac{\text{Im }C}{\pi} \right) \partial_{\theta'_m}.
\end{align*}
By linearity,
\begin{align*}
J_3(\partial_{\theta'_m}) & = \text{const} \cdot \partial_{\theta'_m} - \frac{R\pi}{V_0 \pi + R\text{Im }C} \partial_{\theta_e}\\
J_3(\partial_{\theta_e}) & = \left( \frac{V_0 \pi + R\text{Im }C }{\pi R} + \frac{(\text{Re }C)^2 R}{\pi(V_0 \pi 
  + R\text{Im }C)} \right)\partial_{\theta'_m} + \text{const} \cdot \partial_{\theta_e}.
\end{align*}
With this we can compute
\begin{align*}
g(\partial_{\theta'_m}, \partial_{\theta'_m}) & = \omega_3(\partial_{\theta'_m}, J_3(\partial_{\theta'_m}))\\
& = \frac{1}{4\pi(V_0 \pi + R\text{Im }C)}\\
g(\partial_{\theta_e}, \partial_{\theta_e}) & = \omega_3(\partial_{\theta_e}, J_3(\partial_{\theta_e}))\\
& = \frac{V_0 \pi + R\text{Im }C}{4\pi^3 R^2} + \frac{(\text{Re }C)^2}{4\pi^3(V_0 \pi + R\text{Im }C)}\\
& = \frac{B_0}{4\pi^3 R^2} + \frac{(\text{Re }C)^2}{4\pi^3 B_0}
\end{align*}

We can see that, if $B_0 > 0$, the metric at $a = 0$ is
\begin{equation}
g = \frac{1}{B_0} \left( \frac{d\theta'_m}{2\pi}\right)^2 + B_0 d\vec{x}^2 + \left(\frac{R\cdot\text{Re }C}{\pi}\right)^2 \frac{dx_3^2}{B_0}.
\end{equation}
This metric can be extended to the point $\theta_e = 0$ ($r = 0$ in \S \ref{clasmet}) exactly as before, by writing $g$ as the Taub-NUT metric plus smooth corrections and observing that, since $\lim_{\theta_e \to 0} B_0 = \infty$,
\[ \lim_{\theta_e \to 0} \left(\frac{R\cdot\text{Re }C}{\pi}\right)^2 \frac{dx_3^2}{B_0} = 0. \]
\end{proof}

\section{The Pentagon case}\label{pent}

\subsection{Monodromy Data}\label{solut}

Now we will extend the results of the Ooguri-Vafa case to the general problem. We will start with the Pentagon example. This  example is presented in detail in \cite{notes}. By \cite{wkb}, this example represents the moduli space of Higgs bundles with gauge group $\text{SU}(2)$ over $\cpone$ with 1 irregular singularity at $z = \infty$.

Here $\mathcal{B} = \cone$ with discriminant locus a 2-point set, which we can assume is $\{-2,2\}$  in the complex plane. Thus $\mathcal{B}'$ is the twice-punctured plane. $\mathcal{B}$ is divided into two domains $\mathcal{B}_{\text{in}}$ and $\mathcal{B}_{\text{out}}$ by the locus
\[ W = \{u : Z(\Gamma_u) \text{ is contained in a line in } \cone \} \subset \mathcal{B} \]
See Figure \ref{walls}. Since $\mathcal{B}_{\text{in}}$ is simply connected $\Gamma$ can be trivialized over $\mathcal{B}_{\text{in}}$ by primitive cycles $\gamma_1, \gamma_2$, with $Z_{\gamma_1} = 0$ at $u = -2$, $Z_{\gamma_2} = 0$ at $u = 2$. We can choose them also so that $\left\langle \gamma_1, \gamma_2 \right\rangle = 1$. 

\begin{figure}[htbp]
	\centering
		\includegraphics[width=0.50\textwidth]{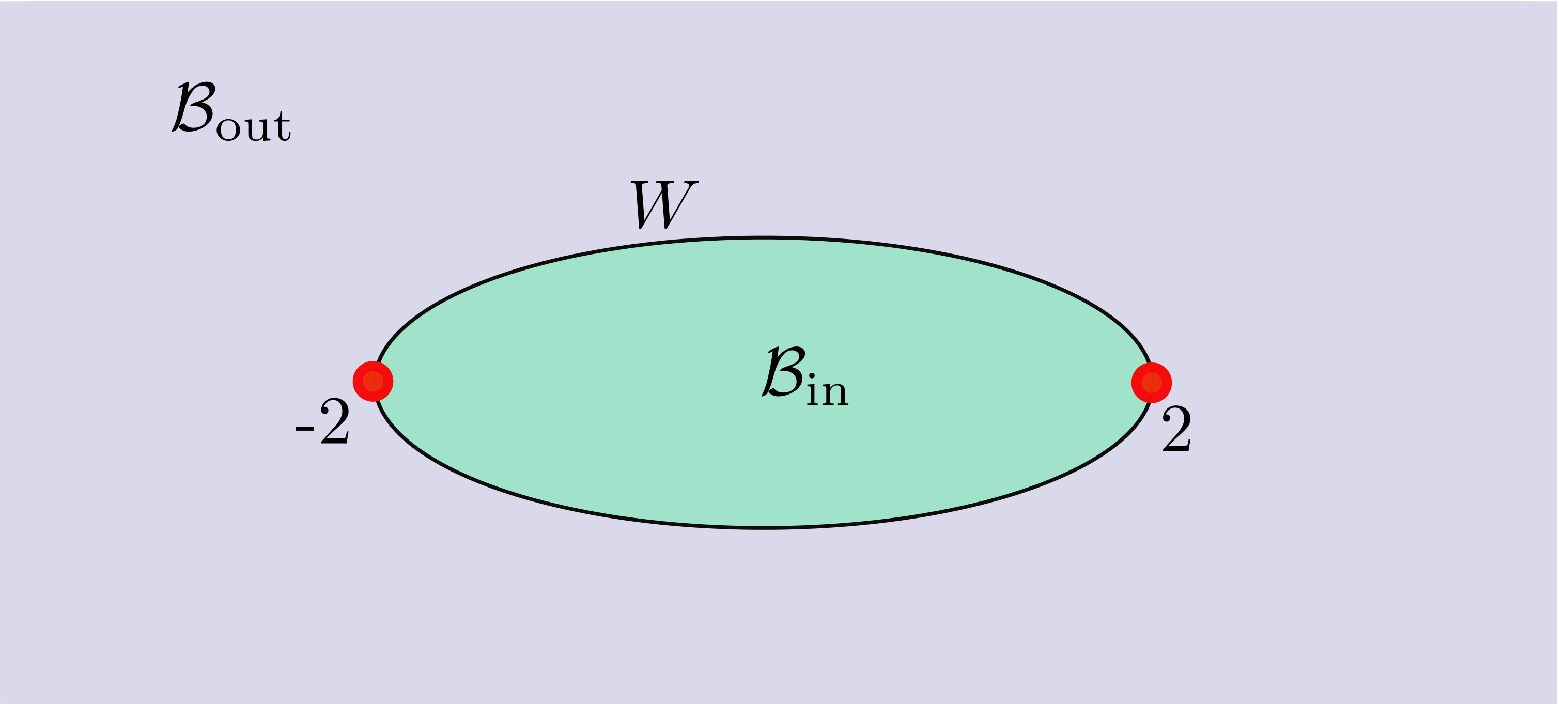}
	\caption{The wall $W$ in $\mathcal{B}$ for the Pentagon case}
	\label{walls}
\end{figure}

Take the set $\{\gamma_1, \gamma_2\}$. To compute its monodromy around infinity, take cuts at each point of $D = \{-2,2\}$ (see Figure \ref{moninf}) and move counterclockwise. By \eqref{piclf}, the jump of $\gamma_2$ when you cross the cut at $-2$ is of the form $\gamma_2 \mapsto \gamma_1 + \gamma_2$. As you return to the original place and cross the cut at $2$, the jump of $\gamma_1$ is of the type $\gamma_1 \mapsto \gamma_1 - \gamma_2$.

\begin{figure}[htbp]
	\centering
		\includegraphics[width=0.50\textwidth]{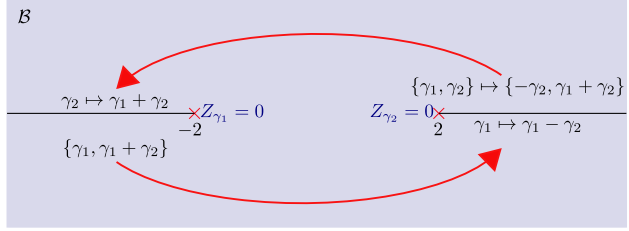}
	\caption{The monodromy around infinity of $\Gamma$}
	\label{moninf}
\end{figure}

Thus, around infinity, $\{\gamma_1, \gamma_2\}$ transforms into $\{-\gamma_2, \gamma_1 + \gamma_2\}$. The set $\{\gamma_1, \gamma_2, -\gamma_1, -\gamma_2, \gamma_1 + \gamma_2, -\gamma_1 - \gamma_2\}$ is therefore invariant under monodromy at infinity and it makes global sense to define

  \begin{align}
  \text{For $u \in \mathcal{B}_\text{in}$}, \hspace{5 mm} \Omega(\gamma; u) = & \left\{ \begin{array}{ll}
    1 & \text{for } \gamma \in \{ \gamma_1, \gamma_2, -\gamma_1, -\gamma_2\}\\
    0 & \text{otherwise}
  \end{array} \right. \notag\\
 \text{For $u \in \mathcal{B}_\text{out}$} , \hspace{5 mm} \Omega(\gamma; u) = & \left\{ \begin{array}{ll}
    1 & \text{for } \gamma \in \{ \gamma_1, \gamma_2, -\gamma_1, -\gamma_2, \gamma_1 + \gamma_2, -\gamma_1 - \gamma_2\}\\
    0 & \text{otherwise}
  \end{array}  \right. \label{omgpar}
  \end{align}

Let $\mathcal{M}'$ denote the torus fibration over $\mathcal{B}'$ constructed in \cite{notes}. Near $u=2$, we'll denote $\gamma_1$ by $\gamma_m$ and $\gamma_2$ by $\gamma_e$ (the labels will change for $u = - 2$). To shorten notation, we'll write $\ell_{e}, Z_e$, etc. instead of $\ell_{\gamma_e}, Z_{\gamma_e}$, etc. Let $\theta$ denote the vector of torus coordinates $(\theta_e, \theta_m)$. With the change of variables $a := Z_e(u)$ we can assume, without loss of generality, that the bad fiber is at $a = 0$ and
\begin{equation}\label{zmnotz}
\lim_{a \to 0} Z_m(a) = c \neq 0.
\end{equation}
Let $T$ denote the complex torus fibration over $\mathcal{M}'$ constructed in \cite{gaiotto}. By the definition of $\Omega(\gamma; a)$, the functions $(\mathcal{X}_e, \mathcal{X}_m)$ both receive corrections. Recall that by \eqref{recurs}, for each $\nu \in \mathbb{N}$, we get a function $\mathcal{X}_\gamma^{(\nu)}$, which is the $\nu$-th iteration of the function $\mathcal{X}_\gamma$. We can write
\[ \mathcal{X}_\gamma^{(\nu)}(a, \zeta, \theta) = \mathcal{X}_\gamma^{\text{sf}}(a, \zeta, \theta)C_{\gamma}^{(\nu)}(a, \zeta, \theta). \]
It will be convenient to rewrite the above equation as in \cite[C.17]{gaiotto}. For that, let $\upsnu$ be the map from $\mathcal{M}_a$ to its complexification $\mathcal{M}_a^{\cone}$ such that
\begin{equation}\label{upsi}
 \mathcal{X}_\gamma^{(\nu)}(a, \zeta, \theta) = \mathcal{X}_\gamma^{\text{sf}}(a, \zeta, \upsnu).
\end{equation}

We'll do a modification in the construction of \cite{gaiotto} as follows: We'll use the term ``BPS ray'' for each ray $\{\ell_\gamma : \Omega(\gamma,a) \neq 0 \}$ as in \cite{gaiotto}. This terminology comes from Physics. In the language of Riemann-Hilbert problems, these are known as ``anti-Stokes'' rays. That is, they represent the contour $\Sigma$ where a function has prescribed discontinuities.

The problem is local on $\mathcal{B}$, so instead of defining a Riemann-Hilbert problem using the BPS rays $\ell_\gamma$, we will cover $\mathcal{B}'$ with open sets $\{U_\alpha : \alpha \in \Delta \}$ such that for each $\alpha$, $\overline{U_\alpha}$ is compact, $\overline{U_\alpha} \subset V_\alpha$, with $V_\alpha$ open and $\left. \mathcal{M}' \right|_{V_\alpha}$ a trivial fibration. For any ray $r$ in the $\zeta$-plane, define $\mathbb{H}_r$ as the half-plane of vectors making an acute angle with $r$. Assume that there is a pair of rays $r, -r$ such that for all $a \in U_\alpha$, half of the rays  lie inside $\mathbb{H}_r$ and the other half lie in $\mathbb{H}_{-r}$. We call such rays \textit{admissible rays}. If $U_\alpha$ is small enough, there exists admissible rays for such a neighborhood. We are allowing the case that $r$ is a BPS ray $\ell_\gamma$, as long as it satisfies the above condition. As $a$ varies in $U_\alpha$, some BPS rays (or anti-Stokes rays, in RH terminology) converge into a single ray (wall-crossing phenomenon) (see Figures \ref{3rays} and \ref{2rays}). 

\begin{figure}[htbp]
	\centering
		\includegraphics[width=0.40\textwidth]{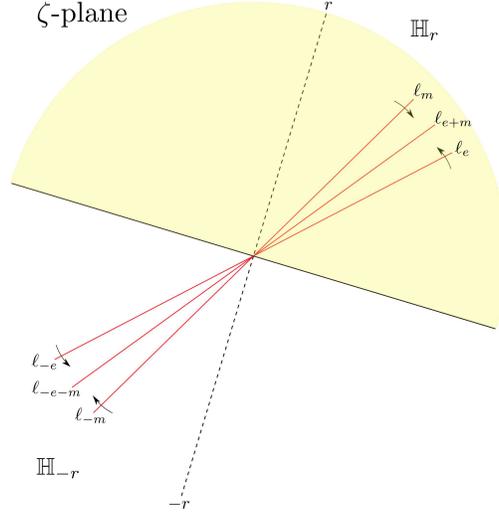}
	\caption{3 anti-Stokes rays before hitting the wall}
	\label{3rays}
\end{figure}

\begin{figure}[htbp]
	\centering
		\includegraphics[width=0.40\textwidth]{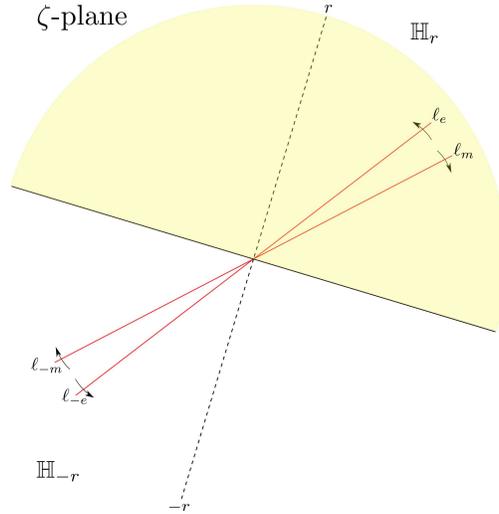}
	\caption{At the other side of the wall there are only 2 anti-Stokes rays}
	\label{2rays}
\end{figure}

 For $\gamma \in \Gamma$, we define $\gamma > 0$ (resp. $\gamma < 0$) as $\ell_\gamma \in \mathbb{H}_r$ (resp. $\ell_\gamma \in \mathbb{H}_{-r}$). Our Riemann-Hilbert problem will have only two anti-Stokes rays, namely $r$ and $-r$. The specific discontinuities at the anti-Stokes rays for the function we're trying to obtain are called \textit{Stokes factors} (see \cite{boalch}). In \eqref{invjmp}, the Stokes factor was given by $S^{-1}_\ell$.
 
 In this case, the Stokes factors are the concatenation of all the Stokes factors $S^{-1}_\ell$ in \eqref{stkfac} in the counterclockwise direction:
\begin{align*}
S_+ & = \prod^\ccwarrow_{\gamma > 0}{\mathcal{K}^{\Omega(\gamma; a)}_\gamma}\\
S_- & = \prod^\ccwarrow_{\gamma < 0}{\mathcal{K}^{\Omega(\gamma; a)}_\gamma}
\end{align*}

We will denote the solutions of this Riemann-Hilbert problem by $\mathcal{Y}$. As in \eqref{upsi}, we can write $\mathcal{Y}$ as
\begin{equation}\label{thet}
 \mathcal{Y}_\gamma(a, \zeta, \theta) = \mathcal{X}_\gamma^{\text{sf}}(a, \zeta, \Theta),
\end{equation}
for $\Theta : \mathcal{M}_a \to \mathcal{M}_a^\cone$.

A different choice of admissible pairs $r', -r'$ gives an equivalent Riemann-Hilbert problem, where the two solutions $\mathcal{Y}, \mathcal{Y}'$ differ only for $\zeta$ in the sector defined by the rays $r,r'$, and one can be obtained from the other by analytic continuation. 

In the case of the Pentagon, we have two types of wall-crossing phenomenon. Namely, as $a$ varies, $\ell_e$ moves in the $\zeta$-plane until it coincides with the $\ell_m$ ray for some value of $a$ in the wall of marginal stability (Fig. \ref{3rays} and \ref{2rays}). We'll call this type I of wall-crossing. In this case we have the Pentagon identity
\begin{equation}\label{pentid}
 \mathcal{K}_e \mathcal{K}_m = \mathcal{K}_m \mathcal{K}_{e+m} \mathcal{K}_e,
\end{equation}
As $a$ goes around 0, the $\ell_e$ ray will then intersect with the $\ell_{-m}$ ray now. Because of the monodromy $\gamma_m \mapsto \gamma_{-e+m}$ around 0, $\ell_m$ becomes $\ell_{-e+m}$. This second type (type II) of wall-crossing is illustrated in Fig. \ref{2rays2} and \ref{3rays2}.

\begin{figure}[htbp]
	\centering
		\includegraphics[width=0.40\textwidth]{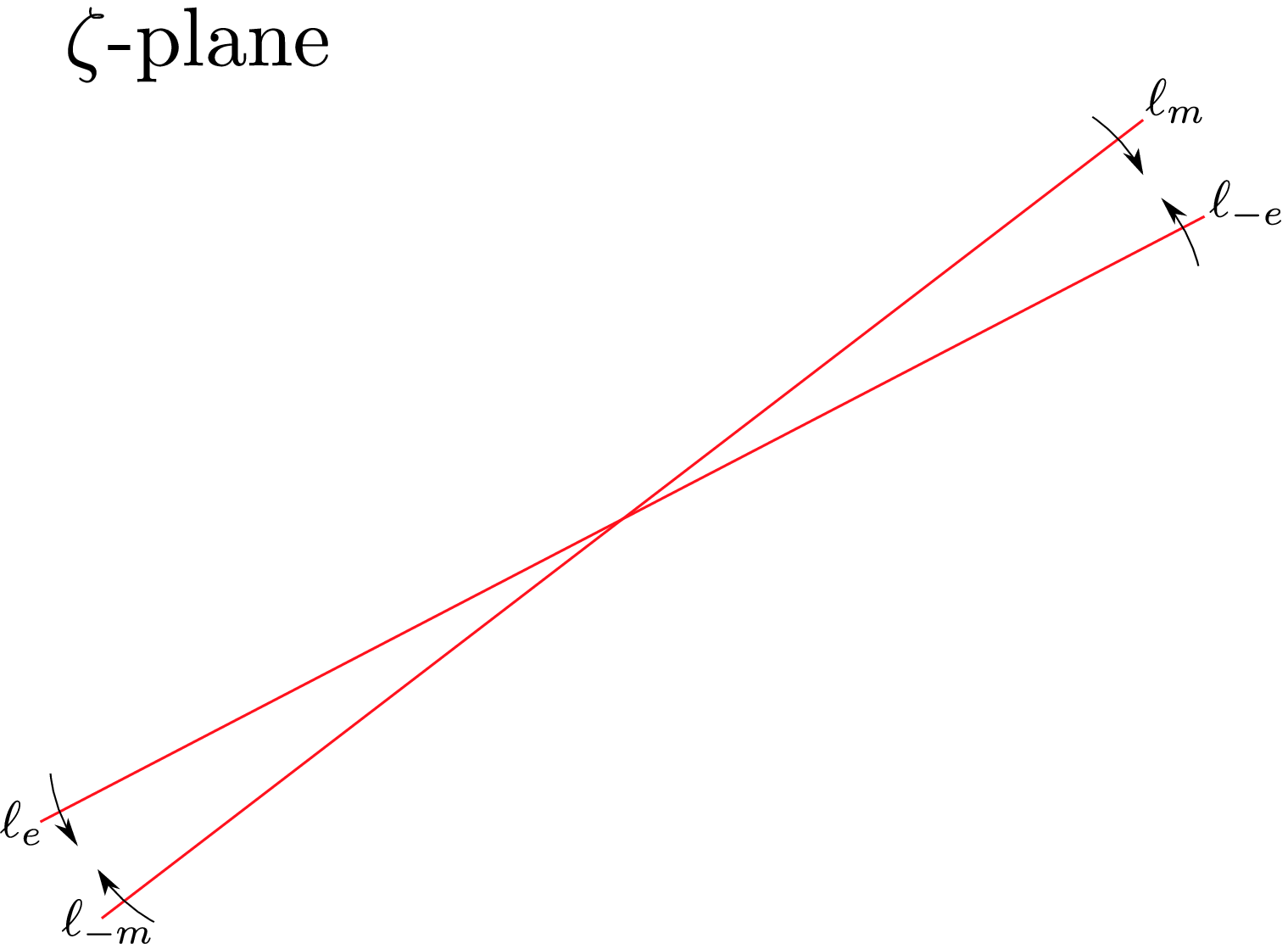}
	\caption{2 anti-Stokes rays before hitting the wall}
	\label{2rays2}
\end{figure}

\begin{figure}[htbp]
	\centering
		\includegraphics[width=0.40\textwidth]{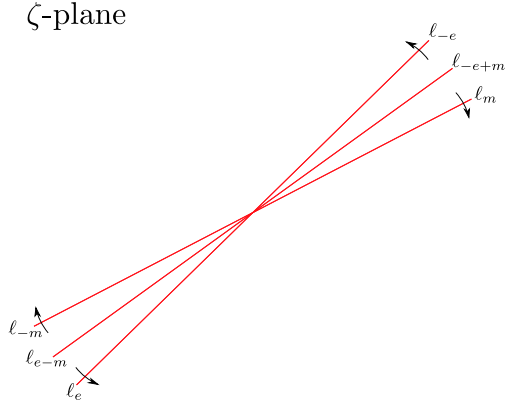}
	\caption{At the other side of the wall there are now 3 anti-Stokes rays}
	\label{3rays2}
\end{figure}

This gives a second Pentagon identity
\[ \mathcal{K}_{-e} \mathcal{K}_m = \mathcal{K}_m \mathcal{K}_{-e+m} \mathcal{K}_{-e} \]

In any case, the Stokes factors above remain the same even if $a$ is in the wall of marginal stability. The way we defined $S_+, S_-$ makes this true for the general case also.

Specifically, in the Pentagon the two Stokes factors for the first type of wall-crossing are given by the maps:
\begin{align}
\left. \begin{array}{ll}
          \mathcal{Y}_m & \mapsto \mathcal{Y}_m(1-\mathcal{Y}_e(1-\mathcal{Y}_m))^{-1} \\
          \mathcal{Y}_e & \mapsto \mathcal{Y}_e(1-\mathcal{Y}_m)
          \end{array}  \right\} & S_+ \label{newj1}\\
\intertext{and, similarly}
\left. \begin{array}{ll}
          \mathcal{Y}_m & \mapsto \mathcal{Y}_m(1-\mathcal{Y}^{-1}_e(1-\mathcal{Y}^{-1}_m))\\
          \mathcal{Y}_e & \mapsto \mathcal{Y}_e(1-\mathcal{Y}^{-1}_m)^{-1}
          \end{array} \right\} & S_- \label{newj2}
\end{align}

\noindent For the second type:

\begin{align}
\left. \begin{array}{ll}
          \mathcal{Y}_m & \mapsto \mathcal{Y}_m(1-\mathcal{Y}^{-1}_e) \\
          \mathcal{Y}_e & \mapsto \mathcal{Y}_e(1-\mathcal{Y}_m(1-\mathcal{Y}^{-1}_e))
          \end{array}  \right\} & S_+ \label{2newj1}\\
\left. \begin{array}{ll}
          \mathcal{Y}_m & \mapsto \mathcal{Y}_m(1-\mathcal{Y}_e)^{-1}\\
          \mathcal{Y}_e & \mapsto \mathcal{Y}_e(1-\mathcal{Y}^{-1}_m(1-\mathcal{Y}_e))^{-1}
          \end{array} \right\} & S_- \label{2newj2}
\end{align}

\subsection{Solutions}

In \cite{rhprob} we prove the following theorem (in fact, a more general version is proven).

\begin{theorem}\label{yfunctions}
There exist functions $ \mathcal{Y}_m(a, \zeta, \theta_e, \theta_m),  \mathcal{Y}_e(a, \zeta, \theta_e, \theta_m)$ defined for $a \neq 0$,  smooth on $a$, $\theta_e$ and $\theta_m$. The functions are sectionally analytic on $\zeta$ and obey the jump condition
\[ \begin{array}{rll}
	\mathcal{Y}^+ & = S_+ \mathcal{Y}^-, & \qquad \text{along $r$} \\
	\mathcal{Y}^+ & = S_{-}  \mathcal{Y}^-, & \qquad \text{along $-r$}
	\end{array} \]
Moreover, $\mathcal{Y}_m,  \mathcal{Y}_e$ obey the reality condition \eqref{realcond} and the asymptotic condition \ref{asymptotic}.
\end{theorem}

\begin{remark}
Our construction used integrals along a fixed admissible pair $r,-r$ and our Stokes factors are concatenation of the Stokes factors in \cite{gaiotto}. Thus, the coefficients $f^{\gamma'}$ are different here, but they are still obtained by power series expansion of the explicit Stokes factor. In particular, it may not be possible to express
\[ f^{\gamma'} = c_{\gamma'} \gamma' \]
for some constant $c_{\gamma'}$. For instance, in the pentagon, wall-crossing type I, we have, for $0\leq j\leq i$ and $\gamma' = \gamma_{ie +jm}$:
\[ f^{\gamma'} = \frac{(-1)^{j}\binom{i}{j}}{i^2} \gamma_{ie}. \]
Because of this, we didn't use the Cauchy-Schwarz property of the norm in $\Gamma$ in the estimates above as in \cite{gaiotto}. Nevertheless, the tameness condition on the $\Omega(\gamma',a)$ invariants still give us the desired contraction.
\end{remark}

Observe that, since we used admissible rays, the Stokes matrices don't change at the walls of marginal stability and we were able to treat both sides of the wall indistinctly. Thus, the functions $\mathcal{Y}$ in Theorem \ref{yfunctions} are smooth across the wall.

Let's reintroduce the solutions in \cite{gaiotto}. Denote by $\mathcal{X}_e, \mathcal{X}_m$ the solutions to the Riemann-Hilbert problem with jumps of the form $S_\ell^{-1}$ at each BPS ray $\ell$ with the same asymptotics and reality condition as $\mathcal{Y}_e, \mathcal{Y}_m$. In fact, we can see that the functions $\mathcal{Y}$ are the analytic continuation of $\mathcal{X}$ up until the admissible rays $r, -r$. 

In a patch $U_\alpha \subset \mathcal{B}'$ containing the wall of marginal stability, define the admissible ray $r$ as the ray where $\ell_e, \ell_m$ (or $\ell_e, \ell_{-m}$) collide. Since one is the analytic continuation of the other, $\mathcal{X}$ and $\mathcal{Y}$ differ only in a small sector in the $\zeta$-plane bounded by the $\ell_e, \ell_m$ ($\ell_e, \ell_{-m}$) rays, for $a$ not in the wall. As $a$ approaches the wall, such a sector converges to the single admissible ray $r$. Thus, away from the ray where the two BPS rays collide, the solutions $\mathcal{X}$ in \cite{gaiotto} are continuous in $a$.

\section{Extension to the singular fibers}\label{sfiber}

In this paper we will only consider the Pentagon example and in this section we will extend the Darboux coordinates $\mathcal{X}_e, \mathcal{X}_m$ obtained above to the singular locus $D \subset \mathcal{B}$ where one of the charges $Z_\gamma$ approaches zero.

Let $u$ be a coordinate for $\mathcal{B} = \cone$. We can assume that the two bad fibers of $\mathcal{M}$ are at $-2,2$ in the complex $u$-plane. For almost all $\zeta \in \cpone$, the BPS rays converge in a point of the wall of marginal stability away from any bad fiber:

\begin{figure}[htbp]
	\centering
		\includegraphics[width=0.50\textwidth]{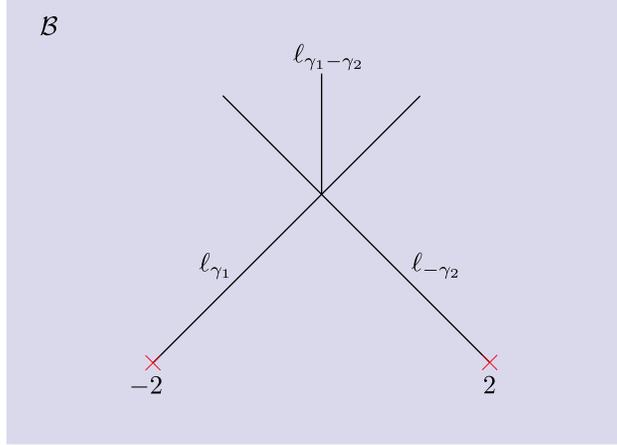}
	\caption{For general $\zeta$, there is only 1 pair of rays at each fiber}
	\label{raysinb}
\end{figure}

It is assumed that $\lim_{u \to 2} Z_{\gamma_1}$ exists and it is nonzero. If we denote this limit by $c = |c|e^{i\phi}$, then for $\zeta$ such that $\arg \zeta \to \phi + \pi$, the ray $\ell_{\gamma_1}$ emerging from -2 approaches the other singular point $u = 2$ (see Figure \ref{oneside}).

\begin{figure}[htbp]
	\centering
		\includegraphics[width=0.60\textwidth]{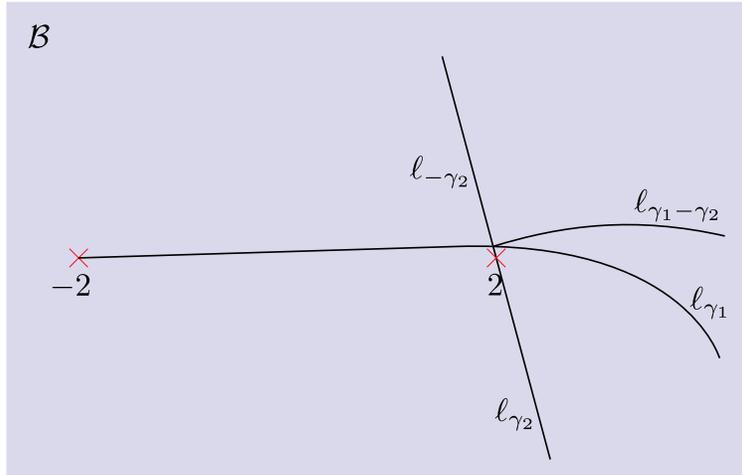}
	\caption{The BPS rays in $\mathcal{B}$ nearly coalesce at the singular locus}
	\label{oneside}
\end{figure}

When $\arg \zeta = \phi + \pi$, the locus $\{ u : Z_{\gamma}(u)/\zeta \in \rone_-\}$, for some $\gamma$ such that $\Omega(\gamma;u) \neq 0$ crosses $u = 2$. See Figure \ref{otherside}.

\begin{figure}[htbp]
	\centering
		\includegraphics[width=0.60\textwidth]{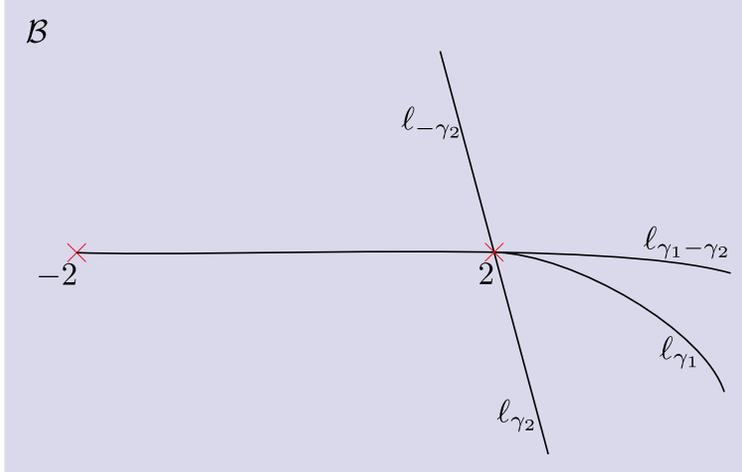}
	\caption{For $\zeta$ in a special ray, the rays intersect $u = 2$}
	\label{otherside}
\end{figure}

As $\zeta$ keeps changing, the rays leave the singular locus, but near $u = 2$, the tags change due to the monodromy of $\gamma_1$ around $u=2$. Despite this change of labels, near $u = 2$ only the rays $\ell_{\gamma_2}, \ell_{-\gamma_2}$ pass through this singular point. See Figure \ref{finalside}

\begin{figure}[htbp]
	\centering
		\includegraphics[width=0.60\textwidth]{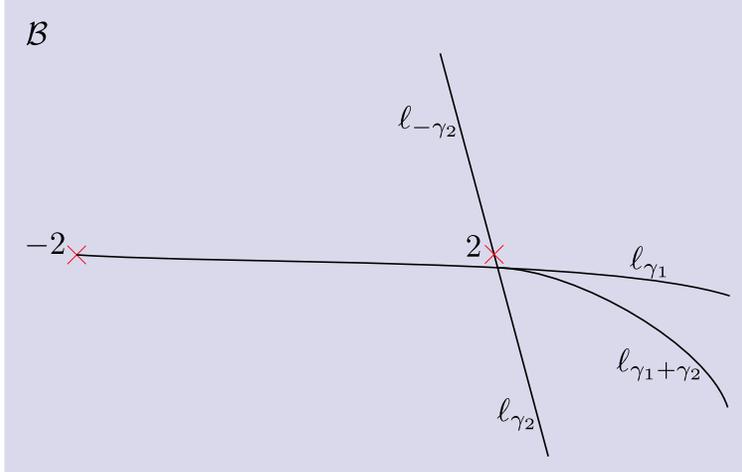}
	\caption{After the critical value of $\zeta$, the rays leave $u = 2$ and their tags change}
	\label{finalside}
\end{figure}

In the general case of Figures \ref{raysinb}, \ref{oneside} or \ref{finalside}, the picture near $u = 2$ is like in the Ooguri-Vafa case, Figure \ref{3reg}.

In any case, because of the specific values of the invariants $\Omega$, it is possible to analytically extend the function $\mathcal{X}_{\gamma_1}$ around $u = 2$. The global jump coming from the rays $\ell_{\gamma_2}, \ell_{-\gamma_2}$ is the opposite of the global monodromy coming from the Picard-Lefschetz monodromy of $\gamma_1 \mapsto \gamma_1 - \gamma_2$ (see \eqref{piclf}). Thus, it is possible to obtain a function $\widetilde{\mathcal{X}}_{\gamma_1}$ analytic on a punctured disk on $\mathcal{B}'$ near $u = 2$ extending $\mathcal{X}_{\gamma_1}$.

From this point on, we use the original formulation of the Riemann-Hilbert problem using BPS rays as in \cite{gaiotto}. We also use $a = Z_{\gamma_2}(u)$ to coordinatize a disk near $u = 2$, and we label $\{\gamma_1, \gamma_2\}$ as $\{\gamma_m, \gamma_e\}$ as in the Ooguri-Vafa case. Recall that, to shorten notation, we write $\ell_e, \mathcal{X}_e$, etc. instead of $\ell_{\gamma_e}, \mathcal{X}_{\gamma_e}$, etc.

By our work in the previous section, solutions $\mathcal{X}_\gamma$ (or, taking logs, $\Upsilon_\gamma$) to the Riemann-Hilbert problem are continuous at the wall of marginal stability for all $\zeta$ except those in the ray $\ell_m = Z_{m}/\zeta \in \rone_- = \ell_e$ (to be expected by the definition of the RH problem). We want to extend our solutions to the bad fiber located at $a=0$. We'll see that to achieve this, it is necessary to introduce new $\theta$ coordinates.


For convenience, we rewrite the integral formulas for the Pentagon in terms of $\Upsilon$ as in \cite{notes}. We will only write the part in $\mathcal{B}_\text{in}$, the $\bout$ part is similar.
\begin{align}
\Upsilon_e(a,\zeta) & = \theta_e -
\frac{1}{4\pi}\left\{ \int_{\ell_m} \frac{d\zeta'}{\zeta'} \frac{\zeta' + \zeta}{\zeta' -\zeta}\log\left[ 1 - \mathcal{X}_m^{\text{sf}}(a,\zeta', \Upsilon_m) \right] - \int_{\ell_{-m}} \frac{d\zeta'}{\zeta'} \frac{\zeta' + \zeta}{\zeta' -\zeta}\log\left[ 1 - \mathcal{X}_{-m}^{\text{sf}}(a,\zeta', \Upsilon_{-m})\right] \right\}, \label{upsefor}\\
\Upsilon_m(a,\zeta) & = \theta_m +
\frac{1}{4\pi}\left\{ \int_{\ell_e} \frac{d\zeta'}{\zeta'} \frac{\zeta' + \zeta}{\zeta' -\zeta}\log\left[ 1 - \mathcal{X}_e^{\text{sf}}(a,\zeta', \Upsilon_e) \right]  - \int_{\ell_{-e}} \frac{d\zeta'}{\zeta'} \frac{\zeta' + \zeta}{\zeta' -\zeta}\log\left[ 1 - \mathcal{X}_{-e}^{\text{sf}}(a,\zeta', \Upsilon_{-e}) \right] \right\} \label{upsmfor}
\end{align}


We can focus only on the integrals above, so write $\Upsilon_\gamma(a,\zeta) = \theta_\gamma + \dfrac{1}{4\pi} \Phi_\gamma(a,\zeta)$, for $\gamma \in \{\gamma_m, \gamma_e\}$. To obtain the right gauge transformation of the torus coordinates $\theta$, we'll split the integrals above into four parts and then we'll show that two of them define the right change of coordinates (in $\bein$, and a similar transformation for $\bout$) that simplify the integrals and allow an extension to the singular fiber.

By Theorem \ref{yfunctions}, both $\Upsilon_m, \Upsilon_e$ satisfy the ``reality condition'', which expresses a symmetry in the behavior of the complexified coordinates $\Upsilon$:
 \begin{equation}\label{corresp}
 \overline{\Upsilon_\gamma(a, \zeta)} = \Upsilon_\gamma \left(a, -1/\overline{\zeta}\right), \qquad a \neq 0
 \end{equation} 
 If we write as $\Upsilon_0$ (resp. $\Upsilon_\infty$) the asymptotic of this function as $\zeta \to 0$ (resp. $\zeta \to \infty$) so that
 \[ \Upsilon_0 = \theta + \frac{1}{4\pi} \Phi_0,  \]
 for a suitable correction $\Phi_0$. A similar equation holds for the asymptotic as $\zeta \to \infty$. By the asymptotic condition \ref{asymptotic}, $\Phi_0$ is imaginary. 

Condition \eqref{corresp} also shows that $\Phi_0 = - \Phi_\infty$. This and the fact that $\Phi_0$ is imaginary give the reality condition
 \begin{equation}\label{reali}
 \Upsilon_0 = \overline{\Upsilon_\infty}
 \end{equation}
 
Split the integrals in \eqref{upsmfor} into four parts as in \eqref{integs}. For example, if we denote by $\zeta_e := -a/|a|$, the intersection of the unit circle with the $\ell_e$ ray, then
\begin{align}
& \int_{\ell_e} \frac{d\zeta'}{\zeta'} \frac{\zeta' + \zeta}{\zeta' -\zeta}\log\left( 1 - \mathcal{X}_e^{\text{sf}}(a,\zeta', \Upsilon_e) \right) = \notag\\
& -\int_0^{\zeta_e} \frac{d\zeta'}{\zeta'} \log\left( 1 - \mathcal{X}_e^{\text{sf}}(a,\zeta', \Upsilon_e) \right) + \int_{\zeta_e}^{\zeta_e \infty} \frac{d\zeta'}{\zeta'} \log\left( 1 - \mathcal{X}_e^{\text{sf}}(a,\zeta', \Upsilon_e) \right) \notag\\
& + \int_0^{\zeta_e} \frac{2 d\zeta'}{\zeta'-\zeta} \log\left( 1 - \mathcal{X}_e^{\text{sf}}(a,\zeta', \Upsilon_e) \right)  + \int_{\zeta_e}^{\zeta_e \infty} 2d\zeta' \left\{\frac{1}{\zeta'-\zeta} -\frac{1}{\zeta'}\right\} \log\left( 1 - \mathcal{X}_e^{\text{sf}}(a,\zeta', \Upsilon_e) \right) \label{4ints}
\end{align}

We consider the first two integrals apart from the rest. If we take the limit $a \to 0$ the exponential decay in $\mathcal{X}_e^{\text{sf}}$:
\[ \exp\left( \frac{\pi R a}{\zeta'} + \pi R \zeta' \overline{a} \right)\]
vanishes and the integrals are no longer convergent. 
 
 By combining the two integrals with their analogues in the $\ell_{-e}$ ray we obtain:
\begin{align}
 & -\int_0^{\zeta_e} \frac{d\zeta'}{\zeta'} \log\left( 1 - \mathcal{X}_e^{\text{sf}}(a,\zeta', \Upsilon_e) \right) + \int_{\zeta_e}^{\zeta_e \infty} \frac{d\zeta'}{\zeta'} \log\left( 1 - \mathcal{X}_e^{\text{sf}}(a,\zeta', \Upsilon_e) \right) \notag\\
 &  \int_0^{-\zeta_e} \frac{d\zeta'}{\zeta'} \log\left( 1 - {\mathcal{X}_e^{\text{sf}}}^{-1}(a,\zeta', -\Upsilon_e) \right) - \int_{-\zeta_e}^{-\zeta_e \infty} \frac{d\zeta'}{\zeta'} \log\left( 1 - {\mathcal{X}_e^{\text{sf}}}^{-1}(a,\zeta', -\Upsilon_e) \right)  \label{gen4int}
 \end{align}
The parametrization in the first pair of integrals is of the form $\zeta' = t\zeta_e$, and in the second pair $\zeta' = -t\zeta_e$. Making the change of variables $\zeta' \mapsto 1/\zeta'$, we can pair up these integrals in a more explicit way as:
\begin{align}
 & -\int_0^1 \frac{dt}{t} \left\{ \log\left[ 1 - \exp\left( -\pi R |a| \left(\frac{1}{t} + t \right) +i\Upsilon_e(a,-te^{i\arg a})  \right) \right] \right. \notag\\
 & \left. + \log\left[ 1 - \exp\left( -\pi R |a| \left(\frac{1}{t} + t \right) - i\Upsilon_e(a,\frac{1}{t} e^{i\arg a})  \right) \right] \right\} \notag \\
 &  + \int_0^1 \frac{dt}{t} \left\{ \log\left[ 1 - \exp\left( -\pi R |a| \left(\frac{1}{t} + t \right) +i\Upsilon_e(a,-\frac{1}{t} e^{i\arg a}) \right) \right] \right. \notag\\
 & \left. + \log\left[ 1 - \exp\left( -\pi R |a| \left(\frac{1}{t} + t \right) -i\Upsilon_e(a,te^{i\arg a}) \right) \right]  \right\} \label{4intreal}
 \end{align}
By \eqref{corresp}, the integrands come in conjugate pairs. Therefore, we can rewrite \eqref{4intreal} as:
 \begin{align*}
 -2\int_0^1 \frac{dt}{t} \text{Re } \left\{\vphantom{\int_0^1}\right. & \log\left[ 1 - \exp\left( -\pi R |a| \left(\frac{1}{t} + t \right) +i\Upsilon_e(a,-te^{i\arg a}) \right) \right] -  \\
 &  \left. \log\left[ 1 - \exp\left( -\pi R |a| \left(\frac{1}{t} + t \right) -i\Upsilon_e(a,te^{i\arg a})  \right) \right] \right\}
 \end{align*}
 \begin{equation}\label{simpl4}
 = -2\int_0^1 \frac{dt}{t} \log \left| \frac{1 - \exp\left( -\pi R |a| \left(t^{-1} + t \right) +i\Upsilon_e(a,-te^{i\arg a})  \right)}{1 - \exp\left( -\pi R |a| \left(t^{-1} + t \right) -i\Upsilon_e(a,te^{i\arg a}) \right)} \right|
 \end{equation}

Observe that \eqref{simpl4} itself suggest the correct transformation of the $\theta$ coordinates that fixes this. Indeed, for a fixed $a \neq 0$ and $\theta_e$, let $Q$ be the map
\begin{equation*}
Q(\theta_m) =  \theta_m + \psi(a,\theta),
\end{equation*}
where
 \begin{align}
 \psi_{\text{in}}(a,\theta)& = \frac{1}{2\pi}\int_0^1 \frac{dt}{t} \log\left| \frac{1 - \exp\left( -\pi R |a| \left(t^{-1} + t \right) +i\Upsilon_e(a,-te^{i\arg a})  \right)}{1 - \exp\left( -\pi R |a| \left(t^{-1} + t \right) -i\Upsilon_e(a,te^{i\arg a})  \right)} \right| \notag\\
 & = \frac{1}{2\pi}\int_0^1 \frac{dt}{t} \log\left| \frac{1 - \left[\mathcal{X}_e\right](-te^{i\arg a})}{1 - \left[\mathcal{X}_{-e}\right](te^{i\arg a})} \right| \label{newmp}
 \end{align}
 for $a \in \bein$. For $a \in \bout$ where the wall-crossing is of type I, let $\varphi = \arg (Z_{\gamma_e + \gamma_m}(a))$, with $\zeta'= -t e^{i\varphi}$ parametrizing the $\ell_{e + m}$ ray:
 \begin{align}
 \psi_{\text{out}}(a,\theta) & = \frac{1}{2\pi}\int_0^1 \frac{dt}{t} \left\{ \log\left| \frac{1 - \exp\left( -\pi R |a| \left(t^{-1} + t \right) +i\Upsilon_e(a,-te^{i\arg a})  \right)}{1 - \exp\left( -\pi R |a| \left(t^{-1} + t \right) -i\Upsilon_e(a,te^{i\arg a})  \right)} \right| \right. \notag\\
 & + \left. \log\left| \frac{1 - \exp\left( -\pi R |Z_{\gamma_e + \gamma_m}| \left(t^{-1} + t \right) +i\Upsilon_{e +m}(a,-te^{i\arg \varphi})  \right)}{1 - \exp\left( -\pi R |Z_{\gamma_e + \gamma_m}| \left(t^{-1} + t \right) -i\Upsilon_{e+m}(a,te^{i\arg \varphi})  \right)} \right|  \right\} \notag\\
 & = \frac{1}{2\pi}\int_0^1 \frac{dt}{t} \left\{ \log\left| \frac{1 - \left[\mathcal{X}_e\right](-te^{i\arg a})}{1 - \left[\mathcal{X}_{-e}\right](te^{i\arg a})} \right| + \log\left| \frac{1 - \left[\mathcal{X}_{e+m}\right](-te^{i\varphi})}{1 - \left[\mathcal{X}_{-e-m}\right](te^{i\varphi})} \right| \right\} \label{outmp}
 \end{align}
 
 Similarly, for wall-crossing of type II, $\varphi = \arg (Z_{\gamma_{-e} + \gamma_m}(a))$, with $\zeta'= -t e^{i\varphi}$ for the $\ell_{-e + m}$ ray:
  \begin{align}
 \psi_{\text{out}}(a,\theta) & = \frac{1}{2\pi}\int_0^1 \frac{dt}{t} \left\{ \log\left| \frac{1 - \exp\left( -\pi R |a| \left(t^{-1} + t \right) +i\Upsilon_e(a,-te^{i\arg a})  \right)}{1 - \exp\left( -\pi R |a| \left(t^{-1} + t \right) -i\Upsilon_e(a,te^{i\arg a})  \right)} \right| \right. \notag\\
 & + \left. \log\left| \frac{1 - \exp\left( -\pi R |Z_{\gamma_{-e} + \gamma_m}| \left(t^{-1} + t \right) +i\Upsilon_{-e +m}(a,-te^{i\arg \varphi})  \right)}{1 - \exp\left( -\pi R |Z_{\gamma_{-e} + \gamma_m}| \left(t^{-1} + t \right) -i\Upsilon_{-e+m}(a,te^{i\arg \varphi})  \right)} \right|  \right\} \notag\\
 & = \frac{1}{2\pi}\int_0^1 \frac{dt}{t} \left\{ \log\left| \frac{1 - \left[\mathcal{X}_e\right](-te^{i\arg a})}{1 - \left[\mathcal{X}_{-e}\right](te^{i\arg a})} \right| + \log\left| \frac{1 - \left[\mathcal{X}_{-e+m}\right](-te^{i\varphi})}{1 - \left[\mathcal{X}_{e-m}\right](te^{i\varphi})} \right| \right\} \label{outmp2}
 \end{align}
 
 As $a$ approaches the wall of marginal stability $W$, $\arg a \to \varphi$. We need to show the following
\begin{lemma}
The two definitions $\psi_{\text{in}}$ and $\psi_{\text{out}}$ coincide at the wall of marginal stability.
\end{lemma} 
\begin{proof}
  First let $a$ approach $W$ from the ``in'' region, so we're using definition \eqref{newmp}. Start with the pair of functions $(\mathcal{X}_e, \mathcal{X}_m)$ in the $\zeta$-plane and let $\widetilde{\mathcal{X}}_e$ denote the analytic continuation of $\mathcal{X}_e$. See Figure \ref{jumpxe}. When they reach the $\ell_e$ ray, $\mathcal{X}_e$ jumped to $\mathcal{X}_e(1-\mathcal{X}_m)$ by \eqref{kjump} and \eqref{invjmp}. Thus $\mathcal{X}_e = \widetilde{\mathcal{X}}_e(1-\mathcal{X}_m)$ along the $\ell_e$ ray.

\begin{figure}[htbp]
	\centering
		\includegraphics[width=0.45\textwidth]{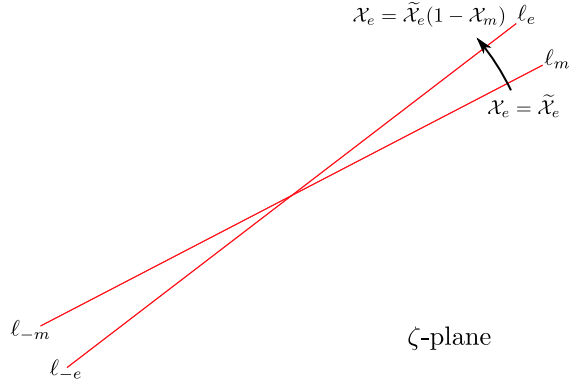}
	\caption{Jump of $\mathcal{X}_e$}
	\label{jumpxe}
\end{figure}

Therefore, 
\begin{equation*}
\psi_{\text{in}}(a,\theta) = \frac{1}{2\pi}\int_0^1 \frac{dt}{t} \log\left| \frac{1 - \left[\mathcal{X}_e(1-\mathcal{X}_m)\right](-te^{i\arg a})}{1 - \left[\mathcal{X}_{-e}(1-\mathcal{X}_m)^{-1}\right](te^{i\arg a})} \right|
\end{equation*}

Now starting from the ``out'' region, and focusing on the wall-crossing of type I for the moment, we start with the pair $(\mathcal{X}_e, \mathcal{X}_m)$ as before. This time, $\mathcal{X}_e$ at the $\ell_e$ ray has not gone to any jump yet. See Figure \ref{jumpxem}. Only $\mathcal{X}_{e+m}$ undergoes a jump at the $\ell_{e+m}$ ray and it is of the form $\mathcal{X}_{e+m} \mapsto \mathcal{X}_{e+m}(1-\mathcal{X}_{e})^{-1}$.

\begin{figure}[htbp]
	\centering
		\includegraphics[width=0.60\textwidth]{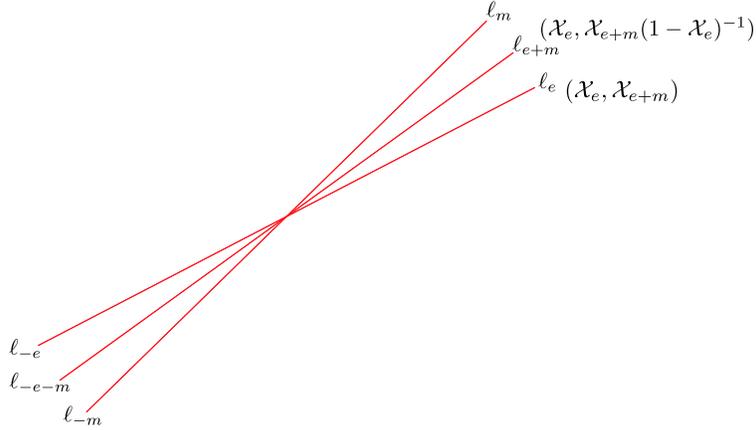}
	\caption{Only $\mathcal{X}_{e+m}$ has a jump}
	\label{jumpxem}
\end{figure}

When $a$ hits the wall $W$, $\varphi = \arg a$ and the integrals are taken over the same ray. Thus, we can combine the logs and obtain:

 \begin{align}
 \psi_{\text{out}}(a,\theta) &  = \frac{1}{2\pi}\int_0^1 \frac{dt}{t} \left\{ \log\left| \frac{1 - \left[\mathcal{X}_e\right](-te^{i\arg a})}{1 - \left[\mathcal{X}_{-e}\right](te^{i\arg a})} \right| + \log\left| \frac{1 - \left[\mathcal{X}_{e+m}(1-\mathcal{X}_{e})^{-1}\right](-te^{i\arg a})}{1 - \left[\mathcal{X}_{-e-m}(1-\mathcal{X}_{e})\right](te^{i\arg a})} \right| \right\} \notag\\
 & = \frac{1}{2\pi}\int_0^1 \frac{dt}{t} \log\left| \frac{1 - \left[\mathcal{X}_e(1-\mathcal{X}_m)\right](-te^{i\arg a})}{1 - \left[\mathcal{X}_{-e}(1-\mathcal{X}_m)^{-1}\right](te^{i\arg a})} \right|
 \end{align}
 
\noindent and the two definitions coincide. For the wall-crossing of type II the proof is entirely analogous.

 

\end{proof}
 
 \begin{theorem}\label{homeps}
 $Q$ is a reparametrization in $\theta_m$; that is, a diffeomorphism of $\mathbb{R}/2\pi \mathbb{Z}$.
 \end{theorem}
 \begin{proof}
 To show that $Q$ is injective, it suffices to show that $\left|\frac{\partial \psi}{\partial \theta_m}\right| < 1$. We will show this in the $\bein$ region. The proof for the $\bout$ region is similar.
 
 To simplify the calculations, write
  \begin{equation}\label{psii}
  \psi(a,\theta) = 2\int_0^1 \frac{dt}{t} \log\left| \frac{1-Cf(\theta_m)}{1-Cg(\theta_m)}
     \right|
   \end{equation}
  for functions $f, g$ of the form $e^{i\Upsilon_\gamma}$ for different choices of $\gamma$ (they both depend on other parameters, but they're fixed here) and a factor $C$ of the form
  \[ C = \exp\left( -\pi R |a| (t^{-1} + t)\right) \]
  Now take partials in both sides of \eqref{psii} and bring the derivative inside the integral. After an application of the chain rule we get the estimate
  \[ \left| \frac{\partial \psi}{\partial \theta_m} \right| \leq 2\int_0^1 \frac{dt}{t} |C| \left\{ \frac{|f||\frac{\partial \Upsilon_e(t)}{\partial \theta_m}|}{|1-Cf|} + \frac{|g||\frac{\partial \Theta_e(-t)}{\partial \theta_m}|}{|1-Cg|} \right\} \] 
 By the estimates in \cite[\S 3.2]{rhprob}, $\left|\frac{\partial \Upsilon_e}{\partial \theta_m}\right| < 1$. In \cite[Lemma 3.2]{rhprob}, we show that $|f|, |g|$ can be bounded by 2. The part $C$ has exponential decay so if $R$ is big enough we can bound the above by 1 and injectivity is proved. For surjectivity, just observe that $\psi(\theta_m + 2\pi) = \psi(\theta_m)$, so $Q(\theta_m + 2\pi) = \theta_m + 2\pi$. 
 \end{proof}
 
With respect to the new coordinate $\theta'_m$, the functions $\Upsilon_e, \Upsilon_m$ satisfy the equation:

\begin{align}
\Upsilon_e(a,\zeta)  = \theta_e +
\frac{1}{4\pi}\sum_{\gamma'} \Omega(\gamma';a) \left\langle \gamma_e, \gamma' \right\rangle   & \int_{\gamma'} \frac{d\zeta'}{\zeta'} \frac{\zeta' + \zeta}{\zeta' -\zeta} \log\left[ 1 - \mathcal{X}_{\gamma'}^{\text{sf}}(a,\zeta', \Upsilon_{\gamma'}) \right]  \label{inteq1} \\
\Upsilon_m(a,\zeta)  = \theta'_m +
\frac{1}{2\pi}\sum_{\gamma'} \Omega(\gamma';a) \left\langle \gamma_m, \gamma' \right\rangle  \left\{ \vphantom{\int_0^b} \right.
& \int_{0}^{b'} \frac{d\zeta'}{\zeta' - \zeta} \log\left[ 1 - \mathcal{X}_{\gamma'}^{\text{sf}}(a,\zeta', \Upsilon_{\gamma'}) \right] + \notag \\
& \left.  \int_{b'}^{b' \infty} \frac{\zeta d\zeta'}{\zeta'(\zeta' - \zeta)} \log\left[ 1 - \mathcal{X}_{\gamma'}^{\text{sf}}(a,\zeta', \Upsilon_{\gamma'}) \right] \right\} \label{inteq2} ,
\end{align}
for $b'$ the intersection of the unit circle with the $\ell_{\gamma'}$ ray. The $\Omega(\gamma';a)$ jump at the wall, but in the Pentagon case, the sum is finite.

In order to show that $\Upsilon$ converges to some function, even at $a = 0$, observe that the integral equations in \eqref{inteq1} and \eqref{inteq2} still make sense at the singular fiber, since in the case of \eqref{inteq1}, $\lim_{a \to 0} Z_m = c \neq 0$ and the exponential decay is still present, making the integrals convergent. In the case of \eqref{inteq2}, the exponential decay is gone, but the different kernel makes the integral convergent, at least for $\zeta \in \cone^\times$. The limit function $\lim_{a \to 0} \Upsilon$ should be then a solution to the integral equations obtained by recursive iteration, as in \cite[\S 3]{rhprob}.

We have to be specially careful with the Cauchy integral in \eqref{inteq2}. It will be better to obtain each iteration $\Upsilon^{(\nu)}_m$ when $|a| \to 0$ by combining the pair of rays $\ell_{\gamma'}, \ell_{-\gamma'}$ into a single line $L_{\gamma'}$, where in the case of the Pentagon, $\gamma'$ can be either $\gamma_e$ or $\gamma_{e+m}$, depending on the side of the wall we're at. We formulate a boundary problem over each infinite curve $L_{\gamma'}$ as in \S \ref{altrh}. As in the Ooguri-Vafa case, the jump function\footnote{Since we do iterations of boundary problems, we abuse notation and use simply $G(\zeta)$ where it should be $G^{(\nu)}(\zeta)$. This shouldn't cause any confusion, as our main focus in this section is how to obtain \emph{any} iteration of $\mathcal{X}_m$} $G(\zeta)$ has discontinuities of the first kind at 0 and $\infty$, but we also have a new difficulty: For $\theta_e$ close to 0, the jump function $G(\zeta) = 1-e^{i\Upsilon^{(\nu - 1)}_{\gamma'}}(\zeta)$ may be 0 for some values of $\zeta$.

Since the asymptotics of $\Upsilon^{(\nu)}_{e}$ as $\zeta \to 0$ or $\zeta \to \infty$ are $\theta_e \pm i\phi_e \neq 0$, the jump function $G(\zeta)$ can only attain the 0 value inside a compact interval away from 0 or $\infty$, hence these points are isolated in $L_{\gamma'}$. By the symmetry relation expressed in Lemma \ref{corresp}, the zeroes of $G(\zeta)$ come in pairs in $L_{\gamma'}$ and are of the form $\zeta_k, -1/\overline{\zeta_k}$. By our choice of orientation for $L_{\gamma'}$, one of the jumps is inverted so that $G(\zeta)$ has only zeroes along $L_{\gamma'}$ and no poles.

Thus, as in \S \ref{altrh}, we have a Riemann-Hilbert problem of the form\footnote{To simplify notation, we omit the iteration index $\nu$ in the Riemann-Hilbert problem expressed. By definition, $\mathcal{X}_m = \mathcal{X}^{\text{sf}}_m X_m$, for \emph{any} iteration $\nu$}
\begin{equation}\label{rhpent}
X_m^+(\zeta) = G(\zeta) X_m^-(\zeta)
\end{equation}

In \cite[Lemma 4.2]{rhprob}, we show that the solutions of \eqref{rhpent} exist and are unique, given our choice of kernel in \eqref{inteq2}. We thus obtain each iteration $\Upsilon_m^{(\nu)}$ of \eqref{inteq2}. Moreover, since by \cite{rhprob}, $\mathcal{X}_m^+ = 0$ at points $\zeta$ in the $L_e$ ray where $G(\zeta) = 0$, $\Upsilon_m^{(\nu)+}$ has a logarithmic singularity at such points.


\subsection{Estimates and a new gauge transformation}

As we've seen in the Ooguri-Vafa case, we expect our solutions $\lim_{a \to 0} \Upsilon$ to be unbounded in the $\zeta$ variable. 
Define a Banach space $\mathrsfs{X}$ as the completion under the sup norm of the space of functions $\Phi: \cone^\times \times \mathbb{T} \times U \to \cone^{2n}$  that are piecewise holomorphic on $\cone^\times$, smooth on $\mathbb{T} \times U$, for $U$ an open subset of $\mathcal{B}$ containing $0$ and such that \eqref{inteq1}, \eqref{inteq2} hold. 

Like in the Ooguri-Vafa case, let $a \to 0$ fixing $\arg a$. We will later get rid of this dependence on $\arg a$ with another gauge transformation of $\theta_m$. The following estimates on $\Upsilon^{(\nu)}$ will clearly give us that the sequence converges to some limit $\Upsilon^{(\nu)}$.

\begin{lemma}\label{estbadfib}
In the Pentagon case, at the bad fiber $a = 0$:
\begin{align}
\Upsilon_e^{(\nu + 1)} & = \Upsilon_e^{(\nu)} + O\left( e^{-2\pi \nu R |Z_m|} \right), \hspace{5 mm} \nu \geq 2 \label{elstm}\\
\Upsilon_m^{(\nu + 1)} & = \Upsilon_m^{(\nu)} + O\left( e^{-2\pi \nu R |Z_m|} \right), \hspace{5 mm} \nu \geq 1 \label{mgstm}
\end{align}
\end{lemma}
\begin{proof}
As before, we prove this by induction. Note that $\Upsilon^{(1)}_m = \Upsilon^{\text{OV}}$, the extension of the Ooguri-Vafa case obtained in \eqref{xmnice}, and $\Upsilon^{(1)}_m$ differs considerably from $\theta_m$ because of the $\log \zeta$ term. Hence the estimates cannot start at $\nu = 0$. Because of this reason, $\Upsilon^{(2)}_e$ differs considerably from $\Upsilon^{(1)}_e$ since this is the first iteration where $\Upsilon^{(1)}_m$ is considered.

Let $\nu = 1$. The integral equations for $\Upsilon_e$ didn't change in this special case. By Lemma 3.3 in \cite{rhprob}, we have for the general case:

\begin{equation}\label{frappx}
\Upsilon^{(1)}_e = \theta_e + \sum_{\gamma'} \Omega(\gamma',a) \left\langle \gamma_e, \gamma' \right\rangle \frac{e^{-2\pi R |Z_{\gamma'}|}}{4\pi i \sqrt{R |Z_{\gamma'}|}}\frac{\zeta_{\gamma'} + \zeta}{\zeta_{\gamma'} - \zeta} e^{i\theta_{\gamma'}} + O\left( \frac{e^{-2\pi R |Z_{\gamma'}|}}{R}\right)
\end{equation}
where $\zeta_{\gamma'} = -\frac{Z_{\gamma'}}{|Z_{\gamma'}|}$ is the saddle point for the integrals in \eqref{inteq1}, and $\zeta$ is not $\zeta_{\gamma'}$. Note that there is no divergence if $\zeta \to 0$ or $\zeta \to \infty$. If $\zeta = \zeta_{\gamma'}$, again by Lemma 3.3 in \cite{rhprob}, we obtain estimates as in \eqref{frappx} except for the $\sqrt{R}$ terms in the denominator.

In any case, for the Pentagon, the $\gamma'$ in \eqref{frappx} are only $ \gamma_{\pm m},  \gamma_{\pm (e+m)}$, depending on the side of the wall of marginal stability. At $a = 0$, $Z_{e+m} = Z_m$, so \eqref{frappx}  gives that $\log[1 - e^{i \Upsilon^{(1)}_e}] = \log[1 - e^{i\theta_e}] + O(e^{-2\pi R |Z_m|})$ along the $\ell_e$ ray, and a similar estimate holds for $\log[1 - e^{-i \Upsilon^{(1)}_e}]$ along the $\ell_{-e}$ ray. Plugging in this in \eqref{inteq2}, we get \eqref{mgstm} for $\nu = 1$.

For general $\nu$, a saddle point analysis on $\Upsilon^{(\nu)}_e$ can still be performed and obtain as in \eqref{frappx}:
\begin{equation}\label{appgnu}
\Upsilon^{(\nu+1)}_e = \theta_e + \frac{e^{-2\pi R |Z_m|}}{4\pi i \sqrt{R |Z_m|}} \left\{ \frac{\zeta_m + \zeta}{\zeta_m - \zeta} e^{i\Upsilon^{(\nu)}_m(\zeta_m)} -  \frac{\zeta_m - \zeta}{\zeta_m + \zeta} e^{-i\Upsilon^{(\nu)}_m(-\zeta_m)} \right\} + O\left( \frac{e^{-2\pi R |Z_{\gamma'}|}}{R}\right),
\end{equation}
from one side of the wall. On the other side (for type I) it will contain the extra terms
\begin{equation}
\frac{e^{-2\pi R |Z_m|}}{4\pi i \sqrt{R |Z_m|}} \left\{ \frac{\zeta_m + \zeta}{\zeta_m - \zeta} e^{i(\Upsilon^{(\nu)}_m(\zeta_m) + \Upsilon^{(\nu)}_e(\zeta_m))} -  \frac{\zeta_m - \zeta}{\zeta_m + \zeta} e^{-i(\Upsilon^{(\nu)}_m(-\zeta_m) - \Upsilon^{(\nu)}_e(-\zeta_m))} \right\}.
\end{equation}
\noindent Observe that for this approximation we only need $\Upsilon^{(\nu)}$ at the point $\zeta_m$. By the previous part, for $\nu = 2$,
\[  e^{i\Upsilon^{(2)}_m(\zeta_m)} = e^{i\Upsilon^{(1)}_m(\zeta_m)} \left(1 + O\left( e^{-2\pi R |Z_m|}\right) \right) \]
Thus, for $\nu = 2$,
\begin{align}
\Upsilon^{(3)}_e & = \theta_e + \frac{e^{-2\pi R |Z_m|}}{4\pi i \sqrt{R |Z_m|}} \left\{ \frac{\zeta_m + \zeta}{\zeta_m - \zeta} e^{i\Upsilon^{(1)}_m(\zeta_m)} \left(1 + O\left( e^{-2\pi R |Z_m|}\right) \right) \right. \notag\\
& -  \left. \frac{\zeta_m - \zeta}{\zeta_m + \zeta} e^{-i\Upsilon^{(1)}_m(-\zeta_m)} \left(1 + O\left( e^{-2\pi R |Z_m|}\right) \right) \right\} + O\left( R^{1/2}\right)\notag\\
& = \Upsilon^{(2)}_e + O\left( e^{-4\pi R|Z_m|}\right) \label{appgnu2}
\end{align}
and similarly in the other side of the wall. For general $\nu$, the same arguments show that \eqref{elstm}, \eqref{mgstm} hold after the appropriate $\nu$.
\end{proof}


There is still one problem: the limit of $\widetilde{\mathcal{X}}_m$ we obtained as $a \to 0$ for the analytic continuation of $\mathcal{X}_m$ was only along a fixed ray $\arg a = $ constant. To get rid of this dependence, it is necessary to perform another gauge transformation on the torus coordinates $\theta$. Recall that we are restricted to the Pentagon case. Let $a \to 0$ fixing $\arg a$. Let $\zeta_{\gamma}$ denote $Z_{\gamma}/|Z_{\gamma}|$. In particular, $\zeta_e = a/|a|$ and this remains constant since we're fixing $\arg a$. Also, $\zeta_m = Z_m/|Z_m|$ and this is independent of $\arg a$ since $Z_m$ has a limit as $a \to 0$. The following lemma will allow us to obtain the correct gauge transformation.
 
\begin{lemma}
For the limit $\left. \widetilde{\mathcal{X}}_m\right|_{a=0}$ obtained above, its imaginary part is independent of the chosen ray $\arg a = c$ along which $a \to 0$.
\end{lemma}
\begin{proof}
Let $\widetilde{\Upsilon}_m$ denote the analytic continuation of $\Upsilon_m$ yielding $\widetilde{\mathcal{X}}_m$. Start with a fixed value $\arg a \equiv \rho_0$, for $\rho_0$ different from $\arg Z_m(0), \arg (-Z_m(0))$. For another ray $\arg a \equiv \rho$, we compute $\left. \Upsilon_m\right|_{\substack{a=0\\ \arg a = \rho}} - \left. \Upsilon_m\right|_{\substack{a=0\\ \arg a = \rho_0}}$ (without analytic continuation for the moment).

The integrals in \eqref{inteq2} are of two types. One type is of the form
\begin{equation}\label{1sttype}
\int_0^{\zeta_{\pm e}} \frac{d\zeta'}{\zeta' - \zeta} \log\left[ 1 - e^{i\Upsilon_{\pm e}(\zeta')} \right] +  \int_{\zeta_{\pm e}}^{\zeta_{\pm e}\infty} \frac{\zeta d\zeta'}{\zeta'(\zeta' - \zeta)} \log\left[ 1 - e^{i\Upsilon_{\pm e}(\zeta')} \right]
\end{equation}
The other type appears only in the outside part of the wall of marginal stability. Since $Z : \Gamma \to \cone$ is a homomorphism, $Z_{\gamma_e + \gamma_m} = Z_{\gamma_e} + Z_{\gamma_m}$. At $a = 0$, $Z_e = a = 0$, so $Z_{e+m} = Z_m$. Hence, $\ell_m = \ell_{e+m}$ at the singular fiber. This second type of integral is thus of the form
\begin{equation}\label{2ndtype}
\int_0^{\zeta_{\pm m}} \frac{d\zeta'}{\zeta' - \zeta} \log\left[ 1 - e^{i\Upsilon_{\pm (e+m)}(\zeta')} \right] +  \int_{\zeta_{\pm m}}^{\zeta_{\pm m}\infty} \frac{\zeta d\zeta'}{\zeta'(\zeta' - \zeta)} \log\left[ 1 - e^{i\Upsilon_{\pm (e+m)}(\zeta')} \right]
\end{equation}
Since the $\ell_m$ stays fixed at $a = 0$ independently of $\arg a$, \eqref{2ndtype} does not depend on $\arg a$, so this has a well-defined limit as $a \to 0$. We should focus then only on integrals of the type \eqref{1sttype}. For a different $\arg a$, $\zeta_e$ changes to another point $\widetilde{\zeta}_e$ in the unit circle. See Figure \ref{zetilde}. The paths of integration change accordingly. We have two possible outcomes: either $\zeta$ lies outside the sector determined by the two paths, or $\zeta$ lies inside the region.

\begin{figure}[htbp]
	\centering
		\includegraphics[width=0.45\textwidth]{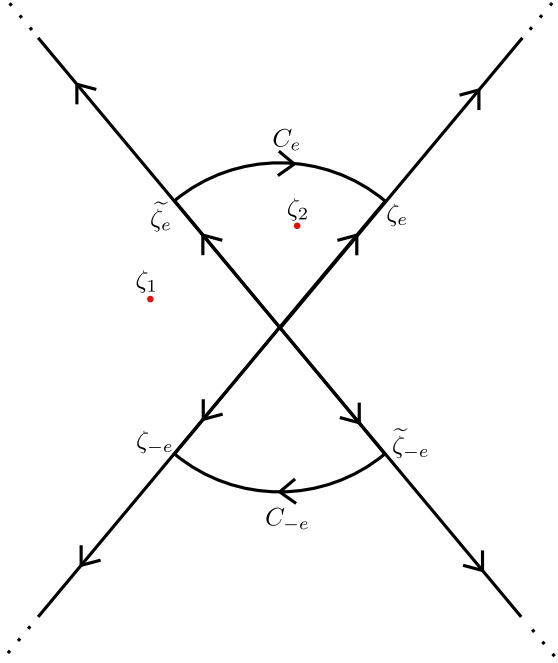}
	\caption{As $\arg a$ changes, the paths of integration change}
	\label{zetilde}
\end{figure}

In the first case ($\zeta_1$ on Figure \ref{zetilde}), the integrands
\begin{equation}\label{2kernl}
 \frac{\log[1-e^{i\Upsilon_{\pm e}(\zeta')}]}{\zeta'-\zeta}, \hspace{5 mm} \frac{\zeta\log[1-e^{i\Upsilon_{\pm e}(\zeta')}]}{\zeta'(\zeta'-\zeta)}
\end{equation}
are holomorphic on $\zeta'$ in the sector between the two paths. By Cauchy's formula, the difference between the two integrals is just the integration along a path $C_{\pm e}$ between the two endpoints $\zeta_{\pm e}, \widetilde{\zeta}_{\pm e}$. If $f(s)$ parametrizes the path $C_e$, let $C_{-e} = -1/\overline{f(s)}$. The orientation of $C_{e}$ in the contour containing $\infty$ is opposite to the contour containing 0. Similarly for $C_{-e}$. Thus, the difference of $\Upsilon_m$ for these two values of $\arg a$ is the integral along $C_{e}, C_{-e}$ of the difference of kernels \eqref{2kernl}, namely:
\begin{equation}\label{realdiff}
\int_{C_e} \frac{d\zeta'}{\zeta'} \log[1-e^{i\Upsilon_e(\zeta')}] - \int_{C_{-e}} \frac{d\zeta'}{\zeta'} \log[1-e^{-i\Upsilon_e(\zeta')}]
\end{equation} 

Even if $e^{i\Upsilon_e(\zeta')} = 1$ for $\zeta'$ in the contour, the integrals in \eqref{realdiff} are convergent, so this is well-defined for any values of $\theta_e \neq 0$. By symmetry of $C_e, C_{-e}$ and the reality condition \eqref{corresp}, the second integral is the conjugate of the first one. Thus \eqref{realdiff} is only real.

When $\zeta$ hits one of the contours, $\zeta$ coincides with one of the $\ell_e$ or $\ell_{-e}$ rays, for some value of $\arg a$. The contour integrals jump since $\zeta$ lies now inside the contour ($\zeta_2$ in Figure \ref{zetilde}). The jump is by the residue of the integrands \eqref{2kernl}. This gives the jump of $\mathcal{X}_m$ that the analytic continuation around $a = 0$ cancels. Therefore, only the real part of $\Upsilon_m$ depends on $\arg a$.
\end{proof}

By the previous lemma, $\left. \widetilde{\Upsilon}_m\right|_{\substack{a=0\\ \arg a = \rho}} - \left. \widetilde{\Upsilon}_m\right|_{\substack{a=0\\ \arg a = \rho_0}}$ is real and is given by \eqref{realdiff}. Define then a new gauge transformation:
\begin{equation}\label{fingauge}
\widetilde{\theta}_m = \theta'_m - \frac{1}{2\pi} \left\{\int_{C_e} \frac{d\zeta'}{\zeta'} \log[1-e^{i\Upsilon_e(\zeta')}] + \int_{C_{-e}} \frac{d\zeta'}{\zeta'} \log[1-e^{-i\Upsilon_e(\zeta')}] \right\}
\end{equation}

This eliminates the dependence on $\arg a$ for the limit $\left. \widetilde{\mathcal{X}}_m \right|_{a=0}$. As we did in \S \ref{clasov} in Theorem \ref{mprtom}, we can extend the torus fibration $\mathcal{M}'$ by gluing a $S^1$-fiber bundle of the form $D \times (0, 2\pi) \times S^1$ for $D$ a disk around $a = 0$, $\theta_e \in (0,2\pi)$ and $\widetilde{\theta}_m$ the new coordinate of the $S^1$ fibers. Using Taub-NUT space as a local model for this patch, the trivial $S^1$ bundle can be extended to $\theta_e = 0$ where the fiber degenerates into a point (nevertheless, in Taub-NUT coordinates the space is still locally isomorphic to $\cone^2$). Since $\widetilde{\mathcal{X}}_m \equiv 0$ if $\theta_e = 0$ as in \S \ref{clasov}, in this new manifold $\mathcal{M}$ we  thus obtain a well defined function $\widetilde{\mathcal{X}}_m$.

\subsection{Extension of the derivatives}\label{exderv}
\index{Extension of the derivatives@\emph{Extension of the derivatives}}%

So far we were able to extend the functions $\mathcal{X}_e, \widetilde{\mathcal{X}}_m$ to $\mathcal{M}$. Unfortunately, we can no longer bound uniformly on $\nu$ the derivatives of $\widetilde{\mathcal{X}}_m$ near $a = 0$, so the Arzela-Ascoli arguments no longer work here. Since there's no difference on the definition of $\mathcal{X}_e$ at $a = 0$ from that of the regular fibers, this function extends smoothly to $a = 0$.

We have to obtain the extension of all derivatives of $\widetilde{\mathcal{X}}_m$ directly from its definition. It suffices to extend the derivatives of $\mathcal{X}_m$ only, as the analytic continuation doesn't affect the symplectic form $\varpi(\zeta)$ (see below).

\begin{lemma}
$\log \mathcal{X}_m$ extends smoothly to $\mathcal{M}$, for $\theta_e \neq 0$.
\end{lemma}
\begin{proof}
For convenience, we rewrite $\Upsilon_m$ with the final magnetic coordinate $\widetilde{\theta_m}$:
\begin{align*}
\Upsilon_m & = \widetilde{\theta_m} + \frac{1}{2\pi} \left\{\int_{C_e} \frac{d\zeta'}{\zeta'} \log[1-e^{i \Upsilon_e(\zeta')}] -
\int_{C_{-e}} \frac{d\zeta'}{\zeta'} \log[1-e^{-i \Upsilon_e(\zeta')}] \right\}\\
& + \frac{1}{2\pi}\sum_{\gamma'} \Omega(\gamma';a) \left\langle \gamma_m, \gamma' \right\rangle  \left\{ \vphantom{\int_0^b}  \int_{0}^{\zeta_{\gamma'}} \frac{d\zeta'}{\zeta' - \zeta} \log\left[ 1 - \mathcal{X}_{\gamma'}^{\text{sf}}(a,\zeta', \Upsilon_{\gamma'}) \right] \right. +  \\
& \left.  \int_{\zeta_{\gamma'}}^{\zeta_{\gamma'} \infty} \frac{\zeta d\zeta'}{\zeta'(\zeta' - \zeta)} \log\left[ 1 - \mathcal{X}_{\gamma'}^{\text{sf}}(a,\zeta', \Upsilon_{\gamma'}) \right] \right\}
\end{align*}
where $e^{i \Upsilon_e(\zeta')}$ is evaluated only at $a = 0$. For $\gamma'$ of the type $\pm \gamma_e \pm \gamma_m$, $\mathcal{X}_{\gamma'}$ and its derivatives still have exponential decay along the $\ell_{\gamma'}$ ray, so these parts in $\Upsilon_m$ extend to $a =0$ smoothly. It thus suffices to extend only

\begin{align}
\Upsilon_m & = \widetilde{\theta_m} + \frac{1}{2\pi} \left\{\int_{C_e} \frac{d\zeta'}{\zeta'} \log[1-e^{i \Upsilon_e(\zeta')}] -
\int_{C_{-e}} \frac{d\zeta'}{\zeta'} \log[1-e^{-i \Upsilon_e(\zeta')}] \right. \notag\\
& +  \int_{0}^{\zeta_{e}} \frac{d\zeta'}{\zeta' - \zeta} \log\left[ 1 - \mathcal{X}_{e}^{\text{sf}}(a,\zeta', \Upsilon_{e}) \right] + \int_{\zeta_{e}}^{\zeta_{e} \infty} \frac{\zeta d\zeta'}{\zeta'(\zeta' - \zeta)} \log\left[ 1 - \mathcal{X}_{e}^{\text{sf}}(a,\zeta', \Upsilon_{e}) \right] \notag\\
& -\left. \int_{0}^{-\zeta_{e}} \frac{d\zeta'}{\zeta' - \zeta} \log\left[ 1 - {\mathcal{X}_{e}^{\text{sf}}}^{-1}(a,\zeta', -\Upsilon_{e}) \right] -  \int_{-\zeta_{e}}^{-\zeta_{e} \infty} \frac{\zeta d\zeta'}{\zeta'(\zeta' - \zeta)} \log\left[ 1 - {\mathcal{X}_{e}^{\text{sf}}}^{-1}(a,\zeta', -\Upsilon_{e}) \right]\right\}\label{upsmcompl}
\end{align}
together with the semiflat part $\pi R \frac{Z_m}{\zeta} + \pi R \zeta \overline{Z_m}$, which we assume is as in the Generalized Ooguri-Vafa case, namely:
\begin{equation}\label{xmcomplet}
 \mathcal{X}_m =  \exp\left( \frac{-i R }{2\zeta}(a\log a - a + f(a)) + i \Upsilon_m + \frac{i\zeta R}{2} (\overline{a} \log
\overline{a} - \overline{a} + \overline{f(a)} )\right)
\end{equation}
for a holomorphic function $f$ near $a = 0$ and such that $f(0) \neq 0$. The derivatives of the terms involving $f(a)$ clearly extend to $a = 0$, so we focus on the rest, as in \S \ref{c1ext}.

We show first that $\dfrac{\partial \log \mathcal{X}_m}{\partial_{\theta_e}}, \dfrac{\partial \log \mathcal{X}_m}{\partial_{\theta_m}}$ extend to $a = 0$. Since there is no difference in the proof between electric or magnetic coordinates, we'll denote by $\partial_\theta$ a derivative with respect to any of these two variables.

We have:
\begin{align*}
\frac{\partial}{\partial \theta}\log \mathcal{X}_m & = \frac{-i}{2\pi} \left\{ \int_{C_e} \frac{d\zeta'}{\zeta'} \frac{e^{i \Upsilon_e(\zeta')}}{1-e^{i \Upsilon_e(\zeta')}} \frac{\partial \Upsilon_e(\zeta')}{\partial \theta} - \int_{C_{-e}} \frac{d\zeta'}{\zeta'} \frac{e^{-i \Upsilon_e(\zeta')}}{1-e^{-i \Upsilon_e(\zeta')}} \frac{\partial \Upsilon_e(\zeta')}{\partial \theta} \right.\\
& +  \int_{0}^{\zeta_{e}} \frac{d\zeta'}{\zeta' - \zeta} \frac{\mathcal{X}_e(\zeta')}{1-\mathcal{X}_e(\zeta')}\frac{\partial \Upsilon_e(\zeta')}{\partial \theta} + \int_{\zeta_{e}}^{\zeta_{e} \infty} \frac{\zeta d\zeta'}{\zeta'(\zeta' - \zeta)} \frac{\mathcal{X}_e(\zeta')}{1-\mathcal{X}_e(\zeta')}\frac{\partial \Upsilon_e(\zeta')}{\partial \theta}\\
& \left. +  \int_{0}^{-\zeta_{e}} \frac{d\zeta'}{\zeta' - \zeta} \frac{\mathcal{X}^{-1}_e(\zeta')}{1-\mathcal{X}^{-1}_e(\zeta')}\frac{\partial \Upsilon_e(\zeta')}{\partial \theta} + \int_{-\zeta_{e}}^{\zeta_{e} \infty} \frac{\zeta d\zeta'}{\zeta'(\zeta' - \zeta)} \frac{\mathcal{X}^{-1}_e(\zeta')}{1-\mathcal{X}^{-1}_e(\zeta')}\frac{\partial \Upsilon_e(\zeta')}{\partial \theta} \right\}
\end{align*}

when $a \to 0$, $ \dfrac{\mathcal{X}_e(\zeta')}{1-\mathcal{X}_e(\zeta')} \to \dfrac{e^{i \Upsilon_e(\zeta')}}{1-e^{i \Upsilon_e(\zeta')}}$. The integrals along $C_e$ and $C_{-e}$ represent a difference of integrals along the contour in the last integrals and a fixed contour, as in Figure \ref{zetilde}. Thus, when $a = 0$,

\begin{align}
\left. 2\pi i \frac{\partial}{\partial \theta}\log \Upsilon_m\right|_{a =0} & =  \int_{0}^{b} \frac{d\zeta'}{\zeta' - \zeta} \frac{\mathcal{X}_e(\zeta')}{1-\mathcal{X}_e(\zeta')}\frac{\partial \Upsilon_e(\zeta')}{\partial \theta} + \int_{b}^{b \infty} \frac{\zeta d\zeta'}{\zeta'(\zeta' - \zeta)} \frac{\mathcal{X}_e(\zeta')}{1-\mathcal{X}_e(\zeta')}\frac{\partial \Upsilon_e(\zeta')}{\partial \theta} \notag\\
& \left. +  \int_{0}^{-b} \frac{d\zeta'}{\zeta' - \zeta} \frac{\mathcal{X}^{-1}_e(\zeta')}{1-\mathcal{X}^{-1}_e(\zeta')}\frac{\partial \Upsilon_e(\zeta')}{\partial \theta} + \int_{-b}^{-b \infty} \frac{\zeta d\zeta'}{\zeta'(\zeta' - \zeta)} \frac{\mathcal{X}^{-1}_e(\zeta')}{1-\mathcal{X}^{-1}_e(\zeta')}\frac{\partial \Upsilon_e(\zeta')}{\partial \theta} \right\} \label{polec}
\end{align}
for a fixed point $b$ in the unit circle, independent of $a$. If $\Upsilon_e(\zeta') = 1$ for a point $c$ in the line $L$ passing through the origin and $b$, then as seen in \cite[Lemma 4.2]{rhprob}, the function $\mathcal{X}_m$ develops a zero on the right side of such line. Nevertheless, the analytic continuation $\widetilde{\mathcal{X}}_m$ around $a = 0$ introduces a factor of the form $(1 - \mathcal{X}_e)^{-1}$ when $a$ changes from region III to region I in Figure \ref{3reg}, so the pole at $c$ on the right side of $L$ for the derivative $\dfrac{\partial}{\partial \theta}\log \Upsilon_m$ coming from the integrand in \eqref{polec} is canceled by analytic continuation. Hence, the integrals are well defined and thus the left side has an extension to $a = 0$.

Now, for the partials with respect to $a, \overline{a}$, there are two different types of dependence: one is the dependence of the contours, the other is the dependence of the integrands. The former dependence is only present in \eqref{upsmcompl}, as the contours in Figure \ref{zetilde} change with $\arg a$. A simple application of the Fundamental Theorem of Calculus in each integral in \eqref{upsmcompl} gives that this change is:

\begin{align*}
\left. -2\pi i  \frac{\partial}{\partial \arg a}\log \Upsilon_m\right|_{a =0}  & = \log[1-e^{-i \Upsilon_e(\zeta_e)}]  - \log[1-e^{-i \Upsilon_e(\zeta_e)}]\\
& - \log[1-e^{-i \Upsilon_e(\zeta_e)}] + \log[1-e^{-i \Upsilon_e(\zeta_e)}] = 0,
\end{align*}
where we again used the fact that the integrals along $C_e$ and $C_{-e}$ represent the difference between the integrals in the other pairs with respect to two different rays, one fixed. By continuity on parameters, the terms are still 0 if $\Upsilon_e(\zeta_e) = 0$. Compare this with \eqref{dap1}, where we obtained this explicitly.

Then there is the dependence on $a, \overline{a}$ on the integrands and the semiflat part. Focusing on $a$ only, we take partials on $\log \mathcal{X}_m$ in \eqref{xmcomplet} (ignoring constants and parts that clearly extend to $a = 0$). This is:

\begin{equation}\label{logadiv}
\frac{\log a}{\zeta} + \int_0^{\zeta_e} \frac{d\zeta'}{\zeta'(\zeta' - \zeta)} \frac{\mathcal{X}_e}{1-\mathcal{X}_e} + \int_0^{-\zeta_e} \frac{d\zeta'}{\zeta'(\zeta' - \zeta)} \frac{\mathcal{X}^{-1}_e}{1-\mathcal{X}^{-1}_e}
\end{equation}

This is the equivalent of \eqref{last3} in the general case. In the limit $a \to 0$, we can do an asymptotic expansion of $\dfrac{e^{i \Upsilon_e(\zeta')}}{1-e^{i \Upsilon_e(\zeta')}} = \dfrac{e^{i \Upsilon_e(0)}}{1-e^{i \Upsilon_e(0)}} + O(\zeta')$. Clearly when we write this expansion in \eqref{logadiv}, the only divergent term at $a = 0$ is the first degree approximation in the integral. Thus, we can focus on that and assume that the $\dfrac{\mathcal{X}_e}{1-\mathcal{X}_e}$ (resp. $\dfrac{\mathcal{X}^{-1}_e}{1-\mathcal{X}^{-1}_e}$) factor is constant. If we do the partial fraction decomposition, we can run the same argument as in Eqs. \eqref{simpler} up to \eqref{cancel} and obtain that \eqref{logadiv} is actually 0 at $a = 0$. The only identity needed is
\[ \frac{1}{1-e^{i\Upsilon_e(0)}} +  \frac{1}{1-e^{-i\Upsilon_e(0)}}  =  1 \]

The argument also works for the derivative with respect to $\overline{a}$, now with an asymptotic expansion around $\infty$ of $\Upsilon_e$.

This shows that $\widetilde{\mathcal{X}_m}$ extends in a $C^1$ way to $a = 0$. For the $C^\infty$ extension, derivatives with respect to any $\theta$ coordinate work in the same way, all that was used was the specific form of the contours $C_e, C_{-e}$. The same thing applies to the dependence on the contours $C_e, C_{-e}$. For derivatives with respect to $a, \overline{a}$ in the integrands, we can again do an asymptotic expansion of $\Upsilon_e$ at 0 or $\infty$ and compare it to the asymptotic of the corresponding derivative of $a \log a - a$ as $a \to 0$.

\end{proof}

Nothing we have done in this section is particular of the Pentagon example. We only needed the specific values of $\Omega(\gamma;u)$ given in \eqref{omgpar} to obtain the Pentagon identities at the wall and to perform the analytic continuation of $\mathcal{X}_m$ around $u = 2$.  For any integrable systems data as in section \ref{intsys}  with suitable invariants $\Omega(\gamma;u)$ allowing the wall-crossing formulas and analytic continuation, we can do the same isomonodromic deformation of putting all the jumps at a single admissible ray, perform saddle-point analysis and obtain the same extensions of the Darboux coordinates $\mathcal{X}_\gamma$. This finishes the proof of Theorem \ref{extbf}.

What is exclusive of the Pentagon case is that we have a well-defined hyperk\"{a}hler metric $g_\text{OV}$ that we can use as a local model of the metric to be constructed here.

The extension of the holomorphic symplectic form $\varpi(\zeta)$ is now straightforward. We proceed as in \cite{gaiotto} by first writing:
\begin{equation*}
\varpi(\zeta) = -\frac{1}{4\pi^2 R} \frac{d\mathcal{X}_e}{\mathcal{X}_e} \wedge \frac{d\mathcal{X}_m}{\mathcal{X}_m}
\end{equation*} 

Where we used the fact that the jumps of the functions $\mathcal{X}_\gamma$ are via the symplectomorphisms $\mathcal{K}_{\gamma'}$ of the complex torus $T_a$ (see \eqref{kjump}) so $\varpi(\zeta)$ remains the same if we take $\mathcal{X}_m$ or its analytic continuation $\widetilde{\mathcal{X}_m}$.

We need to show that $\varpi(\zeta)$ is of the form
\begin{equation}
-\frac{i}{2\zeta}\omega_+ + \omega_3 -\frac{i \zeta}{2} \varpi_-
\end{equation}
that is, $\varpi(\zeta)$ must have simple poles at $\zeta = 0$ and $\zeta = \infty$, even at the singular fiber where $a = 0$.

By definition, $\mathcal{X}_e = \exp(\frac{\pi R a}{\zeta} + i\Upsilon_e + \pi R \zeta \overline{a})$. Thus

\begin{equation*}
\frac{d\mathcal{X}_e(\zeta)}{\mathcal{X}_e(\zeta)} = \frac{\pi R da}{\zeta} + i d\Upsilon_e(\zeta) +\pi R \zeta d\overline{a}
\end{equation*}

By \eqref{inteq1}, and since $\lim_{a \to 0} Z_m \neq 0$, $\mathcal{X}_m$ (resp. $\mathcal{X}_{-m}$) of the form $\exp(\frac{\pi R Z_m(a)}{\zeta} + i\Upsilon_m + \pi R \zeta \overline{Z_m(a)})$ still has exponential decay when $\zeta$ lies in the $\ell_m$ ray (resp. $\ell_{-m}$), even if $a = 0$. The differential $d \Upsilon_e(\zeta)$ thus exists for any $\zeta \in \cpone$ since the integrals defining it converge for any $\zeta$.

As in \cite{gaiotto}, we can write
\[ \frac{d\mathcal{X}_e}{\mathcal{X}_e} \wedge \frac{d\mathcal{X}_m}{\mathcal{X}_m} = \frac{d\mathcal{X}_e}{\mathcal{X}_e} \wedge \left( \frac{d\mathcal{X}^{\text{sf}}_m}{\mathcal{X}^{\text{sf}}_m} + \mathcal{I_{\pm}} \right), \]
for $\mathcal{I}_\pm$ denoting the corrections to the semiflat function. By the form of $\mathcal{X}^{\text{sf}} =  \exp(\frac{\pi R Z_m(a)}{\zeta} + i\theta_m + \pi R \zeta \overline{Z_m(a)})$, the wedge involving only the semiflat part has only simple poles at $\zeta = 0$ and $\zeta = \infty$, so we can focus on the corrections. These are of the form
\begin{align*}
\frac{d\mathcal{X}_e(\zeta)}{\mathcal{X}_e(\zeta)} \wedge \mathcal{I_{\pm}} & = \frac{-i}{2\pi} \left\{ \int_0^{\zeta_e} \frac{d\zeta'}{\zeta'-\zeta} \frac{\mathcal{X}_e(\zeta')}{1-\mathcal{X}_e(\zeta')} \frac{d\mathcal{X}_e(\zeta)}{\mathcal{X}_e(\zeta)} \wedge \frac{d\mathcal{X}_e(\zeta')}{\mathcal{X}_e(\zeta')}\right.\\
& + \int_{\zeta_e}^{\zeta_e \infty} \frac{\zeta d\zeta'}{\zeta'(\zeta'-\zeta)}\frac{\mathcal{X}_e(\zeta')}{1-\mathcal{X}_e(\zeta')}\frac{d\mathcal{X}_e(\zeta)}{\mathcal{X}_e(\zeta)} \wedge \frac{d\mathcal{X}_e(\zeta')}{\mathcal{X}_e(\zeta')}\\
& + \int_0^{-\zeta_e} \frac{d\zeta'}{\zeta'-\zeta} \frac{\mathcal{X}^{-1}_e(\zeta')}{1-\mathcal{X}^{-1}_e(\zeta')} \frac{d\mathcal{X}_e(\zeta)}{\mathcal{X}_e(\zeta)} \wedge \frac{d\mathcal{X}_e(\zeta')}{\mathcal{X}_e(\zeta')}\\
& + \left. \int_{-\zeta_e}^{-\zeta_e \infty} \frac{\zeta d\zeta'}{\zeta'(\zeta'-\zeta)}\frac{\mathcal{X}^{-1}_e(\zeta')}{1-\mathcal{X}^{-1}_e(\zeta')}\frac{d\mathcal{X}_e(\zeta)}{\mathcal{X}_e(\zeta)} \wedge \frac{d\mathcal{X}_e(\zeta')}{\mathcal{X}_e(\zeta')} \right\}
\end{align*}

In the ``inside'' part of the wall of marginal stability. A similar equation holds in the other side. We can simplify the wedge products above by taking instead
\begin{equation}
\frac{d\mathcal{X}_e(\zeta)}{\mathcal{X}_e(\zeta)} \wedge \left(\frac{d\mathcal{X}_e(\zeta)}{\mathcal{X}_e(\zeta)} - \frac{d\mathcal{X}_e(\zeta')}{\mathcal{X}_e(\zeta')} \right) = \pi R \left[ \left( \frac{1}{\zeta} - \frac{1}{\zeta'}\right)da + (\zeta - \zeta')d\overline{a}\right] +i \left( d\Phi_e(\zeta) - d\Phi_e(\zeta') \right)
\end{equation}

Recall that $\Phi_e$ represents the corrections to $\theta_e$, so $\Upsilon_e = \theta_e + \Phi_e$. By \S \ref{solut}, $\Phi_e$ and $d\Phi_e$ are defined for $\zeta = 0$ $\zeta = \infty$ even if $a = 0$, since $\lim_{a \to 0} Z_m(a) \neq 0$ and the exponential decay in $\mathcal{X}_m^\text{sf}$ still present guarantees convergence of the integrals in \ref{inteq1}. Hence, the terms involving $d\Phi_e(\zeta) - d\Phi_e(\zeta')$ are holomorphic for any $\zeta \in \cpone$. It thus suffices to consider the other terms. After simplifying the integration kernels, we obtain

\begin{align*}
\frac{\pi R da}{\zeta} \int_0^{\zeta_e} \frac{d\zeta'}{\zeta'} \frac{\mathcal{X}_e(\zeta')}{1-\mathcal{X}_e(\zeta')} & +\pi R da \int_{\zeta_e}^{\zeta_e \infty} \frac{d\zeta'}{(\zeta')^2} \frac{\mathcal{X}_e(\zeta')}{1-\mathcal{X}_e(\zeta')}\\
\frac{\pi R da}{\zeta} \int_0^{-\zeta_e} \frac{d\zeta'}{\zeta'} \frac{\mathcal{X}^{-1}_e(\zeta')}{1-\mathcal{X}^{-1}_e(\zeta')} & +\pi R da \int_{-\zeta_e}^{-\zeta_e \infty} \frac{d\zeta'}{(\zeta')^2} \frac{\mathcal{X}^{-1}_e(\zeta')}{1-\mathcal{X}^{-1}_e(\zeta')}\\
-\pi R d\overline{a} \int_0^{\zeta_e} d\zeta' \frac{\mathcal{X}_e(\zeta')}{1-\mathcal{X}_e(\zeta')} & -\pi R \zeta d\overline{a} \int_{\zeta_e}^{\zeta_e \infty} \frac{d\zeta'}{\zeta'} \frac{\mathcal{X}_e(\zeta')}{1-\mathcal{X}_e(\zeta')}\\
-\pi R d\overline{a} \int_0^{\zeta_e} d\zeta' \frac{\mathcal{X}^{-1}_e(\zeta')}{1-\mathcal{X}^{-1}_e(\zeta')} & -\pi R \zeta d\overline{a} \int_{\zeta_e}^{\zeta_e \infty} \frac{d\zeta'}{\zeta'} \frac{\mathcal{X}^{-1}_e(\zeta')}{1-\mathcal{X}^{-1}_e(\zeta')}\\
\end{align*}

The only dependence on $\zeta$ is in the factors $\zeta, 1/\zeta$. Thus $\varpi(\zeta)$ has only simple poles at $\zeta = 0$ and $\zeta = \infty$.

Finally, the estimates in Lemma \ref{estbadfib} show that if we recover the hyperk\"{a}hler metric $g$ from the holomorphic symplectic form $\varpi(\zeta)$ as in \S \ref{extmetric} and \S \ref{genextmtr}, we obtain that the hyperk\"{a}hler metric for the Pentagon case is the metric obtained in \ref{extmetric} for the Ooguri-Vafa case plus smooth corrections near $a = 0, \theta_e = 0$, so it extends to this locus.

 This gives Theorem \ref{smfrm}.

\bibliographystyle{amsplain}
\bibliography{diss}        

\end{document}